\documentclass[nontitlepage]{article}

 \usepackage[matrix,arrow,curve]{xy}
 \usepackage{graphicx}
 \usepackage[english]{babel}
 \usepackage{amsmath}
 \usepackage{amssymb, indentfirst}
 \usepackage{amsthm}
 \usepackage{color}
\newcommand{\Add}{{\rm Add}}

\newtheorem{cor}{Corollary}
\newtheorem{lem}{Lemma}
\newtheorem{rem}{Remark}
\newtheorem{ex}{Example}
\newtheorem{opr}{Definition}
\newtheorem{thm}{Theorem}
\newtheorem{prop}{Proposition}
\newtheorem{ques}{Question}
\newtheorem{setup}{Setup}
\newtheorem{rems}{Remarks}

\begin{document}

\title{Silting Theory in triangulated categories with coproducts}

\author{Pedro Nicol\'as \thanks{The authors thank Chrysostomos Psaroudakis, Jorge Vit\'oria, Francesco Mattiello and Luisa Fiorot for their careful reading of two earlier versions of the paper and for their subsequent comments and suggestions which helped us a lot. We also thank Jan \v{S}\v{t}ov\'\i\v{c}ek for telling us about Lemma \ref{lem.Stovicek}. Finally, the authors deeply thank the referee for the careful reading of the paper and for her/his comments and suggestions.   Nicol\'as and Saor\'in are supported   by research projects from the  Spanish Ministerio de Econom\'ia y Competitividad (MTM2016-77445-P)  and from the Fundaci\'on `S\'eneca' of Murcia (19880/GERM/15), with a part of FEDER funds. Zvonareva is supported by the RFFI Grant 16-31-60089. 
The authors
thank these institutions for their help. Zvonareva also thanks the University of Murcia for its hospitality during her visit, on which this research started.}
\\
Departamento de Did\'actica de las Ciencias Matem\'aticas y Sociales\\ Universidad de Murcia, Aptdo. 4021  \\ 30100 Espinardo, Murcia\\ SPAIN\\
{\it pedronz@um.es} \\  \\ Manuel Saor\'in 
\\ Departamento de Matem\'aticas\\
Universidad de Murcia, Aptdo. 4021\\
30100 Espinardo, Murcia\\
SPAIN\\ {\it msaorinc@um.es} \\ \\ Alexandra Zvonareva \\ Chebyshev Laboratory\\ St. Petersburg State University\\ 14th Line 29B,  St. Petersburg 199178\\ RUSSIA\\ {\it alexandrazvonareva@gmail.com}}


\date{}

\maketitle


\begin{abstract}

{\bf We introduce the notion of noncompact (partial) silting and (partial) tilting sets and objects in any triangulated category $\mathcal{D}$ with arbitrary (set-indexed) coproducts. We show that equivalence classes of partial silting sets are in bijection with t-structures generated by their co-heart whose heart has a generator, and in case $\mathcal{D}$ is compactly generated, this  restricts to: i) a bijection between equivalence classes of self-small partial silting objects and left nondegenerate t-structures in $\mathcal{D}$ whose heart is a module category and whose associated cohomological functor preserves products; ii) a bijection between equivalence classes of classical silting objects and  nondegenerate smashing and co-smashing t-structures whose heart is a module category.  

We describe the objects in the aisle of the t-structure associated to a partial silting set $\mathcal{T}$ as  Milnor (or homotopy) colimits  of sequences of morphisms with successive cones in $\text{Sum}(\mathcal{T})[n]$. We use this fact to develop a  theory of tilting objects in very general AB3 abelian categories, a setting and its dual in which we show the validity of several well-known results of tilting and cotilting theory of modules. Finally, we show that if $\mathcal{T}$ is a bounded tilting set in a compactly generated algebraic triangulated category $\mathcal{D}$ and $\mathcal{H}$ is the heart of the associated t-structure, then the inclusion $\mathcal{H}\hookrightarrow\mathcal{D}$ extends to a triangulated equivalence $\mathcal{D}(\mathcal{H})\stackrel{\sim}{\longrightarrow}\mathcal{D}$ which restricts to bounded levels.  }
 
\end{abstract}

{\bf Mathematics Subjects Classification:} 18E30, 18E10, 18GXX.

\vspace*{1cm}

\section{Introduction} \label{sec.introduction}

Silting sets and objects in triangulated categories were introduced by Keller and Vossieck \cite{KV} and later studied by Aihara and Iyama \cite{AI}, as a way of overcoming a problem inherent to tilting objects, namely,  that mutations are sometimes impossible to define. By extending the class of tilting objects to the wider class of silting objects Aihara and Iyama were able to define a concept of silting mutation that always worked. In the initial definition of a silting object, a strong generation property was required, in the sense that the ambient triangulated category had to be the thick subcategory generated by the object. As a consequence, the study of silting objects was mainly concentrated on `categories of compact objects', especially on the perfect derived category of an algebra. 

From contributions of several authors (see \cite{AI}, \cite{MSSS}, \cite{IJY}, \cite{KN}...) it soon became clear that silting complexes were connected with several concepts existing in the literature. For instance, with co-t-structures (equivalently,  weight structures), as defined in \cite{Bo} and  \cite{P}, with t-structures as defined in \cite{BBD} and, in the context of Representation Theory, also with the so-called simple minded collections (see \cite{KN}). As the final point of this route, K\"onig and Yang \cite{KY} gave, for a finite dimensional algebra $\Lambda$,  a precise bijection between equivalence classes of silting complexes in $\text{per}(\Lambda )\cong\mathcal{K}^b(\Lambda )$, bounded co-t-structures in $\mathcal{K}^b(\Lambda )$, bounded t-structures in $\mathcal{D}^b(\text{mod}-\Lambda)$ whose heart is a length category and equivalence classes of simple minded collections in $\mathcal{D}^b(\text{mod}-\Lambda )$. In addition, they showed that these bijections were compatible with the concepts of mutation defined for each set. Similar results, also presented in several meetings,  were independently obtained in \cite{KN2} for homologically smooth and homologically nonpositive dg algebras with finite dimensional homology.

In a route similar to the one followed by tilting modules and, more generally, tilting complexes, a few authors (see \cite{W} and \cite{AMV}) extended the notion of silting object to the unbounded derived category $\mathcal{D}(R)$ of a ring $R$. 
The strong generation condition had necessarily to be dropped, but the newly defined concept of `big' silting complex allowed them to extend the  K\"onig-Yang bijections, except for the simple-minded collections,  to the unbounded setting (see \cite{AMV}).  A further step in this direction is done independently in \cite{PV} and in this paper. Here we shall introduce a notion of partial silting object in any triangulated category with coproducts, which will still allow a sort of K\"onig-Yang bijection. In fact any such partial silting object defines a t-structure in the triangulated category whose heart has a projective generator. This leads naturally to the question, whether this is the way of obtaining all t-structures whose heart has a projective generator. Even more specifically, whether this is the way of obtaining all t-structures whose heart is the module category over a small $K$-category or over an ordinary algebra. Our results in the paper  give partial answers to these questions and, using the dual concept of partial cosilting object, we can also address the question of characterising t-structures whose heart is a Grothendieck category. 

If one follows the development of tilting theory for modules and complexes of modules, one sees that the so-called classical (i.e. compact) tilting complexes give equivalences of categories. Indeed, as shown in the work by  Rickard and Keller (see \cite{Ri1}, \cite{Ri2} and \cite{K}), if $T$ is a classical tilting complex in $\mathcal{D}(A)$ and $B=\text{End}_{\mathcal{D}(A)}(T)$, then there is an equivalence of triangulated categories $\mathcal{D}(A)\stackrel{\sim}{\longrightarrow}\mathcal{D}(B)$ which takes $T$ to $B$.  When replacing such a classical (=compact) tilting complex by a noncompact one, we do not have an equivalence of categories but, after replacing $T$ by some power $T^{(I)}$, one actually has a recollement of triangulated categories (see \cite{BMT} and \cite[Section 7]{NS2}). One of the common features is that a tilting complex $T$, be it compact or noncompact,  defines the t-structure $\tau_T=(T^{\perp_{>0}},T^{\perp_{<0}})$ in $\mathcal{D}(A)$. In the classical (=compact) case the associated heart $\mathcal{H}_T$ turns out to be equivalent to $\text{Mod}-\text{End}_{\mathcal{D}(A)}(B)$ and the inclusion functor $\mathcal{H}_T\hookrightarrow\mathcal{D}(A)$ extends to an equivalence of categories $\mathcal{D}(\mathcal{H}_T)\stackrel{\sim}{\longrightarrow}\mathcal{D}(A)$ (see \cite{FMT} for the case of a classical tilting object in an abelian category). Surprisingly, with the tilting theory for AB3 abelian categories developed in Section \ref{sect.tilting theory abelian}, one has that this phenomenon is  still true when $T$ is an infinitely generated $n$-tilting module  and, even more generally, for any ($n$-)tilting object in such an AB3 abelian category (see \cite{CXW} and \cite{FMS}). So it is very natural to ask if a similar phenomenon holds for (nonclassical) tilting  objects in any triangulated category with coproducts. Note that the phenomenon is discarded if, more generally, one deals just with silting objects (see \cite[Corollary 5.2]{PV}). 

In this paper we define the co-heart of a t-structure in a triangulated category (see the first paragraph of Section \ref{sect.coheart}) and, when such a category has coproducts, we introduce the notions of (partial) tilting and (partial) silting sets of objects, calling them classical when they consist of compact objects (see Definitions \ref{def.strong presilting} and \ref{def.classical silting objects}). Our  first main result and two of its consequences are the following, all stated for any triangulated category
 $\mathcal{D}$  with coproducts: 

- {\it (Part of Theorem \ref{teor.presilting t-structures})  A t-structure $\tau$ in $\mathcal{D}$ is generated by a partial silting set if and only if it is generated by its co-heart and its heart has a generator. When, in addition, $\mathcal{D}$ is compactly generated, this is equivalent to saying that $\tau$ is left nondegenerate, its heart has a projective generator and the cohomological functor $\tilde{H}:\mathcal{D}\longrightarrow\mathcal{H}$ preserves products.}

- {\it(Part of Corollary \ref{cor.bijection for comp.generated}) When $\mathcal{D}$ is compactly generated, there is a bijection between equivalence classes of self-small  partial silting sets (resp. objects)  and  left nondegenerate  t-structures in $\mathcal{D}$ whose heart is the module category over a small $K$-category (resp. an ordinary $K$-algebra) and whose associated cohomological functor $\tilde{H}$ preserves products. }

- {\it (Part of Proposition \ref{prop.bijection smashing-cosmashing})  When $\mathcal{D}$ is compactly generated, there are: i) a bijection between equivalence classes of classical silting sets (resp. objects) and nondegenerate smashing and co-smashing t-structures whose heart is a module category over a small $K$-category (resp. ordinary algebra); ii) a bijection between equivalence classes of cosilting pure-injective objects $Q$ such that ${}^{\perp_{<0}}Q$ is closed under taking products and nondegenerate smashing and co-smashing t-structures whose heart is a Grothendieck category.} 

These bijections are unbounded versions of the bijection in \cite[Theorem 6.1]{KY} (classes (1) and (3)) (see also  \cite[Theorem 1.1]{Liu} for a related result in a bounded setting). They follow the trend of extending  K\"onig-Yang's bijections  to  the unbounded context (see, e.g., \cite[Theorem 4.6]{AMV} or the parallel result \cite[Corollary 4.7]{PV}).

The concept of partial silting set in our general setting has the problem that it is sometimes difficult to check its defining conditions for a given set of objects. On the other hand, even in the case of a silting object $T$, where the aisle is $T^{\perp_{>0}}$, it is not clear how the objects of this aisle can be defined in terms of $T$. Our second main result partially solves these problems:

- {\it (Part of Theorem \ref{teor.weakly equivalent to presilting}) When $\mathcal{T}$ is a strongly nonpositive set in $\mathcal{D}$ (see Definition \ref{def.strongly nonpositive}), it is partial silting if and only if there is a t-structure $(\mathcal{V},\mathcal{V}^\perp [1])$ such that $\mathcal{T}\subset\mathcal{V}$ and, for some  $q\in\mathbb{Z}$, the functor $\text{Hom}_{D}(T,?)$ vanishes on $\mathcal{V}[q]$ for  all $T\in\mathcal{T}$. Moreover, each object in the aisle  of the associated t-structure is a Milnor (or homotopy) colimit of a sequence $$X_0\stackrel{x_1}{\longrightarrow}X_1\stackrel{x_2}{\longrightarrow}\cdots\stackrel{x_n}{\longrightarrow}X_n\stackrel{x_{n+1}}{\longrightarrow}\cdots, $$  where $X_0\in\text{Sum}(\mathcal{T})$ and $\text{cone}(x_n)\in\text{Sum}(\mathcal{T})[n]$, for all $n\geq 0$.}

This description of the aisle has an important (not straightforward) consequence, when $\mathcal{D}=\mathcal{D}(\mathcal{A})$ is the derived category of an abelian category $\mathcal{A}$ and $T$ is an object of $\mathcal{A}$ which is  partial silting in $\mathcal{D}(\mathcal{A})$ . In this case,  the objects in the aisle are precisely those chain complexes which are isomorphic in $\mathcal{D}(\mathcal{A})$ to complexes  $\cdots\longrightarrow T^{-n}\longrightarrow \cdots\longrightarrow T^{-1}\longrightarrow T^0\longrightarrow 0\longrightarrow\cdots$, with all the $T^{-k}$ in $\text{Sum}(T)$ (see Proposition \ref{prop.aisle of partial tilting object}). This led us to think that it might be possible to extend the well-established theory of   tilting modules (see, e.g.,  \cite{BB}, \cite{HR} and \cite{Mi} for the classical part, and \cite{AC} and \cite{CT} for the infinitely generated part) to any abelian category $\mathcal{A}$ whose derived category has $\text{Hom}$ sets and arbitrary   coproducts. This is indeed the case and Section \ref{sect.tilting theory abelian} is devoted to developing such a theory. Definition \ref{def.(partial) tilting object} introduces the concept of tilting object in such an abelian category and the main result of the section, {\it Theorem \ref{teor.characterization of tilting objects}}, shows that several known characterizations of tilting modules also work in this general setting. The advantage of the new theory is that it is apt to dualization. In this way tilting and cotilting theory are two sides of a unique theory. This has already been exploited in \cite{FMS}. 

The final of the main results   and its consequences provide a partial answer to the question of whether the inclusion from the heart can be extended to a triangulated equivalence. 

- {\it (Theorem \ref{teor.tilting derived equivalence} and Corollaries \ref{cor.restriction of the equivalence} and \ref{cor.restriction dg algebra}) Let $\mathcal{D}$ be any compactly generated algebraic triangulated category and let $\mathcal{T}$ be a bounded tilting set in $\mathcal{D}$ (see Definition \ref{def.big tilting object}). If $\mathcal{H}=\mathcal{H}_\mathcal{T}$ is the heart of the associated t-structure in $\mathcal{D}$, then the inclusion $\mathcal{H}\hookrightarrow\mathcal{D}$ extends to a triangulated equivalence $\Psi :\mathcal{D}(\mathcal{H})\stackrel{\sim}{\longrightarrow}\mathcal{D}$ which restricts to equivalences $\mathcal{D}^*(\mathcal{H})\stackrel{\sim}{\longrightarrow}\mathcal{D}^*$, for $*\in\{+,-,b\}$, with an appropriate definition of $\mathcal{D}^*$ which is the classical one when $\mathcal{D}=\mathcal{D}(A)$, for a dg algebra $A$.} 

\vspace*{0.3cm}

The organization of the paper goes as follows. Section \ref{sect.preliminares} is of preliminaries, and there we introduce most of the needed terminology. In Section \ref{sect.coheart} we introduce the co-heart of a t-structure and study its properties. In Section \ref{sect.silting bijection} partial silting and partial tilting sets and objects in a triangulated category with coproducts $\mathcal{D}$ are introduced and the mentioned Theorem \ref{teor.presilting t-structures} is proved together with some corollaries which give a K\"onig-Yang-like bijection between equivalence classes of partial silting objects and some t-structures in $\mathcal{D}$. In Section \ref{sect.aisle silting t-structure} we prove Theorem \ref{teor.weakly equivalent to presilting}. Then Section \ref{sect.tilting theory abelian} is devoted to developing a tilting theory for abelian categories whose derived category has $\text{Hom}$ sets and coproducts and to the proof of the mentioned Theorem \ref{teor.characterization of tilting objects}. Section \ref{sect.triangulated equivalences} gives the  statement and    proof of Theorem \ref{teor.tilting derived equivalence} and its corollaries. The final section \ref{sec.exceptional sequences} shows that exceptional sequences, as studied in Algebraic Geometry and Representation Theory,  and natural generalizations of them give rise to examples of partial silting sets.

\section{Preliminaries and terminology} \label{sect.preliminares}

All throughout this paper, we shall work over a commutative ring $K$, fixed from now on. All  categories will be $K$-categories. That is, the morphisms between two objects  form a $K$-module and the composition of morphisms will be $K$-bilinear. All of them are assumed to have $\text{Hom}$ sets and, in case of doubts as for the derived category of an abelian category, this $\text{Hom}$ set hypothesis will be  required.  Unless explicitly said otherwise  categories will be also additive and all subcategories will  be full and closed under taking isomorphisms.   Coproducts and products will be always  small (i.e. set-indexed). The expression `$\mathcal{A}$ has coproducts (resp. products)' will then mean that $\mathcal{A}$ has arbitrary set-indexed coproducts (resp. products).   
When $\mathcal{S}\subset\text{Ob}(\mathcal{A})$ is any class of objects, we shall
denote by $\text{sum}_\mathcal{A}(\mathcal{S})$ (resp. $\text{Sum}_\mathcal{A}(\mathcal{S})$) the subcategory of objects which are finite (resp. arbitrary) coproducts
of objects in $\mathcal{S}$. Then 
$\text{add}_\mathcal{A}(\mathcal{S})$ (resp.
$\text{Add}_\mathcal{A}(\mathcal{S})$) will denote the subcategory of objects which are direct summands of objects in $\text{sum}_\mathcal{A}(\mathcal{S})$ (resp. $\text{Sum}_\mathcal{A}(\mathcal{A})$).
Also, we will denote by $\text{Prod}_\mathcal{A}(\mathcal{S})$ the class of objects which are direct
summands of arbitrary products of objects in $\mathcal{S}$. When
$\mathcal{S}=\{V\}$, for some object $V$, we will simply write $\text{sum}_\mathcal{A}(V)$ (resp. $\text{Sum}_\mathcal{A}(V)$) 
$\text{add}_\mathcal{A}(V)$ (resp. $\text{Add}_\mathcal{A}(V)$) and
$\text{Prod}_\mathcal{A}(V)$.  

If  $\mathcal{S}$ is a class (resp. set) as above, we will say that it is a 
 \emph{class (resp. set) of
generators} when, given any nonzero morphism $f:X\longrightarrow Y$ in $\mathcal{A}$, there is a morphism $g:S\longrightarrow X$, for some $S\in\mathcal{S}$, such that $f\circ g\neq 0$. Note that when $\mathcal{S}$ is a set, this is equivalent to saying that the functor $\coprod_{S\in\mathcal{S}}\text{Hom}_\mathcal{A}(S,?):\mathcal{A}\longrightarrow\text{Ab}$ is faithful. An object $G$ is a
\emph{generator} of $\mathcal{A}$, when $\{G\}$ is a set of
generators. When $\mathcal{A}$ is abelian and  $\mathcal{P}$ is a class (resp. set)  of projective objects, then $\mathcal{P}$  is a class (resp. set) of generators if and only if, given any  $0\neq X\in\text{Ob}(\mathcal{A})$, there is a nonzero morphism $P\longrightarrow X$, for some $P\in\mathcal{P}$.   When $\mathcal{A}$ is  AB3 abelian  (see definition below) and $\mathcal{S}$ is a set, it is a set of generators exactly when each object $X$ of $\mathcal{A}$ is an epimorphic image of a coproduct of objects of $\mathcal{S}$ (see \cite[Proposition IV.6.2]{St}).  
 The concepts of a \emph{class (resp. a set) of cogenerators} and of a \emph{cogenerator} are defined dually. 
Sometimes, in the case of an abelian category $\mathcal{A}$ we will employ a stronger version of these concepts. Namely, a class $\mathcal{S}\subseteq\text{Ob}(\mathcal{A})$  will be called a \emph{generating (resp. cogenerating)
class} of $\mathcal{A}$ when, for  each object $X$ of $\mathcal{A}$,
there is  an epimorphism $S\twoheadrightarrow X$ (resp. monomorphism
$X\rightarrowtail S$), for some $S\in\mathcal{S}$.  

We shall say that \emph{idempotents split in $\mathcal{A}$}, when given any object $X$ of $\mathcal{A}$ and any idempotent endomorphism $e=e^2\in\text{End}_\mathcal{A}(X)$, there is an isomorphism $f:X\stackrel{\sim}{\longrightarrow}X_1\coprod X_2$, for some $X_1,X_2\in\text{Ob}(\mathcal{A})$, such that $e$ is the composition 
$$X\stackrel{f}{\longrightarrow}X_1\coprod X_2\stackrel{\begin{pmatrix}1 & 0 \end{pmatrix}}{\longrightarrow}X_1\stackrel{\begin{pmatrix}1\\ 0 \end{pmatrix}}{\longrightarrow}X_1\coprod X_2\stackrel{f^{-1}}{\longrightarrow}X.$$ 
When $\mathcal{A}$ is abelian, idempotents split in it.   

When $\mathcal{A}$ has  coproducts, we shall say that an object $X$ is a
\emph{compact (or small) object} when the functor
$\text{Hom}_\mathcal{A}(X,?):\mathcal{A}\longrightarrow\text{Ab}$
preserves coproducts. That is, when the canonical map $\coprod_{i\in I}\text{Hom}_\mathcal{A}(T,X_i)\longrightarrow\text{Hom}_\mathcal{A}(T,\coprod_{i\in I}X_i)$ is bijective, for each family $(X_i)_{i\in I}$ of objects of $\mathcal{A}$. More generally, we will say that a set $\mathcal{S}$ of objects is \emph{self-small}, when the canonical map $\coprod_{i\in I}\text{Hom}_\mathcal{A}(S,S_i)\longrightarrow\text{Hom}_\mathcal{A}(S,\coprod_{i\in I}S_i)$ is bijective, for each $S\in\mathcal{S}$ and each family $(S_i)_{i\in I}$ of objects of $\mathcal{S}$. 
An object $X$ will be called \emph{self-small} when $\{X\}$ is a self-small set. 

If $\mathcal{X}$ is a subcategory  and $M$ is an object of $\mathcal{A}$, a morphism $f:X_M\longrightarrow M$ is an \emph{$\mathcal{X}$-precover} (or \emph{right $\mathcal{X}$-approximation}) of $M$ when $X_M$ is in $\mathcal{X}$ and, for  every morphism $g:X\longrightarrow M$ with $X\in\mathcal{X}$, there is a morphism $v:X\longrightarrow X_M$ such that $f\circ v=g$. When $\mathcal{A}$ has coproducts and $\mathcal{S}$ is a set of objects, every object $M$ admits a morphism which is both an $\text{Add}(\mathcal{S})$-precover and a $\text{Sum}(\mathcal{S})$-precover, namely the canonical morphism $\epsilon_M:\coprod_{S\in\mathcal{S}}S^{(\text{Hom}_\mathcal{A}(S,M))}\longrightarrow M$. If, for each $S\in\mathcal{S}$ and each $f\in\text{Hom}_{\mathcal{A}}(S,M)$, we denote by $\iota_{(S,f)}:S\longrightarrow\coprod_{S\in\mathcal{S}}S^{(\text{Hom}_\mathcal{A}(S,M))}$ the corresponding injection into the coproduct, then $\epsilon_M$ is the unique morphism such that $\epsilon_M\circ\iota_{(S,f)}=f$, for all $S\in\mathcal{S}$ and $f\in\text{Hom}_\mathcal{A}(S,M)$. 

We will frequently use the following `hierarchy' among abelian categories
introduced by Grothendieck (\cite{Gr}). Let $\mathcal{A}$ be an
abelian category.

\begin{enumerate}
\item[] - $\mathcal{A}$ is \emph{AB3 (resp. AB3*)} when it has
coproducts (resp. products);
\item[] - $\mathcal{A}$ is \emph{AB4 (resp. AB4*)} when it is AB3 (resp. AB3*)
and the coproduct functor $\coprod
:[I,\mathcal{A}]\longrightarrow\mathcal{A}$ (resp. product functor
$\prod :[I,\mathcal{A}]\longrightarrow\mathcal{A}$) is  exact, for
each set $I$;
\item[] - $\mathcal{A}$ is \emph{AB5 (resp. AB5*)} when it is AB3 (resp. AB3*) and the direct limit
functor $\varinjlim :[I,\mathcal{A}]\longrightarrow\mathcal{A}$
(resp. inverse limit functor $\varprojlim
:[I^{op},\mathcal{A}]\longrightarrow\mathcal{A}$) is exact, for each
 directed set $I$.
\end{enumerate}
Note that the AB3 (resp. AB3*) condition is equivalent to the fact
that $\mathcal{A}$ is cocomplete (resp. complete).  An AB5 abelian category $\mathcal{G}$ having a set of generators
(equivalently, a generator),  is called a \emph{Grothendieck
category}. A classical example of a Grothendieck category is the  category of right modules $\text{Mod}-\mathcal{C}$ over a (skeletally) small (not necessarily additive) $K$-category $\mathcal{C}$. Its objects are the $K$-linear functors $\mathcal{C}^{op}\longrightarrow\text{Mod}-K$ and its morphisms are the natural transformations. As a particular case,  an ordinary (=associative unital) $K$-algebra $A$ may be viewed as a $K$-category with just one object, where the morphisms are the elements of $A$ and where the composition is the anti-product. In that case $\text{Mod}-A$ coincides with the usual description of modules over an algebra.   The following result of Gabriel-Mitchell will be frequently used (see \cite[Corollary 6.4]{Po})

\begin{prop} \label{prop.Gabriel-Mitchell}
Let $\mathcal{A}$ be an abelian $K$-category. The following assertions are equivalent:

\begin{enumerate}
\item $\mathcal{A}$ is AB3 and has a set of compact projective generators;
\item $\mathcal{A}$ is equivalent to $\text{Mod}-\mathcal{C}$, for some (skeletally) small $K$-category $\mathcal{C}$.
\end{enumerate}
In particular, $\mathcal{A}$ is equivalent to $\text{Mod}-A$, for some ordinary $K$-algebra $A$, if and only if $\mathcal{A}$ is AB3 and has a \emph{progenerator} (= compact projective generator). 
\end{prop}
A category as in the last proposition will be called a \text{module category} over a small $K$-category  (or over an ordinary algebra, if it is as in the last sentence).

 When $\mathcal{A}$ is an AB3 abelian category, $\mathcal{S}\subset\text{Ob}(\mathcal{A})$ is any class of objects and $n$ 
  is a natural number, we will denote by $\text{Pres}^n(\mathcal{S})$ the subcategory of objects $X\in\text{Ob}(\mathcal{A})$ which admit an exact sequence $\Sigma^{-n}\longrightarrow \cdots\longrightarrow \Sigma^{-1}\longrightarrow\Sigma^{0}\longrightarrow X\longrightarrow 0$, with the $\Sigma^{-k}$  in $\text{Sum} (\mathcal{S})$ for all $k=0,1,\dots,n$.

We refer the reader to \cite{N} for the precise definition of
\emph{triangulated category}, but, diverting from the terminology in
that book, for a given triangulated category $\mathcal{D}$, we will
denote by $?[1]:\mathcal{D}\longrightarrow\mathcal{D}$ its
suspension functor. We will then put $?[0]=1_\mathcal{D}$ and $?[k]$
will denote the $k$-th power of $?[1]$, for each integer $k$.
(Distinguished) triangles in $\mathcal{D}$ will be denoted
$X\longrightarrow Y\longrightarrow Z\stackrel{+}{\longrightarrow}$,
or also $X\longrightarrow Y\longrightarrow
Z\stackrel{w}{\longrightarrow}X[1]$ when the connecting morphism $w$
needs to be emphasized. A \emph{triangulated functor} between
triangulated categories is one which preserves triangles. 

Given any additive category $\mathcal{A}$ and any class $\mathcal{S}$ of objects in it, we shall denote by $\mathcal{S}^\perp$ (resp. ${}^\perp\mathcal{S}$) the subcategory of objects $X\in\text{Ob}(\mathcal{A})$ such  that $\text{Hom}_\mathcal{A}(S,X)=0$ (resp. $\text{Hom}_\mathcal{A}(X,S)=0$), for all $S\in\mathcal{S}$. In the particular case when $\mathcal{A}=\mathcal{D}$ is a triangulated category and $n\in\mathbb{Z}$ is an integer, we will denote by $\mathcal{S}^{\perp_{\geq n}}$ (resp. $\mathcal{S}^{\perp_{\leq n}}$) the subcategory of $\mathcal{D}$ consisting of the objects $Y$ such that $\text{Hom}_\mathcal{D}(S,Y[k])=0$, for all $S\in\mathcal{S}$ and all integers $k\geq n$ (resp. $k\leq n$). Symmetrically,  the subcategory ${}^{\perp_{\geq n}}\mathcal{S}$ (resp. ${}^{\perp_{\leq n}}\mathcal{S}$) will be the one whose objects $X$ satisfy that $\text{Hom}_\mathcal{D}(X,S[k])=0$, for all $S\in\mathcal{S}$ and all $k\geq n$ (resp. $k\leq n$). By analogous recipe, one defines $\mathcal{S}^{\perp_{>n}}$, $\mathcal{S}^{\perp_{<n}}$, ${}^{\perp_{>n}}\mathcal{S}$ and ${}^{\perp_{<n}}\mathcal{S}$. We will use also the symbol $\mathcal{S}^{\perp_{k\in\mathbb{Z}}}$ (resp. ${}^{\perp_{k\in\mathbb{Z}}}\mathcal{S}$) to denote the subcategory of those objects $X$ such that $\text{Hom}_\mathcal{D}(S,X[k])=0$ (resp.  $\text{Hom}_\mathcal{D}(X,S[k])=0$), for all $k\in\mathbb{Z}$. 

Unlike the
terminology used in the  general setting of additive categories, in
the specific context of triangulated categories a weaker version of the term
'class (resp. set) of generators' is commonly used.  Namely, a class (resp.  set)
$\mathcal{S}\subset\text{Ob}(\mathcal{D})$ is called a \emph{class (resp. set) of
generators of $\mathcal{D}$} when $\mathcal{S}^{\perp_{k\in\mathbb{Z}}}=0$. Dually $\mathcal{C}$ is a \emph{class (resp. set) of cogenerators} of $\mathcal{D}$ when ${}^{\perp_{k\in\mathbb{Z}}}\mathcal{C}=0$.  In case $\mathcal{D}$  has
coproducts, we will say that $\mathcal{D}$ is \emph{compactly
generated} when it has a set of compact generators. A triangulated category is called \emph{algebraic} when it is equivalent to the stable category of a Frobenius exact category  (see \cite{H}, \cite{K}). 

Recall that if $\mathcal{D}$ and $\mathcal{A}$ are a triangulated
and an abelian category, respectively, then an additive  functor
$H:\mathcal{D}\longrightarrow\mathcal{A}$ is a \emph{cohomological
functor} when, given any triangle $X\longrightarrow Y\longrightarrow
Z\stackrel{+}{\longrightarrow}$, one gets an induced long exact
sequence in $\mathcal{A}$:

\begin{center}
$\cdots \longrightarrow H^{n-1}(Z)\longrightarrow H^n(X)\longrightarrow
H^n(Y)\longrightarrow H^n(Z)\longrightarrow
H^{n+1}(X)\longrightarrow \cdots$,
\end{center}
where $H^n:=H\circ (?[n])$, for each $n\in\mathbb{Z}$. Each  representable functor $\text{Hom}_\mathcal{D}(X,?):\mathcal{D}\longrightarrow\text{Mod}-K$ is cohomological. We will say that $\mathcal{D}$ \emph{satisfies Brown representability theorem} when $\mathcal{D}$ has coproducts and each cohomological functor $H:\mathcal{D}^{op}\longrightarrow\text{Mod}-K$ that preserves products (i.e. that,  as a contravariant functor $\mathcal{D}\longrightarrow\text{Mod}-K$, it takes coproducts to products) is representable. We will say that $\mathcal{D}$ \emph{satisfies Brown representability theorem for the dual} when $\mathcal{D}^{op}$ satisfies Brown representability theorem. 
Each compactly generated triangulated category satisfies both Brown representability theorem and its dual (\cite[Theorem B and subsequent remark]{Kr}).

Given a triangulated category $\mathcal{D}$, a subcategory $\mathcal{E}$ will be called a \emph{suspended subcategory} when it is closed under taking extensions and $\mathcal{E}[1]\subseteq\mathcal{E}$. If, in addition, we have $\mathcal{E}=\mathcal{E}[1]$, we will say that $\mathcal{E}$ is a \emph{triangulated subcategory}. A triangulated subcategory closed under taking direct summands is called a \emph{thick subcategory}. When the ambient triangulated category $\mathcal{D}$ has coproducts, a triangulated subcategory closed under taking arbitrary coproducts is called a \emph{localizing subcategory}. Note that such a subcategory is always thick  (see the proof of \cite[Proposition 1.6.8]{N}, which also shows that idempotents split in any triangulated category with coproducts). Clearly, there are dual concepts of \emph{cosuspended subcategory} and \emph{colocalizing subcategories}, while those of triangulated and thick subcategory are self-dual. 
Given any class $\mathcal{S}$ of objects of $\mathcal{D}$, we will denote by $\text{susp}_\mathcal{D}(\mathcal{S})$ (resp. $\text{tria}_\mathcal{D}(\mathcal{S})$, resp. $\text{thick}_\mathcal{D}(\mathcal{S})$) the smallest suspended (resp. triangulated, resp. thick) subcategory of $\mathcal{D}$ containing $\mathcal{S}$. When $\mathcal{D}$ has coproducts, we will let $\text{Susp}_\mathcal{D}(\mathcal{S})$ and $\text{Loc}_\mathcal{D}(\mathcal{S})$ be the smallest suspended subcategory closed under taking coproducts and the smallest localizing subcategory containing $\mathcal{S}$, respectively. 
 
Given an additive  category $\mathcal{A}$, we will denote by
$\mathcal{C}(\mathcal{A})$ and  $\mathcal{K}(\mathcal{A})$  the category of chain complexes of
objects of $\mathcal{A}$ and the homotopy category of $\mathcal{A}$. Diverting from the classical notation, we will write superindices for chains, cycles and boundaries in ascending order.  We will denote by $\mathcal{C}^-(\mathcal{A})$ (resp. $\mathcal{K}^-(\mathcal{A})$), $\mathcal{C}^+(A)$ resp. $\mathcal{K}^+(\mathcal{A})$) and $\mathcal{C}^b(\mathcal{A})$ (resp. $\mathcal{K}^b(\mathcal{A})$) the full subcategories of $\mathcal{C}(\mathcal{A})$ (resp. $\mathcal{K}(\mathcal{A})$) consisting of those objects isomorphic to upper bounded, lower bounded and (upper and lower) bounded complexes, respectively. Note that $\mathcal{K}(\mathcal{A})$ is always a triangulated category of which $\mathcal{K}^-(\mathcal{A})$, $\mathcal{K}^+(\mathcal{A})$ and $\mathcal{K}^b(\mathcal{A})$ are triangulated subcategories. Furthermore, when $\mathcal{A}$ has coproducts,  $\mathcal{C}(\mathcal{A})$ and $\mathcal{K}(\mathcal{A})$ also have coproducts, which are calculated pointwise. When $\mathcal{A}$ is an abelian category, we will denote by
$\mathcal{D}(\mathcal{A})$ its derived category, which is the one obtained from $\mathcal{C}(\mathcal{A})$ by formally inverting the quasi-isomorphisms (see \cite{V} for the details). Note that, in principle, the morphisms in $\mathcal{D}(\mathcal{A})$ between two objects do not form a set, but a proper class. Therefore, in several parts of this paper, we will require that $\mathcal{D}(\mathcal{A})$ has $\text{Hom}$ sets. We shall denote by $\mathcal{D}^-(\mathcal{A})$ (resp. $\mathcal{D}^+(\mathcal{A})$, resp. $\mathcal{D}^b(\mathcal{A})$) the full subcategory of $\mathcal{D}(\mathcal{A})$ consisting of those complexes $X^\bullet$ such that $H^k(X^\bullet )=0$, for all $k\gg 0$ (resp. $k\ll 0$, resp. $|k| \gg 0$), where $H^k:\mathcal{D}(\mathcal{A})\longrightarrow\mathcal{A}$ denotes the $k$-th cohomology functor, for each $k\in\mathbb{Z}$. 

When $\mathcal{D}$ is a triangulated category with coproducts,  we will use the term \emph{Milnor colimit} of a sequence of morphisms $X_0\stackrel{x_1}{\longrightarrow}X_1\stackrel{x_2}{\longrightarrow}\cdots\stackrel{x_n}{\longrightarrow}X_n\stackrel{x_{n+1}}{\longrightarrow}\cdots$ what in \cite{N} is called homotopy colimit. It will be denoted $\text{Mcolim}(X_n)$, without reference to the $x_n$.  However, for the dual concept in a triangulated category with products we will retain the term \emph{homotopy limit}, denoted $\text{Holim}(X_n)$.

Given two subcategories $\mathcal{X}$ and $\mathcal{Y}$ of the triangulated category $\mathcal{D}$, we will denote by $\mathcal{X}\star\mathcal{Y}$ the subcategory of $\mathcal{D}$ consisting of the objects $M$ which fit into a triangle $X\longrightarrow M\longrightarrow Y\stackrel{+}{\longrightarrow}$, where $X\in\mathcal{X}$ and $Y\in\mathcal{Y}$. Due to the octahedral axiom, the operation $\star$ is associative, so that  $\mathcal{X}_1\star\mathcal{X}_2\star \cdots\star\mathcal{X}_n$ is well-defined,  for each family of subcategories $(\mathcal{X}_i)_{1\leq i\leq n}$.
 
A
\emph{t-structure} in $\mathcal{D}$ (see \cite[Section 1]{BBD}) is a pair
$\tau =(\mathcal{U},\mathcal{W})$ of full subcategories, closed under
taking direct summands in $\mathcal{D}$,  which satisfy the
 following  properties:

\begin{enumerate}
\item[i)] $\text{Hom}_\mathcal{D}(U,W[-1])=0$, for all
$U\in\mathcal{U}$ and $W\in\mathcal{W}$;
\item[ii)] $\mathcal{U}[1]\subseteq\mathcal{U}$;
\item[iii)] For each $X\in Ob(\mathcal{D})$, there is a triangle $U\longrightarrow X\longrightarrow
V\stackrel{+}{\longrightarrow}$ in $\mathcal{D}$, where
$U\in\mathcal{U}$ and $V\in\mathcal{W}[-1]$ (equivalently, we have  $\mathcal{D}=\mathcal{U}\star (\mathcal{W}[-1])$).
\end{enumerate}
It is easy to see that in such case $\mathcal{W}=\mathcal{U}^\perp
[1]$ and $\mathcal{U}={}^\perp (\mathcal{W}[-1])={}^\perp
(\mathcal{U}^\perp )$. For this reason, we will write a t-structure
as $\tau =(\mathcal{U},\mathcal{U}^\perp [1])$. We will call $\mathcal{U}$
and $\mathcal{U}^\perp$ the \emph{aisle} and the \emph{co-aisle} of
the t-structure, which are, respectively, a suspended and a cosuspended subcategory of $\mathcal{D}$. The objects $U$ and $V$ in the above
triangle are uniquely determined by $X$, up to isomorphism, and
define functors
$\tau_\mathcal{U}:\mathcal{D}\longrightarrow\mathcal{U}$ and
$\tau^{\mathcal{U}^\perp}:\mathcal{D}\longrightarrow\mathcal{U}^\perp$
which are right and left adjoints to the respective inclusion
functors. We call them the \emph{left and right truncation functors}
with respect to the given t-structure. Note that
$(\mathcal{U}[k],\mathcal{U}^\perp [k+1])$ is also a t-structure in
$\mathcal{D}$, for each $k\in\mathbb{Z}$. The full subcategory
$\mathcal{H}=\mathcal{U}\cap\mathcal{W}=\mathcal{U}\cap\mathcal{U}^\perp
[1]$ is called the \emph{heart} of the t-structure and it is an
abelian category, where the short exact sequences `are' the
triangles in $\mathcal{D}$ with its three terms in $\mathcal{H}$.
Moreover, with the obvious abuse of notation,  the assignments
$X\rightsquigarrow (\tau_{\mathcal{U}}\circ\tau^{\mathcal{U}^\perp
[1]})(X)$ and $X\longrightarrow (\tau^{\mathcal{U}^\perp
[1]}\circ\tau_\mathcal{U})(X)$ define   naturally isomorphic
functors $\mathcal{D}\longrightarrow\mathcal{H}$ which are
cohomological (see \cite{BBD}). The t-structure $\tau =(\mathcal{U},\mathcal{U}^\perp [1])$ will be called \emph{left (resp. right) nondegenerate} when $\bigcap_{k\in\mathbb{Z}}\mathcal{U}[k]=0$ (resp. $\bigcap_{k\in\mathbb{Z}}\mathcal{U}^\perp [k]=0$). It will be called \emph{nondegenerate} when it is left and right nondegenerate. We shall say that $\tau$ is a \emph{semi-orthogonal decomposition} when $\mathcal{U}=\mathcal{U}[1]$ (equivalently, $\mathcal{U}^\perp =\mathcal{U}^\perp [1]$). In this case $\tau =(\mathcal{U},\mathcal{U}^\perp )$ and both $\mathcal{U}$ and $\mathcal{U}^\perp$ are thick subcategories of $\mathcal{D}$.

When in the last paragraph $\mathcal{D}$ has coproducts, the aisle $\mathcal{U}$ is closed under coproducts, but the coaisle $\mathcal{U}^\perp$ need not be so. When this is also the case or, equivalently, when the truncation functor $\tau_\mathcal{U}:\mathcal{D}\longrightarrow\mathcal{U}$ preserves coproducts, we shall say that $\tau$ is a \emph{smashing t-structure}. Dually, when $\mathcal{D}$ has products, $\tau$ is said to be a \emph{co-smashing t-structure} when $\mathcal{U}$ is closed under taking products. 
Assuming that $\mathcal{D}$ has coproducts, if $\mathcal{S}\subset\mathcal{U}$ is any class (or set) of objects, we shall say that the t-structure $\tau$ is \emph{generated by $\mathcal{S}$} or that \emph{$\mathcal{S}$ is a class (resp. set) of generators} of $\tau$ when $\mathcal{U}^\perp =\mathcal{S}^{\perp_{\leq 0}}$. We shall say that $\tau$ is
\emph{compactly generated} when there is a set
 of compact objects which generates $\tau$. Note that such a t-structure is always smashing. Generalizing a  bit the classical definition, the following phenomenon will be called \emph{infinite d\'evissage} (see \cite[Theorem 12.1]{KN}).

\begin{lem} \label{lem.infinite devissage}
Let $\mathcal{D}$ be a triangulated category with coproducts, let $\tau =(\mathcal{U},\mathcal{U}^\perp [1])$ be a t-structure (resp. semi-orthogonal decomposition) in $\mathcal{D}$ and let $\mathcal{S}\subset\mathcal{U}$ be  a set of objects which are compact in $\mathcal{D}$ and such that $\mathcal{S}^{\perp_{\leq 0}}=\mathcal{U}^\perp$  (resp. $\mathcal{S}^{\perp_{k\in\mathbb{Z}}}=\mathcal{U}^\perp$). If $\mathcal{V}\subseteq\mathcal{U}$ is a closed under coproduct suspended (resp. localizing) subcategory of $\mathcal{D}$ such that $\mathcal{S}\subset\mathcal{V}$, then we have $\mathcal{V}=\mathcal{U}$. 
\end{lem}

\begin{ex} \label{exem. canonical t-structure}
The following examples of t-structures are relevant for us:

\begin{enumerate}
\item (see \cite[Example 1.3.2]{BBD}) Let $\mathcal{A}$ be an abelian category and, for each $k\in\mathbb{Z}$,  denote by $\mathcal{D}^{\leq
k}(\mathcal{A})$ (resp. $\mathcal{D}^{\geq k}(\mathcal{A})$) the
 subcategory of $\mathcal{D}(\mathcal{A})$ consisting of the
complexes $X^\bullet$ such that $H^j(X^\bullet )=0$, for all $j>k$ (resp. $j<k$). The
pair $(\mathcal{D}^{\leq
k}(\mathcal{A}),\mathcal{D}^{\geq k}(\mathcal{A}))$ is a t-structure
in $\mathcal{D}(\mathcal{A})$ whose heart is equivalent to
$\mathcal{A}$. Its left and right truncation functors will be
denoted by $\tau^{\leq
k}:\mathcal{D}(\mathcal{A})\longrightarrow\mathcal{D}^{\leq
k}(\mathcal{A})$ and
$\tau^{>k}:\mathcal{D}(\mathcal{A})\longrightarrow\mathcal{D}^{>k}(\mathcal{A}):=\mathcal{D}^{\geq
k}(\mathcal{A})[-1]$. For $k=0$, the  t-structure is known as
the \emph{canonical t-structure} in $\mathcal{D}(\mathcal{A})$.
 
\item Let $\mathcal{D}$ be a triangulated category with
coproducts and let $T$ be a
classical tilting object (see Definition \ref{def.classical silting objects}). It is well-known, and will be a particular case of our results in Section \ref{sect.silting bijection}, that the pair $\tau_T=(T^{\perp_{>0}},T^{\perp_{<0}})$ is a t-structure in $\mathcal{D}$ generated by $T$. Its heart $\mathcal{H}_T$ is equivalent to $\text{Mod}-E$, where $E=\text{End}_\mathcal{D}(T)$. 
\end{enumerate}

\end{ex}

\section{The co-heart of a t-structure} \label{sect.coheart}

In this section $\mathcal{D}$ will be a triangulated category and $\tau:=(\mathcal{U},\mathcal{U}^\perp [1])$ will be a t-structure in $\mathcal{D}$. Apart from its heart $\mathcal{H}=\mathcal{U}\cap\mathcal{U}^\perp [1]$, we will also consider its \emph{co-heart} $\mathcal{C}:=\mathcal{U}\cap {}^\perp\mathcal{U}[1]$. Note that we do not assume the existence of any co-t-structure in $\mathcal{D}$ with $\mathcal{C}$ as its co-heart. We will denote by $\tilde{H}:\mathcal{D}\longrightarrow\mathcal{H}$ the associated cohomological functor. 

The objects of the co-heart were called ``$\text{Ext}$-projectives with respect to $\mathcal{U}$'' in \cite{AST}. The proof of assertion 1 of the next lemma is essentially that of \cite[Lemma 1.3]{AST}. We reproduce it here since it will be frequently used throughout the paper. 

\begin{lem} \label{lem.coheart projective in heart}
The following assertions hold:

\begin{enumerate}
\item The restriction to the co-heart $\tilde{H}_{| \mathcal{C}}:\mathcal{C}\longrightarrow\mathcal{H}$ is a fully faithful functor and its image consists of projective objects. 
\item Suppose now that $\mathcal{D}$ has coproducts and let $C\in\mathcal{C}$ be any object of the co-heart. The following statements hold true:

\begin{enumerate}

\item The functor $\tilde{H}_{| \mathcal{C}}:\mathcal{C}\longrightarrow\mathcal{H}$ preserves coproducts;
\item  $\tilde{H}(C)$ is compact in $\mathcal{H}$ if and only if $C$ is compact in $\mathcal{U}$. In particular, when $\tau$ is smashing, $\tilde{H}(C)$ is compact in $\mathcal{H}$ if and only if $C$ is compact in $\mathcal{D}$.
\end{enumerate}
\end{enumerate}
\end{lem}
\begin{proof}

 For each $U\in\mathcal{U}$, we have a canonical truncation triangle 

$$\tau_{\mathcal{U}[1]}U\stackrel{j_U}{\longrightarrow}U\stackrel{p_U}{\longrightarrow}\tilde{H}(U)\stackrel{w_U}{\longrightarrow}(\tau_{\mathcal{U}[1]}U)[1].$$

1) Let $C,C'\in\mathcal{C}$ be any two objects and consider the induced map $\text{Hom}_\mathcal{C}(C,C')\longrightarrow\text{Hom}_\mathcal{H}(\tilde{H}(C),\tilde{H}(C'))$. If $f:C\longrightarrow C'$ is in the kernel of the latter map,  then $p_{C'}\circ f=0$, which implies that $f$ factors in the form $f:C\stackrel{g}{\longrightarrow}\tau_{\mathcal{U}[1]}C'\stackrel{j_{C'}}{\longrightarrow}C'$. But $g=0$ since $C\in {}^\perp\mathcal{U}[1]$, and hence $\tilde{H}_{| \mathcal{C}}$ is faithful.  

Suppose now that $h:\tilde{H}(C)\longrightarrow\tilde{H}(C')$ is a morphism in $\mathcal{H}$. Then the composition $w_{C'}\circ h\circ p_C$ is zero, since $\text{Hom}_\mathcal{D}(C,?)$ vanishes on $\mathcal{U}[2]$. We then get a morphism $f:C\longrightarrow C'$ such that $p_{C'}\circ f=h\circ p_C$. Note that, by definition of $\tilde{H}$, we also have that $p_{C'}\circ f=\tilde{H}(f)\circ p_C$. It follows that $(h-\tilde{H}(f))\circ p_C=0$, which implies that $h-\tilde{H}(f)$ factors in the form $\tilde{H}(C)\stackrel{w_C}{\longrightarrow}(\tau_{\mathcal{U}[1]}C)[1]\stackrel{v}{\longrightarrow}\tilde{H}(C')$. But the second arrow in this composition is zero since it has domain in $\mathcal{U}[2]$ and codomain in $\mathcal{U}^\perp [1]$. Therefore $\tilde{H}_{| \mathcal{C}}$ is also full. 

Note now that if $C\in\mathcal{C}$ and $M\in\mathcal{H}$, then adjunction gives an isomorphism, functorial on both variables,

\begin{center}
$\text{Hom}_{\mathcal{H}}(\tilde{H}(C),M)=\text{Hom}_{\mathcal{U}^\perp [1]}(\tau^{\mathcal{U}^\perp [1]}C,M)\cong\text{Hom}_{\mathcal{D}}(C,M)$.
\end{center} 
If now $\pi :M\longrightarrow N$ is an epimorphism in $\mathcal{H}$ and we put $K':=\text{Ker}_\mathcal{H}(\pi )$, then we get a triangle $M\stackrel{\pi}{\longrightarrow}N\longrightarrow K'[1]\stackrel{+}{\longrightarrow}$. Since $\text{Hom}_\mathcal{D}(C,K'[1])=0$, for all $C\in\mathcal{C}$, we conclude that the induced map 

\begin{center}
$\pi_*:\text{Hom}_{\mathcal{H}}(\tilde{H}(C),M)\cong\text{Hom}_{\mathcal{D}}(C,M)\longrightarrow\text{Hom}_{\mathcal{D}}(C,N)\cong\text{Hom}_{\mathcal{H}}(\tilde{H}(C),N)$
\end{center}
is surjective, which shows that $\tilde{H}(C)$ is a projective object of $\mathcal{H}$.

\vspace*{0.3cm}

2) All throughout the proof of this assertion, we will use that $\mathcal{H}$ is AB3 (see \cite[Proposition 3.2]{PS1}).

a) Note also that if $(M_i)_{i\in I}$ is a family of objects of $\mathcal{H}$, then its coproduct in this category is precisely $\tilde{H}(\coprod_{i\in I}M_i)$. This is a direct consequence of the fact that $\tilde{H}_{| \mathcal{U}}:\mathcal{U}\longrightarrow\mathcal{H}$ is left adjoint to the inclusion $\mathcal{H}\hookrightarrow\mathcal{U}$ (see \cite[Lemma 3.1]{PS1}) and the inclusion $\mathcal{U}\longrightarrow\mathcal{D}$ preserves coproducts. On the other hand, if $(C_i)_{i\in I}$ is a family of objects of $\mathcal{C}$, by \cite[Remark 1.2.2]{N} we have a triangle $$\coprod_{i\in I}\tau_{\mathcal{U}[1]}C_i\longrightarrow\coprod_{i\in I}C_i\longrightarrow\coprod_{i\in I}\tilde{H}(C_i)\stackrel{+}{\longrightarrow} .$$ It immediately follows that $\tilde{H}(\coprod_{i\in I}C_i)\cong\tilde{H}(\coprod_{i\in I}\tilde{H}(C_i))$, and the second member of this isomorphism is precisely the coproduct of the $\tilde{H}(C_i)$ in $\mathcal{H}$.

b)  Fix $C\in\mathcal{C}$. By the first triangle of this proof and the fact that $\text{Hom}_\mathcal{D}(C,?)$ vanishes on $\mathcal{U}[1]$, we have a functorial isomorphism $(p_U)_*:\text{Hom}_\mathcal{D}(C,U)\stackrel{\sim}{\longrightarrow}\text{Hom}_\mathcal{D}(C,\tilde{H}(U))$, for each $U\in\mathcal{U}$. Let now $(M_i)_{i\in I}$ be any family of objects of $\mathcal{H}$. Considering that $\tilde{H}_{| \mathcal{U}}:\mathcal{U}\longrightarrow\mathcal{H}$ is left adjoint to the inclusion functor, we then have a chain of morphisms:

\begin{center}
$\coprod_{i\in I}\text{Hom}_{\mathcal{D}}(C,M_i)=\coprod_{i\in I}\text{Hom}_{\mathcal{U}}(C,M_i)\cong\coprod_{i\in I}\text{Hom}_{\mathcal{H}}(\tilde{H}(C),M_i)\stackrel{can}{\longrightarrow}\text{Hom}_{\mathcal{H}}(\tilde{H}(C),\coprod_{i\in I}^\mathcal{H}M_i)=\text{Hom}_{\mathcal{H}}(\tilde{H}(C),\tilde{H}(\coprod_{i\in I}M_i))\cong\text{Hom}_{\mathcal{U}}(C,\tilde{H}(\coprod_{i\in I}M_i))=\text{Hom}_{\mathcal{D}}(C,\tilde{H}(\coprod_{i\in I}M_i))\cong\text{Hom}_{\mathcal{D}}(C,\coprod_{i\in I}M_i)$, 
\end{center} 
where $\coprod_{i\in I}^\mathcal{H}$ denotes the coproduct in $\mathcal{H}$. 
It follows that $\tilde{H}(C)$ is compact in $\mathcal{H}$ if and only if $\text{Hom}_\mathcal{D}(C,?)$ preserves coproducts (in $\mathcal{D}$!) of objects of $\mathcal{H}$. 

But if $(U_i)_{i\in I}$ is any family of objects in $\mathcal{U}$, then we get a triangle

$$\coprod_{i\in I}\tau_{\mathcal{U}[1]}U_i\longrightarrow\coprod_{i\in I}U_i\longrightarrow\coprod_{i\in I}\tilde{H}(U_i)\stackrel{+}{\longrightarrow} $$
(see \cite[Remark 1.2.2]{N}), which gives an isomorphism $\text{Hom}_\mathcal{D}(C,\coprod_{i\in I}U_i)\stackrel{\sim}{\longrightarrow}\text{Hom}_\mathcal{D}(C,\coprod_{i\in I}\tilde{H}(U_i))$ since $\text{Hom}_\mathcal{D}(C,?)$ vanishes on $\mathcal{U}[1]$. It follows immediately that $\text{Hom}_\mathcal{D}(C,?)$ preserves coproducts (in $\mathcal{D}$) of objects of $\mathcal{H}$ if and only if it preserves coproducts of objects of $\mathcal{U}$. 
This proves that $\tilde{H}(C)$ is compact in $\mathcal{H}$ if and only if $C$ is compact in $\mathcal{U}$.

Assume now that $\tau$ is smashing. Proving that $C$ is compact in $\mathcal{D}$ whenever $\text{Hom}_\mathcal{D}(C,?)$ preserves coproducts of objects in $\mathcal{U}$ is a standard argument. Let $(X_i)_{i\in I}$ be any family of objects of $\mathcal{D}$ and consider the adjoint pair $(\iota_\mathcal{U},\tau_\mathcal{U})$, where $\iota_\mathcal{U}:\mathcal{U}\hookrightarrow\mathcal{D}$ is the inclusion functor. Bearing in mind  that,  due to the smashing condition of $\tau$, the functor $\tau_\mathcal{U}:\mathcal{D}\longrightarrow\mathcal{U}$ preserves coproducts, we then get a chain of isomorphisms

\begin{center}
$\coprod_{i\in I}\text{Hom}_\mathcal{D}(C,X_i)=\coprod_{i\in I}\text{Hom}_\mathcal{D}(\iota_\mathcal{U}(C),X_i)\cong\coprod_{i\in I}\text{Hom}_\mathcal{U}(C,\tau_\mathcal{U}(X_i))\cong\text{Hom}_\mathcal{D}(C,\coprod \tau_\mathcal{U}(X_i))=\text{Hom}_\mathcal{D}(C,\coprod \tau_\mathcal{U}(X_i))\cong\text{Hom}_\mathcal{D}(C,\tau_\mathcal{U}(\coprod X_i))=\text{Hom}_\mathcal{U}(C,\tau_\mathcal{U}(\coprod X_i))\cong\text{Hom}_\mathcal{D}(\iota_\mathcal{U}(C),\coprod X_i)=\text{Hom}_\mathcal{D}(C,\coprod X_i)$.
\end{center}
\end{proof}

 The following result is a generalization of \cite[Proposition 1.3.7]{BBD}:

\begin{lem} \label{lem.description of aisle by homology}
Let  $X$ be an object of $\mathcal{D}$. The following assertions are equivalent:

\begin{enumerate}
\item $\tilde{H}^j(X)=0$, for all $j>0$ (resp. $j\leq 0$).
\item $\tau^{\mathcal{U}^\perp}X$ is in $\bigcap_{n\in\mathbb{Z}}\mathcal{U}^\perp [n]$ (resp. $\tau_\mathcal{U}X$ is in  $\bigcap_{n\in\mathbb{Z}}\mathcal{U}[n]$).
\end{enumerate}
\end{lem}
\begin{proof}
We will prove the `not in-between brackets' assertion, the other one will follow by the duality principle. 

Note that if $U\in\mathcal{U}$ then $\tilde{H}^j(U)=\tilde{H}(U[j])=\tau^{\mathcal{U}^\perp [1]}(U[j])=0$, for all $j>0$. If follows from this that if  $U\in\bigcap_{n\in\mathbb{Z}}\mathcal{U} [n]$, then we have that $\tilde{H}^j(U)=\tilde{H}^1(U[j-1])=0$, for all $j\in\mathbb{Z}$,  because $U[j-1]\in\mathcal{U}$.  By the duality principle, we first get that $\tilde{H}^j$ vanishes on $\mathcal{U}^\perp$, for all $j\leq 0$, and we also get that if $V\in\bigcap_{n\in\mathbb{Z}}\mathcal{U}^\perp [n]$, then $\tilde{H}^j(V)=0$, for all $j\in\mathbb{Z}$. 

In the rest of the proof we put  $U=\tau_\mathcal{U}X$ and $V=\tau^{\mathcal{U}^\perp }X$, and consider the corresponding truncation triangle $U\longrightarrow X\longrightarrow V\stackrel{+}{\longrightarrow}$.
The long exact sequence associated to $\tilde{H}$ gives an exact sequence $$0=\tilde{H}^j(U)\longrightarrow\tilde{H}^j(X)\longrightarrow\tilde{H}^j(V)\longrightarrow\tilde{H}^{j+1}(U)=0 $$ in $\mathcal{H}$,  for all $j>0$. It follows that assertion 1 holds if and only if $\tilde{H}^j(V)=0$, for all $j>0$. But this in turn is equivalent to say that $\tilde{H}^j(V)=0$, for all $j\in\mathbb{Z}$, due to the previous paragraph. The implication $2)\Longrightarrow 1)$ is then clear. 

On the other hand, the implication $1)\Longrightarrow 2)$ reduces to prove that if $V\in\mathcal{U}^\perp$ and $\tilde{H}^j(V)=0$, for all $j\in\mathbb{Z}$, then $V\in\bigcap_{n\in\mathbb{Z}}\mathcal{U}^\perp [n]$. Suppose that this is not the case, so that there exists an integer $n>0$ such that $V\in\mathcal{U}^\perp [-n+1]\setminus\mathcal{U}^\perp [-n]$. That is, we have $V[n-1]\in\mathcal{U}^\perp$, but $V[n]\not\in\mathcal{U}^\perp$. But then we have $0=\tilde{H}^n(V)=\tilde{H}(V[n])=\tau_\mathcal{U}(V[n])$ since $V[n]\in\mathcal{U}^\perp [1]$. It follows that $V[n]\in\mathcal{U}^\perp$, which is a contradiction. 
\end{proof}

\begin{lem} \label{lem.t-structures generated by co-heart}
 The following assertions are equivalent:

\begin{enumerate}
\item $\tau$ is generated by $\mathcal{C}$.
\item $\tau$ is a left non-degenerated t-structure and, for each $M\in\mathcal{H}\setminus\{0\}$, there is a nonzero morphism $f:C\longrightarrow M$, where $C\in\mathcal{C}$.
\end{enumerate}
In this case $\tilde{H}(\mathcal{C})$ is a class of projective generators of $\mathcal{H}$. 
\end{lem}
\begin{proof}
$1)\Longrightarrow 2)$  Suppose that $M\in\mathcal{H}$ and $\text{Hom}_\mathcal{D}(?,M)$ vanishes on $\mathcal{C}$. Then $\text{Hom}_\mathcal{D}(?,M)$ vanishes on $\mathcal{C}[k]$, for each $k\geq 0$, because $M\in\mathcal{U}^\perp [1]$. Assertion 1 then says that $M\in\mathcal{U}^\perp$, so that $M\in\mathcal{U}\cap\mathcal{U}^\perp =0$.

Let us take $U\in\bigcap_{n\in\mathbb{Z}}\mathcal{U}[n]$. By Lemma \ref{lem.description of aisle by homology}, we get that $\tilde{H}^j(U)=0$, for all $j\leq 0$ (in fact, it even follows that $H^j(U)=0$, for all $j\in\mathbb{Z}$). Bearing in mind that $\mathcal{C}\subset\mathcal{U}$, for each $j\leq 0$ and each $C\in\mathcal{C}$, we have that $$0=\text{Hom}_\mathcal{D}(C,\tilde{H}^j(U))=\text{Hom}_\mathcal{D}(C,\tau_\mathcal{U}\tau^{\mathcal{U}^\perp [1]}(U[j]))\cong\text{Hom}_\mathcal{D}(C,\tau^{\mathcal{U}^\perp [1]}(U[j])).$$ But, by the fact that $\text{Hom}_\mathcal{D}(C,?)$ vanishes on $\mathcal{U}[1]$, we also have an isomorphism $\text{Hom}_\mathcal{D}(C,U[j])\cong\text{Hom}_\mathcal{D}(C,\tau^{\mathcal{U}^\perp [1]}(U[j]))$. It follows that $\text{Hom}_\mathcal{D}(C,U[j])=0$, for all $C\in\mathcal{C}$ and all $j\leq 0$, so that $U\in\mathcal{C}^{\perp_{\leq 0}}=\mathcal{U}^\perp$ and hence  $U\in\mathcal{U}\cap\mathcal{U}^\perp =0$.

$2)\Longrightarrow 1)$ Since we have $\mathcal{C}[k]\subseteq\mathcal{U}$, for all $k\geq 0$,  the inclusion $\mathcal{U}^\perp\subseteq\mathcal{C}^{\perp_{\leq 0}}$ is obvious.  Conversely, let us consider $Y\in\mathcal{C}^{\perp_{\leq 0}}$. Since we have an isomorphism $\text{Hom}_\mathcal{D}(C[k],\tau_\mathcal{U}(Y))\cong\text{Hom}_\mathcal{D}(C[k],Y)$, for all $C\in\mathcal{C}$ and all $k\geq 0$, we can assume without loss of generality that $Y\in\mathcal{U}$, and the goal is shifted to prove that $Y=0$. Considering the triangle
$\tau_{\mathcal{U}[1]}Y\stackrel{j_Y}{\longrightarrow}Y\stackrel{p_Y}{\longrightarrow}\tilde{H}(Y)\stackrel{w_Y}{\longrightarrow}(\tau_{\mathcal{U}[1]}Y)[1]$
from the first paragraph of the proof of Lemma \ref{lem.coheart projective in heart}, we see that $(p_Y)_*:\text{Hom}_\mathcal{D}(C,Y)\longrightarrow\text{Hom}_\mathcal{D}(C,\tilde{H}(Y))$ is an isomorphism, for each $C\in\mathcal{C}$. We then have that $\text{Hom}_\mathcal{D}(?,\tilde{H}(Y))$ vanishes on $\mathcal{C}$, which, by assertion 2, means that $\tilde{H}(Y)=0$. That is, we have that $Y\in\mathcal{U}[1]$. But then $Y=Y'[1]$, where $Y'\in\mathcal{U}$ and $\text{Hom}_\mathcal{D}(C[k],Y')=\text{Hom}_\mathcal{D}(C[k+1],Y)=0$, for all $k\geq 0$. It follows that $Y'\in\mathcal{U}[1]$, so that $Y\in\mathcal{U}[2]$. By iterating the process, we conclude that $Y\in\bigcap_{n\geq 0}\mathcal{U}[n]=0$. 

Suppose now that assertions 1 and 2 hold. For each $0\neq M\in\mathcal{H}$, we have an object  $C\in\mathcal{C}$ and a nonzero morphism $C\longrightarrow M$ which,  by the proof of Lemma \ref{lem.coheart projective in heart},  factors in the form $C\longrightarrow\tilde{H}(C)\longrightarrow M$. Since $\tilde{H}(\mathcal{C})$ consists of projective objects (see Lemma \ref{lem.coheart projective in heart}), we get that it is a class of projective generators of $\mathcal{H}$. 
\end{proof}

\section{A silting bijection at the unbounded level} \label{sect.silting bijection}

The following concept will be important for us.

\begin{opr} \label{def.strongly nonpositive}
Let $\mathcal{D}$ be a triangulated category. A class (or set)  $\mathcal{T}$ of objects in such a category will be called \emph{nonpositive} (resp. \emph{exceptional}) when $\text{Hom}_\mathcal{D}(T,T'[k])=0$, for all $T,T'\in\mathcal{T}$ and all integers $k>0$ (resp. $k\neq 0$). When $\mathcal{D}$ has coproducts, we will say that $\mathcal{T}$ is \emph{strongly nonpositive} (resp. \emph{strongly exceptional}) when $\text{Sum}(\mathcal{T})$ (or, equivalently, $\text{Add}(\mathcal{T})$) is a nonpositive (resp. exceptional) class.  Note that  $\mathcal{T}$ is strongly nonpositive (resp. strongly exceptional) if and only if $\text{Hom}_\mathcal{D}(T,\coprod_{i\in I}T_i[k])=0$, for all  $T\in\mathcal{T}$, all families $(T_i)_{i\in I}$  in $\mathcal{T}$ and all integers $k>0$ (resp. $k\neq 0$).
\end{opr}

All throughout the rest of the section, we assume that $\mathcal{D}$ is a triangulated category with coproducts.

\begin{opr} \label{def.strong presilting}
Let $\mathcal{T}$ be a set of objects of $\mathcal{D}$.
 We shall say that $\mathcal{T}$ is \emph{partial silting}   when the following conditions hold:
\begin{enumerate}
\item The pair  $(\mathcal{U}_\mathcal{T},\mathcal{U}_\mathcal{T}^{\perp }[1]):=({}^\perp (\mathcal{T}^{\perp _{\leq 0}}),\mathcal{T}^{\perp _{<0}})$ is a t-structure in $\mathcal{D}$;
\item $\text{Hom}_\mathcal{D}(T,?)$ vanishes on $\mathcal{U}_\mathcal{T}[1]$, for all $T\in\mathcal{T}$. 
\end{enumerate}

Note that if $\mathcal{T}$ is a partial silting set in $\mathcal{D}$, then $\mathcal{T}\subseteq\mathcal{U}_\mathcal{T}$,  and hence $\mathcal{T}$ is strongly nonpositive. 
A strongly exceptional partial silting set will be called \emph{partial tilting}. 
When $\mathcal{T}=\{T\}$ is a partial silting (resp. partial tilting) set, we will say that $T$ is a \emph{partial silting (resp. partial tilting) object} of $\mathcal{D}$. 

A partial silting (resp. partial tilting) set (resp.  object) will be called a \emph{silting (resp. tilting) set (resp. object)} when it generates $\mathcal{D}$ as a triangulated category.

\end{opr}

\begin{rem} \label{rem.cosilting-cotilting}
Obviously, we have dual notions valid in a triangulated category with products $\mathcal{D}$. We make them explicit since they will eventually appear in the paper. A set $\mathcal{Q}$ of objects of $\mathcal{D}$ is \emph{partial cosilting} when the following conditions hold:

\begin{enumerate}
\item The pair $({}^\perp(\mathcal{V}_\mathcal{Q})[-1],\mathcal{V}_\mathcal{Q}):=({}^{\perp_{<0}}\mathcal{Q},({}^{\perp_{\leq 0}}\mathcal{Q})^\perp)$ is a t-structure in $\mathcal{D}$;
\item $\text{Hom}_\mathcal{D}(?,Q)$ vanishes on $\mathcal{V}_\mathcal{Q}[-1]$, for all $Q\in\mathcal{Q}$.
\end{enumerate}

When in addition $\mathcal{Q}$ cogenerates  $\mathcal{D}$ (i.e. ${}^{\perp_{k\in\mathbb{Z}}}\mathcal{Q}=0$) we will say that $\mathcal{Q}$ is a \emph{cosilting set}. The adjective (partial) cosilting is applied to an object $Q$ when  
$\mathcal{Q}=\{Q\}$ is a (partial) cosilting set. Finally, a (partial) cosilting set $\mathcal{Q}$ is said to be a \emph{(partial) cotilting set} when $\text{Hom}_\mathcal{D}(\prod_{i\in I}Q_i,Q[k])=0$, for all $Q\in\mathcal{Q}$, all families
 of objects $(Q_i)_{i\in I}$ in $\mathcal{Q}$ and all integers $k\neq 0$. 

\end{rem}

\begin{rem} \label{rem.Frobenius exat}
For many  `well-behaved' triangulated categories with coproducts, condition 1 of Definition \ref{def.strong presilting} is automatic. For instance, if $\mathcal{E}$ is a Frobenius exact category of Grothendieck type as defined in \cite{Sto}, then it was proved in \cite[Theorem 2.13 and Proposition 3.9]{SS} that, for each set $\mathcal{T}$ of objects of $\mathcal{D}:=\underline{\mathcal{E}}$, the pair $({}^\perp (\mathcal{T}^{\perp _{ \leq 0}}),\mathcal{T}^{\perp _{<0}})$ is a t-structure in $\mathcal{D}$ (see also \cite{AJS}). Actually, ${}^\perp (\mathcal{T}^{\perp _{ \leq 0}})$ consists of those objects which are isomorphic in $\mathcal{D}$ to direct summands of objects that admit a continuous transfinite filtration in $\mathcal{E}$ with successive factors in $\bigcup_{k\geq 0}\mathcal{T}[k]$.  
\end{rem}

\begin{rem} 
In an independent recent work, Psaroudakis and Vit\'oria (see \cite[Definition 4.1]{PV}) call an object $T$ of $\mathcal{D}$ silting when $(T^{\perp_{>0}},T^{\perp_{<0}})$ is a t-structure such that $T\in T^{\perp_{>0}}$. From Theorem \ref{teor.presilting t-structures} below one can easily derive that their definition is equivalent to our definition of silting object. 
\end{rem}

We immediately get some examples. An extension of the second one will be given in later sections. 

\begin{ex} \label{ex.examples partial silting}

\begin{enumerate}
\item Let $\mathcal{D}$ be compactly generated (e.g. $\mathcal{D}=\mathcal{D}(A)$, where $A$ is a dg algebra). 
 Any   set $\mathcal{T}$ of compact objects such that $\text{Hom}_\mathcal{D}(T,T'[k])=0$, for all $T,T'\in\mathcal{T}$ and all $k>0$ (resp. $k\neq 0$), is a partial  silting (resp. partial tilting) set. In such case, it is a silting (resp. tilting) set if and only if  $\mathcal{S}\subseteq\text{thick}_{\mathcal{D}}(\mathcal{T})$, for some (resp. every) set $\mathcal{S}$ of compact generators of $\mathcal{D}$ (e.g. $\mathcal{S}=\{A\}$ when $\mathcal{D}=\mathcal{D}(A)$).   
\item Let $\mathcal{D}=\mathcal{D}(A)$ be the derived category of an ordinary algebra $A$ and $\mathcal{T}=\{T\}$, where $T$ is a \emph{big silting complex} (see \cite{AMV}, also called \emph{big semi-tilting complex} in \cite{W}), that is,  $\text{thick}_{\mathcal{D}(A)}(\text{Add}(T))=\text{thick}_{\mathcal{D}(A)}(\text{Add}(A))=\mathcal{K}^b(\text{Proj}-A)$ and $\text{Hom}_\mathcal{D}(T,T[k]^{(I)})=0$, for all sets $I$ and integers $k>0$, then $T$ is silting in the sense of Definition \ref{def.strong presilting}. 
\end{enumerate}
\end{ex}
\begin{proof}
1) We will prove the statement for the partial silting case. The corresponding statement when replacing `silting' by `tilting' is clear.

By \cite[Theorems 12.1 and 12.2]{KN}, we know that each $X\in\mathcal{U}_\mathcal{T}[1]$ is the Milnor colimit of a sequence $$X_0\stackrel{f_1}{\longrightarrow}X_1\stackrel{f_2}{\longrightarrow}\cdots\stackrel{f_n}{\longrightarrow}X_n\stackrel{f_{n+1}}{\longrightarrow}\cdots,$$ where $X_0$ and all cones of the $f_n$ are in $\text{Sum}(\coprod_{T\in\mathcal{T},k>0}T[k])$. It follows by induction that $\text{Hom}_\mathcal{D}(T,X_n)=0$, for all $T\in\mathcal{T}$ and all $n\in\mathbb{N}$, from which we get that $\text{Hom}_\mathcal{D}(T,X)\cong\varinjlim\text{Hom}_\mathcal{D}(T,X_n)=0$ since each $T\in\mathcal{T}$ is a compact object. 

In this situation $\mathcal{T}$ is a silting set if and only if $\mathcal{T}$ is a set of compact generators of $\mathcal{D}$. By \cite[Theorem 5.3]{K}, this is equivalent to the condition mentioned in the statement.  

2) By  \cite[Proposition 4.2]{AMV}, we know that $(T^{\perp_{>0}},T^{\perp_{<0}})$ is the t-structure generated by $T$ and that $T$ generates $\mathcal{D}(A)$. But $\text{Hom}_{\mathcal{D}(A)}(T,?)$ clearly vanishes on $T^{\perp_{>0}}[1]=T^{\perp_{\geq 0}}$, so that conditions 1 and 2 of Definition \ref{def.strong presilting} are satisfied. This  and the generating condition imply that $T$ is a silting object of $\mathcal{D}(A)$.
\end{proof} 

Historically, silting and tilting objects or complexes were assumed to be compact. The terminology that we use in this paper, in particular Definition \ref{def.strong presilting}, is reminiscent of the one used for tilting modules and this justifies the following.

\begin{opr} \label{def.classical silting objects}
A \emph{classical (partial) silting} (resp. \emph{classical (partial) tilting}) \emph{set} of $\mathcal{D}$ will be a (partial) silting (resp. (partial) tilting) set consisting of compact objects. An object $T$ will be called \emph{classical (partial) silting object} (resp. \emph{classical (partial) tilting object}) when the set $\{T\}$ is so. 
\end{opr}

\begin{ex}
\begin{enumerate}
\item Let $\mathcal{A}$ be an  abelian category such that $\mathcal{D}(\mathcal{A})$ has $\text{Hom}$ sets and arbitrary coproducts (see Setup \ref{setup} below), and let  $\mathcal{P}$ be a set of projective generators of $\mathcal{A}$. Then $\mathcal{P}$ is a tilting (and hence silting) set of $\mathcal{D}(\mathcal{A})$.
\item Let $\mathcal{A}$ be an abelian category  such that $\mathcal{D}(\mathcal{A})$ has $\text{Hom}$ sets and arbitrary products, and let   $\mathcal{I}$ be a set of injective cogenerators of $\mathcal{A}$. Then $\mathcal{I}$ is a cotilting (and hence cosilting) set in $\mathcal{D}(\mathcal{A})$. 
\end{enumerate}
\end{ex}
\begin{proof}
Example 1 `is' \cite[Example 4.2(ii)]{PV}. Example 2 is dual. 
\end{proof}

\begin{lem} \label{lem.SumT-precovers}
Let $\mathcal{T}$ be a partial silting set in $\mathcal{D}$, let $\tau =(\mathcal{U},\mathcal{U}^\perp [1])$ be its associated t-structure. The following assertions hold:

\begin{enumerate}
\item $\mathcal{T}^{\perp_{>0}}$ consists of the objects $X\in\mathcal{D}$ such that $\tau^{\mathcal{U}^\perp}X\in\mathcal{T}^{\perp_{i\in\mathbb{Z}}}$
\item  Let $f:T'\longrightarrow U$ be a morphism, where $T'\in\text{Sum}(\mathcal{T})$ and $U\in\mathcal{U}$. The following statements are equivalent:

\begin{enumerate}
\item $f$ is a $\text{Sum}(\mathcal{T})$-precover.
\item If $T'\stackrel{f}{\longrightarrow}U\longrightarrow Z\stackrel{+}{\longrightarrow}$ is a triangle in $\mathcal{D}$, then $Z\in\mathcal{U}[1]$. 
\end{enumerate} 

In particular, if $(U_i)_{i\in I}$ is a family of objects of $\mathcal{U}$ and $(f_i:T_i\longrightarrow U_i)_{i\in I}$ is a family of $\text{Sum}(\mathcal{T})$-precovers, then the morphism $\coprod f_i:\coprod_{i\in I}T_i\longrightarrow\coprod_{i\in I}U_i$ is also a $\text{Sum}(\mathcal{T})$-precover. 
\end{enumerate}

\end{lem}
\begin{proof}
1) Due to the definition of partial silting set, we have $\mathcal{U}\subseteq\mathcal{T}^{\perp_{>0}}$, and then the inclusion $\mathcal{U}\star\mathcal{T}^{\perp_{i\in\mathbb{Z}}}\subseteq\mathcal{T}^{\perp_{>0}}$ is clear since  $\mathcal{T}^{\perp_{>0}}$ is closed under extensions. Therefore all objects $X$ such that $\tau^{\mathcal{U}^\perp}X\in\mathcal{T}^{\perp_{i\in\mathbb{Z}}}$ are in $\mathcal{T}^{\perp_{>0}}$. 

Conversely, let $X\in\mathcal{T}^{\perp_{>0}}$ and consider the triangle $$X\longrightarrow\tau^{\mathcal{U}^\perp}X\longrightarrow (\tau_\mathcal{U}X)[1]\stackrel{+}{\longrightarrow}$$ given by truncating with respect to $\tau$. Its outer vertices are in $\mathcal{T}^{\perp_{>0}}$ and, hence, its three vertices are in this subcategory. In particular,  $Y:=\tau^{\mathcal{U}^\perp}X$ satisfies that $\text{Hom}_\mathcal{D}(T,Y[k])=0$, for all $k>0$. But we also have that $Y\in \mathcal{U}^\perp =\mathcal{T}^{\perp_{\leq 0}}$. It follows that $Y\in\mathcal{T}^{\perp_{i\in\mathbb{Z}}}$.

2) We only need to prove the implication $a)\Longrightarrow b)$ for the reverse one is obvious since $\text{Hom}_\mathcal{D}(T,?)$ vanishes on $\mathcal{U}[1]$. Note first that $Z\in\mathcal{U}$ since $Z\in\mathcal{U}\star\text{Sum}(\mathcal{T})[1]\subset\mathcal{U}$ . Moreover,  by applying the long exact sequence associated to $\text{Hom}_\mathcal{D}(T,?)$ and the surjectivity of $f_*=\text{Hom}_\mathcal{D}(T,f):\text{Hom}_\mathcal{D}(T,T')\longrightarrow\text{Hom}_\mathcal{D}(T,U)$, one also gets that $\text{Hom}_\mathcal{D}(T,Z)=0$. We then get that $\text{Hom}_\mathcal{D}(T,Z[j])=0$, for all $j\geq 0$ and all $T\in\mathcal{T}$. By assertion 1, we then have $Z\in\mathcal{T}^{\perp_{\geq 0}}=\mathcal{T}^{\perp_{>0}}[1]=\mathcal{U}[1]\star\mathcal{T}^{\perp_{i\in\mathbb{Z}}}$, so that $\tau^{\mathcal{U}^\perp [1]}Z\in\mathcal{T}^{\perp_{i\in\mathbb{Z}}}$. But then the canonical morphism $Z\longrightarrow\tau^{\mathcal{U}^\perp [1]}Z$ is the zero one. Indeed, we have $\mathcal{U}^\perp =\mathcal{T}^{\perp_{\leq 0}}\supset\mathcal{T}^{\perp_{i\in\mathbb{Z}}}$, which implies that $\mathcal{U}={}^{\perp}(\mathcal{T}^{\perp_{\leq 0}})\subset{}^{\perp}(\mathcal{T}^{\perp_{i\in\mathbb{Z}}})$. We then get that the canonical morphism $\tau_{\mathcal{U}[1]}Z\longrightarrow Z$ is a retraction, so that $Z\in\mathcal{U}[1]$. 
\end{proof}

\begin{lem} \label{lem.presilting set}
Let $\mathcal{T}$ be a   partial silting set  in $\mathcal{D}$ and $(\mathcal{U},\mathcal{U}^{\perp }[1]):=({}^\perp (\mathcal{T}^{\perp _{\leq 0}}),\mathcal{T}^{\perp _{<0}})$ be its associated  t-structure.  Then we have $\mathcal{C}=\text{Add}(\mathcal{T})$, where $\mathcal{C}=\mathcal{U}\cap {}^\perp\mathcal{U}[1]$ is the co-heart of the t-structure. 
\end{lem}
\begin{proof}
Let $C\in\mathcal{C}$ be any object. We claim that the canonical $\text{Sum}(T)$-precover $f:T':=\coprod_{T\in\mathcal{T}}T^{(\text{Hom}_\mathcal{D}(T,C))}\longrightarrow C$ is a retraction. Indeed, if we complete  to a triangle $T'\stackrel{f}{\longrightarrow}C\stackrel{g}{\longrightarrow}Z\stackrel{+}{\longrightarrow}$, then the previous lemma says that $Z\in\mathcal{U}[1]$, which implies that  $g=0$ due to the definition of the coheart. Therefore $f$ is a retraction, as claimed. 
\end{proof}

Recall that an abelian category is \emph{locally small} when the subobjects of any object form a set. We are now ready to prove the main result of this section. 

\begin{thm} \label{teor.presilting t-structures}
Let $\mathcal{D}$ be a triangulated category with coproducts and let $\tau =(\mathcal{U},\mathcal{U}^\perp [1])$ be a t-structure in $\mathcal{D}$, with heart $\mathcal{H}$ and co-heart $\mathcal{C}$. The following assertions are equivalent:

\begin{enumerate}
\item $\tau$ is generated by a  partial silting set.
\item $\tau$ is generated by a set $\mathcal{T}$ such that $\mathcal{T}^{\perp_{>0}}=\mathcal{U}\star\mathcal{T}^{\perp_{k\in\mathbb{Z}}}$.
\item $\tau$ is generated by $\mathcal{C}$ and $\mathcal{H}$ has a set of (projective)  generators.
\item $\tau$ is generated by a set, also  generated by $\mathcal{C}$ and $\mathcal{H}$ is locally small. 

When in addition $\mathcal{D}$ satisfies Brown representability theorem for the dual (e.g. when $\mathcal{D}$ is compactly generated), they  are also equivalent to:

\item $\tau$ is left nondegenerate,  $\mathcal{H}$ has a projective generator and the cohomological functor $\tilde{H}:\mathcal{D}\longrightarrow\mathcal{H}$ preserves products. 

\end{enumerate}
\end{thm}
\begin{proof}

$1)\Longrightarrow 2)$ is a direct consequence of Lemma \ref{lem.SumT-precovers}, when taking as $\mathcal{T}$ any partial silting set which generates $\tau$. 

$2)\Longrightarrow 1)$ The inclusion $\mathcal{U}\subseteq\mathcal{T}^{\perp_{>0}}$ gives condition 2) of the definition of  partial silting set for $\mathcal{T}$. Condition 1) of that definition is automatic since, by hypothesis,  $\tau$ is generated by $\mathcal{T}$. 

$1)\Longrightarrow 3)$ Let $\mathcal{T}$ be  a partial silting set which generates $\tau$. By Lemma \ref{lem.presilting set}, we get that $\tau$ is generated by $\mathcal{C}$. Moreover,  this same lemma together with Lemma \ref{lem.coheart projective in heart} give that $\tilde{H}(\mathcal{C})=\text{Add}(\tilde{H}(\mathcal{T}))$. Therefore $\tilde{H}(\mathcal{T})$ is a set of projective generators of $\mathcal{H}$ due to Lemma \ref{lem.t-structures generated by co-heart}.

$3)\Longrightarrow 4)$ By \cite[Proposition IV.6.6]{St}, we know that $\mathcal{H}$ is locally small. It remains to check that $\tau$ is generated by a set. Let $G$ be a generator of $\mathcal{H}$. The images in $\mathcal{H}$ of morphisms $\tilde{H}(C)\longrightarrow G$, with $C\in\mathcal{C}$, form a set,  denoted  by $\mathcal{Y}$ in the sequel, because  $\mathcal{H}$ is locally small. We put $t(G)=\sum_{Y\in\mathcal{Y}}Y$. This sum exists because $\mathcal{H}$ is AB3 (see \cite[Proposition 3.2]{PS1}). If we had $G/t(G)\neq 0$, then we would have a nonzero morphism $f:\tilde{H}(C')\longrightarrow G/t(G)$, for some $C'\in\mathcal{C}$ (see Lemma \ref{lem.t-structures generated by co-heart} and its proof).  Due to the projective condition of $\tilde{H}(C')$ in $\mathcal{H}$, this morphism would factor in the form $f:\tilde{H}(C')\stackrel{g}{\longrightarrow}G\stackrel{\pi}{\longrightarrow}G/t(G)$, where $\pi$ is the projection. But then $\text{Im}_\mathcal{H}(g)\subseteq t(G)$, which would imply that $f=\pi\circ g=0$, and thus a contradiction.   Therefore we have $t(G)=G$.
For each $Y\in\mathcal{Y}$,  fix now a morphism $f_Y:\tilde{H}(C_Y)\longrightarrow G$, with $C_Y\in\mathcal{C}$,  such that $\text{Im}_\mathcal{H}(f_Y)=Y$. It immediately follows that if we put $\mathcal{T}:=\{C_Y:\text{ }Y\in\mathcal{Y}\}$, then $\tilde{H}(\mathcal{T})$ is a set of projective generators of $\mathcal{H}$. But, due to Lemma \ref{lem.coheart projective in heart},  we then have $\tilde{H}(\mathcal{C})=\text{Add}(\tilde{H}(\mathcal{T}))$, and also $\mathcal{C}=\text{Add}(\mathcal{T})$. Then  $\mathcal{S}:=\mathcal{T}$ is a set which generates $\tau$ since $\tau$ is generated by $\mathcal{C}$.

$4)\Longrightarrow 1)$ By Lemma \ref{lem.t-structures generated by co-heart}, we know that $\tilde{H}(\mathcal{C})$ is a class of projective generators of $\mathcal{H}$. This together with the locally small condition of  $\mathcal{H}$ implies that each object $M$ of $\mathcal{H}$ is an epimorphic image of a (set-indexed) coproduct of objects of $\tilde{H}(\mathcal{C})$. Indeed the images of morphisms $\tilde{H}(C)\longrightarrow M$, with $C\in\mathcal{C}$, form a set and their sum is $M$ due to the projectivity of the objects of $\tilde{H}(\mathcal{C})$. Moreover, using Lemma \ref{lem.coheart projective in heart}(2.a) we conclude that there is an epimorphism $\tilde{H}(C_M)\twoheadrightarrow M$, with $C_M\in\mathcal{C}$.

Let now $\mathcal{S}$ be a set of generators of $\tau$. Note that if $0\neq M\in\mathcal{H}$, then there is a nonzero morphism $f:S\longrightarrow M$, for some $S\in\mathcal{S}$, for otherwise we would have $\text{Hom}(S[k],M)=0$, for all $S\in\mathcal{S}$ and $k\geq 0$, because $\mathcal{H}\subseteq\mathcal{U}^\perp [1]$. That is, we would have that $M\in\mathcal{S}^{\perp_{\leq 0}}=\mathcal{U}^\perp$, and hence $M\in\mathcal{U}\cap\mathcal{U}^\perp =0$, which is a contradiction. Note also that, as in the proof of Lemma \ref{lem.coheart projective in heart}, we have a triangle $\tau_{\mathcal{U}[1]}S\longrightarrow S\longrightarrow\tilde{H}(S)\stackrel{+}{\longrightarrow}$ in $\mathcal{D}$ and $f$ vanishes on 
$\tau_{\mathcal{U}[1]}S$. Therefore 
 $f$ factors through $\tilde{H}(S)$ and we have a nonzero morphism $\tilde{H}(S)\longrightarrow M$, with $S\in\mathcal{S}$. 

Fixing now an epimorphism $\pi_S:\tilde{H}(C_S)\twoheadrightarrow\tilde{H}(S)$ in $\mathcal{H}$, where $C_S\in\mathcal{C}$, for each $S\in\mathcal{S}$, we get that if $\mathcal{T}:=\{C_S\text{: }S\in\mathcal{S}\}$ then $\tilde{H}(\mathcal{T})$ is a set of projective generators of $\mathcal{H}$, because, for each $0\neq M\in\mathcal{H}$, there is a nonzero morphism $\tilde{H}(C_S)\longrightarrow M$, for some $S\in\mathcal{S}$. In particular, we get that $\tilde{H}(\mathcal{C})=\text{Add}(\tilde{H}(\mathcal{T}))$ and, by Lemma \ref{lem.coheart projective in heart}, we  conclude that $\mathcal{C}=\text{Add}(\mathcal{T})$. Then $\mathcal{T}$ is the desired partial silting set which generates $\tau$.

 $1)=3)\Longrightarrow 5)$ For this implication we do not need the full strength of the dual Brown representability theorem. It is enough for $\mathcal{D}$ to have products. 
The left nondegeneracy of $\tau$ follows from Lemma \ref{lem.t-structures generated by co-heart}. By \cite[Proposition 3.2]{PS1} and its proof, we know that $\mathcal{H}$ is AB3 and AB3* and the coproduct and product in $\mathcal{H}$ of a family of objects $(M_i)_{i\in I}$ are $\coprod_{i\in I}^*M_i=\tilde{H}(\coprod_{i\in I}M_i)$ and  $\prod_{i\in I}^*M_i=\tilde{H}(\prod_{i\in I}M_i)$. Here we have used a $*$ superscript to denote the (co)product in $\mathcal{H}$. 
On the other hand, the proof of implication $1)\Longrightarrow 3)$ shows that $\tilde{H}(\mathcal{T})$ is a set of projective generators of $\mathcal{H}$, so that this category is also AB4* (use the dual of \cite[Corollary 3.2.9]{Po}). Put now $T_0:=\coprod_{T\in\mathcal{T}}T$. Then $\tilde{H}(T_0)$ a projective generator of $\mathcal{H}$. We claim that the composition of functors $$\mathcal{D}\stackrel{\tilde{H}}{\longrightarrow}\mathcal{H}\stackrel{\text{Hom}_\mathcal{H}(\tilde{H}(T_0),?)}{\longrightarrow}Ab$$ is naturally isomorphic to the functor $\text{Hom}_\mathcal{D}(T_0,?):\mathcal{D}\longrightarrow Ab$. Indeed, by the proof of Lemma \ref{lem.coheart projective in heart},  we know that if $X\in\mathcal{D}$ and and we put $M=\tilde{H}(X)$, $U=\tau_\mathcal{U}X$ and $C=T_0$ in that proof,  then we get an isomorphism $\text{Hom}_\mathcal{H}(\tilde{H}(T_0),\tilde{H}(X))\cong\text{Hom}_\mathcal{D}(T_0,\tilde{H}(X))$ which is functorial on $X$. Moreover we have a triangle $\tau_{\mathcal{U}[1]}X\longrightarrow\tau_\mathcal{U} X\longrightarrow\tilde{H}(X)\stackrel{+}{\longrightarrow}$ since $\tau_{\mathcal{U}[1]}\tau_\mathcal{U}X\cong\tau_{\mathcal{U}[1]}X$ and $\tilde{H}(X)\cong\tau^{\mathcal{U}^\perp[1]}\tau_\mathcal{U}X$. This gives another isomorphism $\text{Hom}_\mathcal{D}(T_0,\tau_\mathcal{U}X)\cong\text{Hom}_\mathcal{D}(T_0,\tilde{H}(X))$ because $\text{Hom}_\mathcal{D}(T_0,?)$ vanishes on $\mathcal{U}[1 ]$. As a result, we get an isomorphism $\text{Hom}_\mathcal{H}(\tilde{H}(T_0),\tilde{H}(X))\cong\text{Hom}_\mathcal{D}(T_0,\tau_\mathcal{U}X)$ which is functorial on $X$. 
 Using finally the adjoint pair $(\iota_\mathcal{U}:\mathcal{U}\hookrightarrow\mathcal{D},\tau_\mathcal{U}:\mathcal{D}\longrightarrow\mathcal{U})$ and the fact that $T_0\in\mathcal{U}$ we  get the desired functorial isomorphism $$\text{Hom}_\mathcal{H}(\tilde{H}(T_0),\tilde{H}(X))\cong\text{Hom}_\mathcal{D}(T_0,\tau_\mathcal{U}X)=\text{Hom}_\mathcal{U}(T_0,\tau_\mathcal{U}X)\cong\text{Hom}_\mathcal{D}(T_0,X).$$

We now prove that $\tilde{H}:\mathcal{D}\longrightarrow\mathcal{H}$ preserves products. Let $(X_i)_{i\in I}$ be a family of objects of $\mathcal{D}$ and consider the canonical morphism $\psi :\tilde{H}(\prod_{i\in I}X_i)\longrightarrow\prod^*_{i\in I}\tilde{H}(X_i)$, where $\prod^*$ stands for the product in $\mathcal{H}$. Due to the fact that $\tilde{H}(T_0)$ is a projective generator of $\mathcal{H}$, in order to prove that $\psi$ is an isomorphism it is enough to prove that the induced map 

\begin{center}
$\psi_*:\text{Hom}_\mathcal{H}(\tilde{H}(T_0),\tilde{H}(\prod_{i\in I}X_i))\longrightarrow\text{Hom}_\mathcal{H}(\tilde{H}(T_0),\prod^*_{i\in I}\tilde{H}(X_i))\cong\prod_{i\in I}\text{Hom}_\mathcal{H}(\tilde{H}(T_0),\tilde{H}(X_i))$
\end{center}
is an isomorphism. But this is a direct consequence of the previous paragraph and the fact that the functor $\text{Hom}_\mathcal{D}(T_0,?):\mathcal{D}\longrightarrow Ab$ preserves products.

$5)\Longrightarrow 1)$ In the rest of the proof we assume that $\mathcal{D}$ satisfies Brown representability theorem for the dual (BRT* in the sequel).  Fix now a projective generator $P$ of $\mathcal{H}$ and consider the composition functor $$\mathcal{D}\stackrel{\tilde{H}}{\longrightarrow}\mathcal{H}\stackrel{\text{Hom}_\mathcal{H}(P,?)}{\longrightarrow}Ab.$$ This functor preserves products and takes triangles to long exact sequences. By  BRT*, there exists an object $T\in\mathcal{D}$ such that $\text{Hom}_\mathcal{D}(T,?)$ is naturally isomorphic to the last composition functor. The rest of the proof is devoted to checking that $T$ is a partial silting object which generates $\tau$.

Let us take $V\in\mathcal{U}^\perp$ arbitrary, so that $\tilde{H}(V)=0$. It follows that $\text{Hom}_\mathcal{D}(T,V)\cong\text{Hom}_\mathcal{H}(P,\tilde{H}(V))=0$, so that $T\in {} ^\perp (\mathcal{U}^\perp )=\mathcal{U}$. Suppose next that $Y\in T^{\perp_{\leq 0}}$. It follows that $\text{Hom}_\mathcal{H}(P,\tilde{H}^j(Y))\cong\text{Hom}_\mathcal{D}(T,Y[j])=0$, and so $\tilde{H}^j(Y)=0$,  for all $j\leq 0$. By Lemma \ref{lem.description of aisle by homology} and the left nondegeneracy of $\tau$, we get that $\tau_\mathcal{U}Y=0$, so that $Y\in\mathcal{U}^\perp$. It follows that $T$ generates $\tau$. Finally, if $U\in\mathcal{U}$, we have $\text{Hom}_\mathcal{D}(T,U[1])\cong\text{Hom}_\mathcal{H}(P,\tilde{H}^1(U))=0$, so that $T$ is a partial silting object. 
\end{proof}

\begin{ques}
Is the local smallness of $\mathcal{H}$ superfluous in assertion 4 of the last theorem? In other words, suppose that $\tau$ is generated by a set and also by its coheart. Is the heart of $\tau$ necessarily a locally small (abelian)  category?
\end{ques}

By \cite[Theorem 7.2]{Porta}, we know that if $\mathcal{D}$ is well-generated in the sense of Neeman (see \cite[Definition 8.1.6 and Remark 8.1.7]{N} and  \cite{Kr}) and it is algebraic, then there exists a dg category $\mathcal{A}$ and a set $\mathcal{S}$ of objects in $\mathcal{D}(\mathcal{A})$ such that $\mathcal{D}\cong\text{Loc}_{\mathcal{D}(\mathcal{A})}(\mathcal{S})^\perp$. In particular $\mathcal{D}$ has products in that case and,  by \cite[Proposition 8.4.2]{N}, we also know that $\mathcal{D}$ satisfies Brown representability theorem. Note that the derived category of a Grothendieck category is an example of a well-generated algebraic triangulated category. 

\begin{cor} \label{cor.injectivecogenerator-implies-cosilting}
Let  $\tau$ be a t-structure in $\mathcal{D}$. The following assertions hold:

\begin{enumerate}
\item If $\mathcal{D}$ satisfies Brown representability theorem for the dual,  and $\tau$ is  left nondegenerate co-smashing and its heart has a projective generator,  then $\tau =({}^\perp(\mathcal{T}^{\perp_{\leq 0}}), \mathcal{T}^{\perp_{<0}})$  for some partial silting set $\mathcal{T}$ in $\mathcal{D}$. 
\item If $\mathcal{D}$ has products and satisfies Brown representability theorem (e.g. if $\mathcal{D}$ is well-generated algebraic), then the following statements are equivalent:

\begin{enumerate}
\item $\tau$ is right nondegenerate, its heart $\mathcal{H}$ has an injective cogenerator and the cohomological functor $\tilde{H}:\mathcal{D}\longrightarrow\mathcal{H}$ preserves coproducts.
\item $\tau =({}^{\perp_{<0}}\mathcal{Q},({}^{\perp_{\leq 0}}\mathcal{Q})^\perp )$, for some partial cosilting set $\mathcal{Q}$ in $\mathcal{D}$ (see Remark \ref{rem.cosilting-cotilting}).
\end{enumerate}
\end{enumerate}
\end{cor}
\begin{proof}
1) Note that if $\tau$ is co-smashing, then  $\mathcal{H}$ is closed under taking products in $\mathcal{D}$, so that products in $\mathcal{H}$ are calculated as in $\mathcal{D}$. On the other hand, if $(M_i)_{i\in I}$ is a family of objects of $\mathcal{D}$, then the associated truncation triangle with respect to $\tau=(\mathcal{U},\mathcal{U}^\perp [1])$ is $$\prod_{i\in I}\tau_\mathcal{U}M_i\longrightarrow\prod_{i\in I}M_i\longrightarrow\prod_{i\in I}\tau^{\mathcal{U}^\perp}M_i\stackrel{+}{\longrightarrow}. $$ Then both truncation functors $\tau_\mathcal{U}:\mathcal{D}\longrightarrow\mathcal{U}$ and $\tau^{\mathcal{U}^\perp }:\mathcal{D}\longrightarrow\mathcal{U}^\perp$ preserve products, and the same can be said about the `shifted' truncations functors $\tau_{\mathcal{U}[1]}$ and $\tau^{\mathcal{U}^\perp [1]}$. In particular, the cohomological functor $\tilde{H}=\tau^{\mathcal{U}^\perp [1]}\circ\tau_{\mathcal{U}}:\mathcal{D}\longrightarrow\mathcal{H}$ preserves products, and assertion 1 follows from Theorem \ref{teor.presilting t-structures}.

\vspace*{0.3cm}

2) The equivalence of assertions 2.a and 2.b is the dual version of the equivalence of assertions 1 and 5 in Theorem \ref{teor.presilting t-structures}. 
\end{proof}

\begin{cor} \label{cor.self-smallness versus compactness}
Let $\mathcal{T}$ be a partial silting set in $\mathcal{D}$, let $\tau$ be the associated t-structure and let $\tilde{H}:\mathcal{D}\longrightarrow\mathcal{H}$ be the induced cohomological functor. The following assertions are equivalent:

\begin{enumerate}
\item $\tilde{H}(\mathcal{T})$ is a set of compact projective generators of $\mathcal{H}$;
\item $\mathcal{T}$ is a self-small set in $\mathcal{D}$. 
\end{enumerate}
\end{cor}
\begin{proof}
By the proof of implication $1)\Longrightarrow 3)$ in the last theorem, we know that $\tilde{H}(\mathcal{T})$ is a set of projective generators of $\mathcal{H}$.

$1)\Longrightarrow 2)$ By Lemma \ref{lem.coheart projective in heart}, we know that all objects of $\mathcal{T}$ are compact in $\mathcal{U}={}^{\perp}(\mathcal{T}^{\leq 0})$. Since $\mathcal{T}\subseteq\mathcal{U}$ the self-smallness of $\mathcal{T}$ is clear. 

$2)\Longrightarrow 1)$ We will prove that each $T\in\mathcal{T}$ is compact in $\mathcal{U}$, which, by Lemma \ref{lem.coheart projective in heart}, will end the proof. Let $(U_i)_{i\in I}$ be a family of objects in $\mathcal{U}$ and fix a $\text{Sum}(\mathcal{T})$-precover $f_i:T_i\longrightarrow U_i$ and complete to a corresponding triangle $T_i\stackrel{f_i}{\longrightarrow}U_i\longrightarrow Z_i\stackrel{+}{\longrightarrow}$, for each $i\in I$. Then  $f=\coprod f_i:\coprod_{i\in I}T_i\longrightarrow\coprod_{i\in I}U_i$ is also a $\text{Sum}(\mathcal{T})$-precover (see Lemma \ref{lem.SumT-precovers}). By this same lemma, we have that $\text{Hom}_\mathcal{D}(T,?)$ vanishes on each $Z_i$ and on $\coprod_{i\in I}Z_i$, for all $T\in\mathcal{T}$. We then have the following commutative square, where the horizontal arrows are epimorphisms:

$$
\xymatrix  {
 \coprod_{i\in I} \text{Hom}_\mathcal{D}(T,T_i)\ar[r] \ar[d]^{\simeq} &\coprod_{i\in I} \text{Hom}_\mathcal{D}(T,U_i) \ar[d]\\
  \text{Hom}_\mathcal{D}(T,\coprod_{i\in I}T_i)\ar[r] &
\text{Hom}_\mathcal{D}(T,\coprod_{i\in I}U_i)}
$$

for every $T\in\mathcal{T}$. Moreover, the left vertical arrow is an isomorphism because  $\mathcal{T}$ is self-small. It follows that the right vertical arrow is an epimorphism. But it is always a monomorphism. Therefore $T$ is compact in $\mathcal{U}$, as desired.
\end{proof}

The following definition is very helpful.

\begin{opr} \label{def.equivalent partial silting}
Two  strongly nonpositive sets $\mathcal{T}$ and $\mathcal{T}'$ in $\mathcal{D}$ will be called \emph{equivalent} when $\text{Add}(\mathcal{T})=\text{Add}(\mathcal{T}')$. In particular, two partial silting objects $T$ and $T'$ will be equivalent when $\text{Add}(T)=\text{Add}(T')$. 
\end{opr}

\begin{cor} \label{cor.bijection t-structures}
Let $\mathcal{D}$ be a  triangulated category with coproducts. The assignment $\mathcal{T}\rightsquigarrow\tau_\mathcal{T}=({}^\perp (\mathcal{T}^{\perp_{\leq 0}}),\mathcal{T}^{\perp_{<0}})$ gives a one-to-one correspondence between equivalence classes of  partial silting sets (with just one object) and t-structures in $\mathcal{D}$ generated by their co-heart whose heart has a  generator. This restricts to:

\begin{enumerate}
\item A bijection between equivalence classes of  silting objects and (right) nondegenerate t-structures  in $\mathcal{D}$ generated by their co-heart whose heart has a generator.  
\item A bijection between equivalence classes of self-small (resp. classical) partial silting sets  and  t-structures (resp. smashing t-structures) in $\mathcal{D}$ generated by their co-heart whose heart is the module category over a small $K$-category.
\item A bijection between equivalence classes of self-small (resp. classical) partial silting  objects and  t-structures (resp. smashing t-structures) in $\mathcal{D}$ generated by their co-heart whose heart is the module category over an ordinary algebra.
\item A bijection  between equivalence classes of self-small (resp. classical) silting  objects and (right) nondegenerate  t-structures (resp. (right) nondegenerate smashing t-structures) in $\mathcal{D}$ generated by their co-heart whose heart is the module category over an ordinary algebra.  
\end{enumerate}
\end{cor}
\begin{proof}
The general bijection is a consequence of Theorem \ref{teor.presilting t-structures} and Lemma \ref{lem.presilting set}. Note that a set $\mathcal{T}$ is  partial silting if and only if $\hat{T}=\coprod_{T\in\mathcal{T}}T$ is a partial silting object. Then the `in bracket' comment is clear. 

Let us now check that this bijection restricts to the indicated bijections:

\vspace*{0.3cm}

1) Given a partial silting set $\mathcal{T}$ of $\mathcal{D}$ and putting $\hat{T}$ as above, we have a chain of double implications 
\begin{center}
$\mathcal{T}$ is a silting set $\Longleftrightarrow$ $\mathcal{T}$ (or $\{\hat{T}\}$) generates $\mathcal{D}$ $\Longleftrightarrow$ $\bigcap_{n\in\mathbb{Z}}\mathcal{T}^\perp [n]=0$ (equivalently $\bigcap_{n\in\mathbb{Z}}\hat{T}^\perp [n]=0$) $\Longleftrightarrow$ $\bigcap_{n\in\mathbb{Z}}\mathcal{U}_{\mathcal{T}}^\perp [n]=0$ $\Longleftrightarrow$ $\tau_\mathcal{T}=(\mathcal{U}_\mathcal{T},\mathcal{U}_\mathcal{T}^\perp [1])$ is right nondegenerate. 
\end{center}
The left nondegeneracy of $\tau_\mathcal{T}$ follows from Lemma \ref{lem.t-structures generated by co-heart}. 

\vspace*{0.3cm}

2) The general part of this bijection, concerning self-small partial silting sets,  follows directly from Corollary \ref{cor.self-smallness versus compactness} and Proposition \ref{prop.Gabriel-Mitchell}.  If in this bijection the set $\mathcal{T}$ consists of compact objects, then $\tau_\mathcal{T}$ is smashing. Conversely, suppose that $\tau$ is a smashing t-structure.  We select then a set $\mathcal{P}$ of compact projective generators of $\mathcal{H}$. Since also $\tilde{H}(\mathcal{C})$ is a class of projective generators of $\mathcal{H}$, each $P\in\mathcal{P}$ is a direct summand of an object  $\tilde{H}(C)$, with $C\in\mathcal{C}$. It follows from Lemma \ref{lem.coheart projective in heart} that $P\cong\tilde{H}(C_P)$, for some $C_P\in\mathcal{C}\cap\mathcal{D}^c$. Then $\mathcal{T}=\{C_P\text{: }P\in\mathcal{P}\}$ is a classical partial silting set. Moreover, we then have that $\tilde{H}(\mathcal{C})=\text{Add}(\mathcal{P})=\text{Add}(\tilde{H}(\mathcal{T}))$,  which implies that $\mathcal{C}=\text{Add}(\mathcal{T})$. Therefore we have $\tau =\tau_\mathcal{T}$. 

\vspace*{0.3cm}

3) The bijection here is an obvious restriction of the bijection in 2.

\vspace*{0.3cm}

4) This bijection follows from the bijection in  1 and from Corollary \ref{cor.self-smallness versus compactness}, for the self-small case, and that this bijection restricts to the one for classical (=compact) silting objects follows as the bijection in 2) or 3). 
\end{proof}

In the particular case when $\mathcal{D}$ satisfies BRT*, we can use assertion 5 of Theorem \ref{teor.presilting t-structures} to obtain the following bijection, whose proof goes along the lines of the previous corollary and is left to the reader. 

\begin{cor} \label{cor.bijection for comp.generated}
Suppose that $\mathcal{D}$ satisfies Brown representability theorem for the dual (e.g. when $\mathcal{D}$ is compactly generated). The assignment $\mathcal{T}\rightsquigarrow\tau_\mathcal{T}=({}^\perp (\mathcal{T}^{\leq 0}),\mathcal{T}^{\perp_{>0}})$ gives a bijection between equivalence classes of partial silting (resp. silting) sets (with just one object) and left (resp. left and right) nondegenerate t-structures in $\mathcal{D}$ whose heart has a projective generator and whose associated cohomological functor preserves products. This bijection restricts to:

\begin{enumerate}
\item A bijection between equivalence classes of self-small (resp. classical) partial silting sets and  left nondegenerate (resp. left nondegenerate smashing) t-structures  in $\mathcal{D}$ whose heart is the module category over a small $K$-category and whose associated cohomological functor preserves products. 
\item A bijection between equivalence classes of self-small (resp. classical) partial silting objects and  left nondegenerate (resp. left nondegenerate  smashing) t-structures  in $\mathcal{D}$ whose heart is the module category over an ordinary algebra and whose associated cohomological functor preserves products. 
\end{enumerate}
\end{cor}

\begin{opr} \label{def.pure-injective}
Let $\mathcal{D}$ have arbitrary coproducts and products. An object $Y$ of $\mathcal{D}$ will be called \emph{pure-injective} when, given any set $I$ and the canonical map $\lambda :Y^{(I)}\longrightarrow Y^I$, the transpose map $\lambda^*:\text{Hom}_\mathcal{D}(Y^I,Y)\longrightarrow\text{Hom}_\mathcal{D}(Y^{(I)},Y)$ is surjective.
\end{opr}

Note that if $\mathcal{D}$ is compactly generated, then $\lambda :Y^{(I)}\longrightarrow Y^I$ is a pure monomorphism in the terminology of  \cite{Kr2}. In particular, by \cite[Theorem 1.8]{Kr2} our notion of pure-injectivity agrees in that case with the classical one. 
Recall also that if $T$ is a compact object of $\mathcal{D}$ and $E$ is the minimal injective cogenerator of $\text{Mod}-K$, then $\text{Hom}_K(\text{Hom}_\mathcal{D}(T,?),E):\mathcal{D}\longrightarrow \text{Mod}-K$ is a contravariant cohomological functor which takes coproducts to products. When $\mathcal{D}$ satisfies Brown representability theorem, this functor is represented by an object $D(T)$, usually called the \emph{Brown-Comenetz dual}  of $T$, uniquely determined up to isomorphism. 

\begin{ex}
Let $\mathcal{D}$ satisfy Brown representability theorem and let $\mathcal{T}$ be a set of compact objects in $\mathcal{D}$. Then $D(\mathcal{T}):=\prod_{T\in\mathcal{T}}D(T)$ is a pure-injective object of $\mathcal{D}$. 
\end{ex}
\begin{proof}
Put $Y:=D(\mathcal{T})$. The map $\lambda^*:\text{Hom}_\mathcal{D}(Y^I,Y)\longrightarrow\text{Hom}_\mathcal{D}(Y^{(I)},Y)$ is surjective if and only if $\lambda^*:\text{Hom}_\mathcal{D}(Y^I,D(T))\longrightarrow\text{Hom}_\mathcal{D}(Y^{(I)},D(T))$ is surjective, for each $T\in\mathcal{T}$. By definition of $D(T)$, this is equivalent to saying that $\lambda_*:\text{Hom}_\mathcal{D}(T,Y^{(I)})\longrightarrow\text{Hom}_\mathcal{D}(T,Y^I)$ is injective, for all $T\in\mathcal{T}$. This is clear due to the compactness of all $T\in\mathcal{T}$. 
\end{proof}

For our next result, we warn the reader that the dual notion of equivalence of (partial) silting objects, that of equivalence of (partial) cosilting objects, is defined by the fact that two (partial) cosilting objects $Q$ and $Q'$ are equivalent exactly when $\text{Prod}(Q)=\text{Prod}(Q')$.

\begin{prop} \label{prop.bijection smashing-cosmashing}
Suppose that both $\mathcal{D}$ and $\mathcal{D}^{op}$ satisfy Brown representability theorem  and consider the classes $\mathcal{S}_i$ ($i=1,\dots,4$) whose elements are the following:

\begin{enumerate}
\item The equivalence classes of classical silting sets (resp. objects) in $\mathcal{D}$;
\item The smashing and co-smashing nondegenerate t-structures in $\mathcal{D}$ whose heart is a module category over some small $K$-category (resp.  ordinary algebra);
\item The smashing and co-smashing nondegenerate t-structures in $\mathcal{D}$ whose heart is a Grothendieck category;
\item  The equivalence classes of pure-injective cosilting objects $Q$ in $\mathcal{D}$ such that ${}^{\perp_{<0}}Q$ is closed under taking products in $\mathcal{D}$.
\end{enumerate}
There are bijections and injections $$\mathcal{S}_1\stackrel{\sim}{\longleftrightarrow}\mathcal{S}_2\rightarrowtail\mathcal{S}_3\rightarrowtail\mathcal{S}_4.$$ The composed map $\mathcal{S}_1\longrightarrow\mathcal{S}_4$ takes $\mathcal{T}$ to the equivalence class of  $D(\mathcal{T}):=\prod_{T\in\mathcal{T}}D(T)$, where $D(T)$ is the Brown-Comenetz dual of $T$. Moreover, 
when $\mathcal{D}$ is compactly generated, the map $\mathcal{S}_3\longrightarrow\mathcal{S}_4$ is bijective.  
\end{prop}
\begin{proof}
The bijection of Corollary \ref{cor.bijection for comp.generated}(1) clearly restricts to a bijection between $\mathcal{S}_1$ and the class of nondegenerate smashing t-structures $\tau$ in $\mathcal{D}$ whose heart $\mathcal{H}=\mathcal{H}_\tau$ is a module category over some small $K$-category and whose associated cohomological functor $\tilde{H}:\mathcal{D}\longrightarrow\mathcal{H}$ preserves products. We will prove that this class of t-structures is precisely $\mathcal{S}_2$, which will give the bijection $\mathcal{S}_1\stackrel{\sim}{\longleftrightarrow}\mathcal{S}_2$. When $\tau =(\mathcal{U},\mathcal{U}^\perp [1])$ is smashing and co-smashing, both $\mathcal{U}$ and $\mathcal{U}^\perp [1]$ are closed under coproducts and products, which easily implies that the inclusion $\mathcal{H}\hookrightarrow\mathcal{D}$ preserves coproducts and products. This in turn implies that $\tilde{H}$ preserves coproducts and products. Conversely, suppose that $\tau$ is nondegenerate smashing and that $\tilde{H}$ preserves products. If $(U_i)_{i\in I}$ is a family of objects of $\mathcal{U}$, then $$\tilde{H}^k(\prod U_i)=\tilde{H}(\prod U_i[k])\cong\prod^*\tilde{H}(U_i[k])\cong\prod^*\tilde{H}^k(U_i)=0,$$ for all $k>0$,  where $\prod^*$ denotes the product in $\mathcal{H}$. By Lemma \ref{lem.description of aisle by homology} and the right nondegeneracy of $\tau$, we get that $\tau^{\mathcal{U}^\perp}(\prod U_i)=0$ and so $\prod U_i\in\mathcal{U}$. That is, $\tau$ is co-smashing.
 
Let us assume now that $\tau$ is a t-structure as in 3. As proved in the previous paragraph,  such a t-structure satisfies condition 2.a of Corollary \ref{cor.injectivecogenerator-implies-cosilting}. Therefore, we have a partial cosilting set $\mathcal{Q}$, uniquely determined up to equivalence, such that $\tau =({}^{\perp_{<0}}\mathcal{Q}, ({}^{\perp_{\leq 0}}\mathcal{Q})^\perp)$. Putting $Q=\prod_{Q'\in\mathcal{Q}}Q'$, we can assume that $\mathcal{Q}=\{Q\}$.
But the left nondegeneracy of $\tau$ implies that $Q$ cogenerates $\mathcal{D}$, so that $Q$ is actually a cosilting set. In particular, by the dual of Theorem \ref{teor.presilting t-structures}, applied to the cosilting situation,  we have that $\tau =({}^{\perp_{<0}}Q,{}^{\perp_{>0}}Q)$. 

In order to have a (clearly injective) map $\mathcal{S}_3\longrightarrow\mathcal{S}_4$, we just need to check that $Q$ is a pure-injective object. We go to a more general situation and assume that $Q$ is a cosilting object, with  $\tau =({}^{\perp_{<0}}Q,{}^{\perp_{>0}}Q)$ as associated t-structure,  such that ${}^{\perp_{<0}}Q$ is closed under taking products. We claim that  $Q$ is pure-injective in $\mathcal{D}$ if, and only if, $\tilde{H}(Q)$ is also pure-injective in $\mathcal{D}$. Note that, as in the first paragraph of this proof,  $\tilde{H}:\mathcal{D}\longrightarrow\mathcal{H}$ preserves products and coproducts. It is convenient to put $\tau =({}^\perp\mathcal{V}[-1],\mathcal{V})$, where ${}^\perp\mathcal{V}={}^{\perp_{\leq 0}}Q$. 
Then, for each $V\in\mathcal{V}$,  in particular for each $V\in\text{Add}(Q)\cup\text{Prod}(Q)$, we have a triangle $\tilde{H}(V)\longrightarrow V\longrightarrow \tau^{\mathcal{V}[-1]}V\stackrel{+}{\longrightarrow}$.  Bearing in mind that $\text{Hom}_\mathcal{D}(?,Q)$ and $\text{Hom}_\mathcal{D}(\tilde{H}(V),?)$ both vanish on $\mathcal{V}[-k]$, for $k=1,2$,  we get induced isomorphisms $$\text{Hom}_\mathcal{D}(V,Q)\stackrel{\sim}{\longrightarrow}\text{Hom}_\mathcal{D}(\tilde{H}(V),Q)\stackrel{\sim}{\longleftarrow}\text{Hom}_\mathcal{D}(\tilde{H}(V),\tilde{H}(Q)),$$  for all $V\in\mathcal{V}$, which are natural on $V$. Given any set $I$, we then get the following commutative diagram, where all the horizontal arrows are isomorphisms and the vertical arrows are the restriction maps:

$$
\begin{xymatrix}{\text{Hom}_\mathcal{D}(Q^I,Q) \ar[d] \ar[r]^{\cong} & \text{Hom}_\mathcal{D}(\tilde{H}(Q)^I,Q) \ar[d] & \text{Hom}_\mathcal{D}(\tilde{H}(Q)^I,\tilde{H}(Q)) \ar[d] \ar[l]_{\cong}\\
\text{Hom}_\mathcal{D}(Q^{(I)},Q) \ar[r]^{\cong} & \text{Hom}_\mathcal{D}(\tilde{H}(Q)^{(I)},Q) & \text{Hom}_\mathcal{D}(\tilde{H}(Q)^{(I)},\tilde{H}(Q))\ar[l]_{\cong}}
\end{xymatrix}
$$

It then follows that the left vertical arrow is an epimorphism if and only if so is the right vertical arrow. That is, $Q$ is pure-injective in $\mathcal{D}$ if and only if so is $\tilde{H}(Q)$, as it was claimed. If now $\tau=({}^{\perp_{<0}}Q, {}^{\perp_{>0}}Q)$ is as in 3, then, due to the fact that $\tilde{H}(Q)$ is an injective cogenerator of $\mathcal{H}$,  the right vertical arrow of the last diagram is an epimorphism.  Then  $Q$ is a pure-injective object of $\mathcal{D}$.  

Note that the composed map $\mathcal{S}_1\longrightarrow\mathcal{S}_4$ takes the equivalence class of  $\mathcal{T}$ to the equivalence class of $Q$ if and only if $(\mathcal{T}^{\perp_{>0}},\mathcal{T}^{\perp_{<0}})=({}^{\perp_{<0}}Q, {}^{\perp_{>0}}Q)$. But, by construction of the Brown-Comenetz dual, we have that $T^{\perp_{>0}}={}^{\perp_{<0}}D(T)$ and $T^{\perp_{<0}}={}^{\perp_{>0}}D(T)$, for each $T\in\mathcal{T}$. It immediately follows that $(\mathcal{T}^{\perp_{>0}},\mathcal{T}^{\perp_{<0}})=({}^{\perp_{<0}}D(\mathcal{T}), {}^{\perp_{>0}}D(\mathcal{T}))$, where $D(\mathcal{T})=\prod_{T\in\mathcal{T}}D(T)$. Moreover, the equality $\text{Hom}_\mathcal{D}(T,T'[j])=0$ gives the equality $\text{Hom}_\mathcal{D}(T'[j],D(T))=0$, for all $T,T'\in\mathcal{T}$ and $j>0$, which implies that $D(T)\in\mathcal{T}^{\perp_{<0}}={}^{\perp_{>0}}D(T)$. This together with the fact that
  $({}^{\perp_{<0}}D(\mathcal{T}), {}^{\perp_{>0}}D(\mathcal{T}))$ is a t-structure implies that $D(\mathcal{T})$ is a cosilting object clearly equivalent to $Q$. 

It remains to see that, when $\mathcal{D}$ is compactly generated, the map $\mathcal{S}_3\longrightarrow\mathcal{S}_4$ is surjective, for which we just need to prove that the heart $\mathcal{H}$ is a Grothendieck category. 
But note that, as seen above, the fact that $Q$ is  pure-injective in $\mathcal{D}$ implies that $\tilde{H}(Q)$ is also pure-injective in $\mathcal{D}$. Then the result is a consequence of the next lemma. 
\end{proof}

Recall from \cite[D\'efinition 1.2.5]{BBD} that an \emph{admissible abelian subcategory} $\mathcal{A}$ of $\mathcal{D}$ is a full subcategory closed under finite coproducts which is abelian and such that the inclusion functor $\mathcal{A}\hookrightarrow\mathcal{D}$ takes short exact sequences to triangles. The following result  was communicated to us by \v{S}\v{t}ov\'i\v{c}ek \cite{Sto3}.

\begin{lem} \label{lem.Stovicek}
Let $\mathcal{D}$ be a compactly generated triangulated category and let $\mathcal{A}$ be an AB3* admissible abelian subcategory such that the inclusion functor $\mathcal{A}\hookrightarrow\mathcal{D}$ preserves products. If $\mathcal{A}$ admits an injective cogenerator $Y$  which is a pure-injective object of $\mathcal{D}$, then $\mathcal{A}$ is a Grothendieck category. 
\end{lem}
\begin{proof}
It is well-known that if $\mathcal{D}^c$ denotes the (skeletally small) subcategory of compact objects, then   the category of contravariant functors $[(\mathcal{D}^c)^{op},Ab]$ has $\text{Hom}$ sets and  is a Grothendieck category (see \cite{Kr2}). Moreover, the generalized Yoneda functor $\Upsilon :\mathcal{D}\longrightarrow [(\mathcal{D}^c)^{op},Ab]$, $N\rightsquigarrow (?,N):=\text{Hom}_\mathcal{D}(?,N)_{| \mathcal{D}^c}$   induces a category equivalence between the pure-injective objects of $\mathcal{D}$ and the injective objects of $[(\mathcal{D}^c)^{op},Ab]$  (see \cite[Corollary 1.9]{Kr2}). If $Y$ is an injective cogenerator of $\mathcal{A}$ which is pure-injective in $\mathcal{D}$, then $\Upsilon$ gives an equivalence of categories $\text{Inj}(\mathcal{A})=\text{Prod}_\mathcal{D}(Y)\cong\text{Prod}_{[(\mathcal{D}^c)^{op},Ab]}((-,Y))$, where $\text{Inj}(-)$ is the subcategory of injective objects. But if $\mathcal{S}$  denotes the class of functors $S\in [(\mathcal{D}^c)^{op},Ab]$ such that 
$\text{Hom}_{[(\mathcal{D}^c)^{op},Ab]}(S,(?,Y))=0$, then $\mathcal{S}$ is a localizing subcategory of $[(\mathcal{D}^c)^{op},Ab]$, so that the quotient category $[(\mathcal{D}^c)^{op},Ab]/ \mathcal{S}$ exists, is a Grothendieck category and the (exact) quotient functor $q:[(\mathcal{D}^c)^{op},Ab]\longrightarrow [(\mathcal{D}^c)^{op},Ab]/ \mathcal{S}$ induces an equivalence 
$\text{Prod}_{[\mathcal{D}^{op},Ab]}((?,Y))\stackrel{\cong}{\longrightarrow}\text{Inj}([(\mathcal{D}^c)^{op},Ab]/ \mathcal{S})$ (see \cite{Ga} for the terminology and details). It then follows that the composition functor $\mathcal{A}\stackrel{\Upsilon}{\longrightarrow}[(\mathcal{D}^c)^{op},Ab]\stackrel{q}{\longrightarrow}[(\mathcal{D}^c)^{op},Ab]/\mathcal{S}$ is additive and induces  an equivalence of categories $\text{Inj}(\mathcal{A})\cong\text{Inj}([(\mathcal{D}^c)^{op},Ab]/\mathcal{S})$. 

The key observation,  already present in \cite[Section I.2]{AR}, is that if $F:\mathcal{A}\longrightarrow\mathcal{B}$ is an additive functor between abelian categories  with enough injectives that induces an equivalence of categories $\text{Inj}(\mathcal{A})\cong\text{Inj}(\mathcal{B})$, then $\mathcal{A}$ and $\mathcal{B}$ are equivalent  (see \cite[Lemma 2.9]{AMV2} for a particular case), which will allow us to conclude, by the last paragraph,  that $\mathcal{A}$ is  a Grothendieck category. Indeed one obviously gets an induced equivalence still denoted the same $F:\text{Mor}(\text{Inj}(A))\stackrel{\cong}{\longrightarrow}\text{Mor}(\text{Inj}(B))$, where 
 $\text{Mor}(\text{Inj}(?))$ denotes the category of morphisms in $\text{Inj}(\mathcal{?})$, for $?\in\{ \mathcal{A}, \mathcal{B}\}$.  But taking kernels gives an equivalence $\text{Mor}(\text{Inj}(\mathcal{A}))/\mathcal{I}_A\stackrel{\cong}{\longrightarrow} \mathcal{A}$,  where $\mathcal{I}_A$ is the ideal of $\text{Mor}(\text{Inj}(\mathcal{A}))$ formed by the injectively trivial morphisms, in the terminology of \cite{AR}. That is, 
 if $f:Y_1\longrightarrow Y_2$ and $f':Y'_1\longrightarrow Y'_2$ are in $\text{Mor}(\text{Inj}(\mathcal{A}))$, then $\mathcal{I}_A(f,f')$ consists of the  $(\alpha_1:Y_1\longrightarrow Y'_1,\alpha_2:Y_2\longrightarrow Y'_2)$ in $\text{Hom}_{\text{Mor}(\text{Inj}(\mathcal{A}))}(f,f')$ such that there exists a morphism $\gamma :Y_2\longrightarrow Y'_1$ in $\text{Inj}(\mathcal{A})$ with $\gamma\circ f=\alpha_1$.  Note that the category $\text{Mor}(\text{Inj}(\mathcal{A}))/\mathcal{I}_A$ is  $\text{Comod}(\text{Inj}(\mathcal{A}))$, with the terminology of \cite[Section I.2]{AR}. 
  One clearly has that $F(\mathcal{I}_A)=\mathcal{I}_B$, so that we get equivalences of categories $$ \mathcal{A}\stackrel{\text{Ker}^{-1}}{\longrightarrow}\text{Mor}(\text{Inj}(\mathcal{A}))/\mathcal{I}_A\stackrel{F}{\longrightarrow}\text{Mor}(\text{Inj}(\mathcal{B}))/\mathcal{I}_B\stackrel{Ker}{\longrightarrow}\mathcal{B}.$$  

\end{proof}

\begin{rem}
Under sufficiently general assumptions on $\mathcal{D}$, as those of Remark \ref{rem.Frobenius exat}, one has that every orthogonal pair in $\mathcal{D}$ generated  by a set gives triangles (see \cite[Proposition 3.3]{SS}). Then, in similarity with the K\"onig-Yang bijections,  one can include co-t-structures in the last bijections. Recall that a \emph{co-t-structure}  (see \cite{P} ) or \emph{weight structure} (see \cite{Bo}) in  $\mathcal{D}$ is a pair of full subcategories $(\mathcal{X},\mathcal{Y})$ such that $\mathcal{Y}[1]\subseteq\mathcal{Y}$  and $(\mathcal{X},\mathcal{Y}[1])$ is an orthogonal pair which gives triangles, i.e., such that for each object $M\in\mathcal{D}$ there is a triangle $X\longrightarrow M\longrightarrow Y[1]\stackrel{+}{\longrightarrow}$, with $X\in\mathcal{X}$ and $Y\in\mathcal{Y}$. The \emph{co-heart} of the co-t-structure is then $\mathcal{C}:=\mathcal{X}\cap\mathcal{Y}$. Concretely, the assignment $\mathcal{T}\rightsquigarrow ({}^\perp (\mathcal{T}^{\perp_{\geq 0}}),\mathcal{T}^{\perp_{ >0}})$ gives a bijection between 
 equivalence classes of partial silting sets and co-t-structures in $\mathcal{D}$ generated by their co-heart and such that this co-heart has an additive generator (i.e. there is an object $T\in\mathcal{C}$ such that $\mathcal{C}=\text{Add}(T)$). Here we say that a co-t-structure $(\mathcal{X},\mathcal{Y})$ is generated by a class $\mathcal{S}$ of objects when $\mathcal{X}^\perp =\mathcal{S}^{\perp_{\geq 0}}$ or, equivalently, when $\mathcal{Y}=\mathcal{S}^{\perp_{>0}}$. The inverse takes $(\mathcal{X},\mathcal{Y})$ to the equivalence class of $\{T\}$, where $T$ is an additive generator of $\mathcal{C}=\mathcal{X}\cap\mathcal{Y}$. Note that the co-heart of the co-t-structure corresponding to the partial silting set $\mathcal{T}$ is precisely the co-heart of the t-structure $({}^\perp (\mathcal{T}^{\perp_{\leq 0}}),\mathcal{T}^{\perp_{ <0}})$ given by the bijection of  Corollary \ref{cor.bijection t-structures}.  This is due to the fact that both bijections assign to the (co-)t-structure the equivalence class of an additive generator of its co-heart. 
  However this co-t-structure and t-structure are not generally adjacent in the sense of \cite{Bo} since the
 inclusion $\mathcal{U}_\mathcal{T}:=({}^\perp (\mathcal{T}^{\perp_{\leq 0}})\subset\mathcal{T}^{\perp_{>0}}$ is strict, except in case $\mathcal{T}$ is silting (see Theorem \ref{teor.presilting t-structures}). 
\end{rem}

One can easily give examples of nonclassical self-small partial silting objects:

\begin{ex}
If $f:A\longrightarrow B$ is a homological epimorphism of ordinary algebras (i.e. the multiplication map $B\otimes_AB\longrightarrow B$ is an isomorphism and $\text{Tor}_i^A(B,B)=0$, for all $i>0$), then the restriction of scalars $f_*:\mathcal{D}(B)\longrightarrow\mathcal{D}(A)$ preserves partial silting (resp.  partial tilting) objects and preserves self-smallness. In particular, $B$ is a self-small partial tilting object of $\mathcal{D}(A)$ which need not be compact. 
\end{ex}
\begin{proof}
The proof is a direct consequence of the fact that $f_*:\mathcal{D}(B)\longrightarrow\mathcal{D}(A)$ is fully faithful and preserves coproducts (see \cite[Theorem 4.4]{GL} and \cite[Lemma 4]{NS}). Taking  the inclusion $\mathbb{Z}\stackrel{f}{\hookrightarrow}\mathbb{Q}$, we see that $B$ need not be compact in $\mathcal{D}(A)$. 
\end{proof}

However, the following seems to be  a more delicate question. As a by-product of our Sections \ref{sect.tilting theory abelian} and \ref{sect.triangulated equivalences},  \cite[Corollary 2.5]{FMS} and our Corollary \ref{cor.self-small versus compact tilting} give partial affirmative answers. 

\begin{ques} \label{ques.self-small silting}
Let $\mathcal{D}$ be a triangulated category with coproducts (even compactly generated). Is any self-small silting object necessarily compact (=classical)? 
\end{ques}

\section{The aisle of a partial silting t-structure} \label{sect.aisle silting t-structure}

The main result of this section, Theorem \ref{teor.weakly equivalent to presilting}, gives a handy criterion to identify those strongly nonpositive sets in a triangulated category with coproducts which are partial silting. In such case, it also gives a precise description of the objects in the aisle of the associated t-structure. We first need a preliminary lemma.

\begin{lem} \label{lem.star-product of subcategories}
Let $\mathcal{\mathcal{D}}$ be a triangulated category and let $\mathcal{E},\mathcal{F}$ be extension-closed subcategories of $\mathcal{D}$ such that $\text{Hom}_\mathcal{D}(E,F[1])=0$, for all $E\in\mathcal{E}$ and $F\in\mathcal{F}$. Then $\mathcal{E}\star\mathcal{F}$ is closed under extensions in $\mathcal{D}$. In particular, if $\mathcal{D}$ has coproducts and $\mathcal{T}$ is a strongly nonpositive class in $\mathcal{D}$, then we have an equality: 

\begin{center}
$\text{thick}_\mathcal{D}(\text{Sum}(\mathcal{T}))=\text{thick}_\mathcal{D}(\text{Add}(\mathcal{T}))=\bigcup (\text{Sum}(\mathcal{T})[r]\star\text{Sum}(\mathcal{T})[r+1])\star \cdots\star\text{Sum}(\mathcal{T})[s])=\bigcup (\text{Add}(\mathcal{T})[r]\star\text{Add}(\mathcal{T})[r+1])\star \cdots\star\text{Add}(\mathcal{T})[s])$, 
\end{center}
where the unions range over all pairs  of integers $(r,s)$ such that  $r\leq s$.
\end{lem}
\begin{proof}
 Let $X,X'\in\mathcal{E}\star\mathcal{F}$ be any
objects and consider triangles in $\mathcal{D}$:

\begin{center}
$E\stackrel{u}{\longrightarrow}
X\stackrel{v}{\longrightarrow}F\stackrel{+}{\longrightarrow}$

$E'\stackrel{u'}{\longrightarrow}
X'\stackrel{v'}{\longrightarrow}F'\stackrel{+}{\longrightarrow}$

$X\stackrel{f}{\longrightarrow}
M\stackrel{g}{\longrightarrow}X'\stackrel{h}{\longrightarrow}X[1]$,
\end{center}
where $E,E'\in\mathcal{E}$ and $F,F'\in\mathcal{F}$. The goal is to
prove that $M\in\mathcal{E}\star\mathcal{F}$. Note that $v\circ
h[-1]\circ u'[-1]\in\text{Hom}_\mathcal{D}(E'[-1],F)=0$. This gives
a (non-unique) morphism $h_E[-1]:E'[-1]\longrightarrow E$ such that
$u\circ h_E[-1]=h[-1]\circ u'[-1]$. Verdier's $3\times 3$ Lemma (see
\cite[Lemma 1.7]{May}) says that we can form a commutative diagram,
with all rows and columns being triangles:

$$
\begin{xymatrix} {E'[-1] \ar[r]^{u'[-1]} \ar[d]^{h_{E}[-1]}&X'[-1]\ar[r] \ar[d]^{h[-1]}&F'[-1] \ar[d]\\
E \ar[r]^{u} \ar[d]&X \ar[r]^{v} \ar[d]&F  \ar[d]\\
\tilde{E} \ar[r]&M \ar[r]&\tilde{F}}
\end{xymatrix}
$$

Clearly $\tilde{E}\in\mathcal{E}$ and $\tilde{F}\in\mathcal{F}$,
which proves that $M\in\mathcal{E}\star\mathcal{F}$. 

The key point for the final equalities is that $\text{Add}(\mathcal{T})[k]$ and $\text{Sum}(\mathcal{T})[k]$ are both closed under taking extensions, whenever $\mathcal{T}$ is a strongly nonpositive class of objects and $k\in\mathbb{Z}$. Then the chain of equalities will follow automatically from the first statement,  once we check that $\bigcup (\text{Sum}(\mathcal{T})[r]\star\text{Sum}(\mathcal{T})[r+1]\star \cdots\star\text{Sum}(\mathcal{T})[s])$ is closed under taking direct summands. Indeed, if $X$ is in  $\text{Sum}(\mathcal{T})[r]\star\text{Sum}(\mathcal{T})[r+1]\star \cdots\star\text{Sum}(\mathcal{T})[s]$, for some integers $r\leq s$, then the same is true for $X^{(\mathbb{N})}$ since coproducts of triangles  in $\mathcal{D}$ are again triangles (see \cite[Proposition 1.2.1 and Remark 1.2.2]{N}). If now $Y$ is a direct summand of $X$ then,  by the proof of \cite[Proposition 1.6.8]{N}, we know that there is a triangle $X^{(\mathbb{N})}\longrightarrow X^{(\mathbb{N})}\longrightarrow Y\stackrel{+}{\longrightarrow}$ in $\mathcal{D}$. It follows that $Y\in\text{Sum}(\mathcal{T})[r]\star\text{Sum}(\mathcal{T})[r+1]\star \cdots\star\text{Sum}(\mathcal{T})[s+1]$.
\end{proof}

Throughout the rest of the  section $\mathcal{D}$ will be a triangulated category with coproducts. 
Apart from the usual equivalence of partial silting sets (see Definition \ref{def.equivalent partial silting}), we will use the following weaker version for strongly nonpositive sets.

\begin{opr} \label{def.weakly equivalent silting sets}
Let $\mathcal{T}$ and $\mathcal{T}'$ be two strongly nonpositive sets of  $\mathcal{D}$. We will say that they are \emph{weakly equivalent} when $\text{thick}_\mathcal{D}(\text{Sum}(\mathcal{T}))=\text{thick}_\mathcal{D}(\text{Sum}(\mathcal{T}'))$.

\end{opr} 

\begin{thm} \label{teor.weakly equivalent to presilting}
Let $\mathcal{D}$ be a triangulated category with coproducts and let  $\mathcal{T}$ be a strongly nonpositive set of objects of $\mathcal{D}$. The following assertions are equivalent:

\begin{enumerate}
\item $\mathcal{T}$ is a partial silting set.
\item $\mathcal{T}$ is weakly equivalent to a partial silting set.
\item There is a t-structure $(\mathcal{V},\mathcal{V}^\perp [1])$ in $\mathcal{D}$ such that

\begin{enumerate}
\item $\mathcal{T}\subset\mathcal{V}$;
\item There is an integer $q$ such that $\text{Hom}_\mathcal{D}(T,?)$ vanishes on $\mathcal{V}[q]$, for all $T\in\mathcal{T}$. 
\end{enumerate}  
\end{enumerate}

In such case, if $\tau_\mathcal{T}=({}^{\perp}(\mathcal{T}^{\perp_{\leq 0}}),\mathcal{T}^{\perp_{<0}})$ is the associated t-structure, then $\mathcal{U}_\mathcal{T}:={}^{\perp}(\mathcal{T}^{\perp_{\leq 0}})$ consists of the objects $X$ in $\mathcal{D}$ which are the Milnor colimit of some sequence $$X_0\stackrel{x_1}{\longrightarrow}X_1\stackrel{x_2}{\longrightarrow}\cdots\stackrel{x_n}{\longrightarrow}X_n\stackrel{x_{n+1}}{\longrightarrow}\cdots$$ such that $X_0\in\text{Sum}(\mathcal{T})$ and $\text{cone}(x_n)\in\text{Sum}(\mathcal{T})[n]$, for each $n>0$.
\end{thm}
\begin{proof}
$1)\Longrightarrow 2)$ is clear.

$2)\Longrightarrow 3)$ Let $\mathcal{S}$ be a partial silting set weakly equivalent to $\mathcal{T}$. We will prove that, up to shift,  the associated t-structure $\tau_\mathcal{S}=(\mathcal{U}_\mathcal{S},\mathcal{U}_\mathcal{S}^{\perp}[1])$ satisfies the requirements. Indeed, put $\hat{T}=\coprod_{T\in\mathcal{T}}T$. Then, by the hypothesis and Lemma \ref{lem.star-product of subcategories},   we have that $\hat{T}\in\text{Add}(\mathcal{S})[p]\star\text{Add}(\mathcal{S})[p+1]\star \cdots\star\text{Add}(\mathcal{S})[p+t]$, for some $p\in\mathbb{Z}$ and $t\in\mathbb{N}$. Replacing $\mathcal{S}$ by $\mathcal{S}[p]$, we can assume without loss of generality that $\text{Sum}(\mathcal{T})\subset\text{Add}(\mathcal{S})\star\text{Add}(\mathcal{S})[1]\star \cdots\star\text{Add}(\mathcal{S})[t]$ (*). Then condition 3.a clearly holds with $\mathcal{V}=\mathcal{U}_\mathcal{S}$. On the other hand, due to the partial silting condition of $\mathcal{S}$, we know that $\text{Hom}_\mathcal{D}(S,?)$ vanishes on $\mathcal{U}_\mathcal{S}[1]$, for all $S\in\mathcal{S}$. But then the inclusion (*) gives that $\text{Hom}_\mathcal{D}(T,?)$ vanishes on $\mathcal{U}_\mathcal{S}[t+1]$, for all $T\in\mathcal{T}$. 

$3)\Longrightarrow 1)$ The proof will have two steps:
\bigskip

\emph{Step a)} We shall prove that if  $$X_0\stackrel{x_1}{\longrightarrow}X_1\stackrel{x_2}{\longrightarrow}\cdots\stackrel{x_n}{\longrightarrow}X_n\stackrel{x_{n+1}}{\longrightarrow}\cdots$$  is a sequence as in the statement and put $X_\infty =\text{Mcolim}X_n$, then  $\text{Hom}_\mathcal{D}(T,X_\infty [1])=0$;
\bigskip

\emph{Step b)} We will show that, for each object $M\in\mathcal{D}$, there is a sequence $(X_n,x_n)$ as in step a) together with a triangle $X_\infty\longrightarrow M\longrightarrow Y\stackrel{+}{\longrightarrow}$ such that $Y\in\mathcal{T}^{\perp_{\leq 0}}$.
\bigskip

After the two steps are fulfilled, the proof of the implication will be finished. Indeed, for the $X_n$'s in the sequence one clearly has that $\coprod_{n\in\mathbb{N}}X_n\in {}^\perp(\mathcal{T}^{\perp_{\leq 0}})$ since each $X_n$ is a finite extension of objects of $\text{Sum}(\mathcal{T})[k]$, with $0\leq k\leq n$, and any object of $\text{Sum}(\mathcal{T})[k]$ is in ${}^\perp (\mathcal{T}^{\perp_{\leq 0}})$. Therefore $X_\infty$ is also in ${}^\perp (\mathcal{T}^{\perp_{\leq 0}})$ since this is a suspended subcategory of $\mathcal{D}$. It follows from this argument and step b)  that the pair $(\mathcal{U}_\mathcal{T},(\mathcal{U}_\mathcal{T}^\perp [1]):=({}^\perp (\mathcal{T}^{\perp_{\leq 0}}),(\mathcal{T}^{\perp_{< 0}})$ is a t-structure in $\mathcal{D}$ and that each object of $\mathcal{U}_\mathcal{T}$
is of the form $X_\infty=\text{Mocolim}X_n$, for some sequence $X_0\stackrel{x_1}{\longrightarrow}X_1\stackrel{x_2}{\longrightarrow}\cdots\stackrel{x_n}{\longrightarrow}X_n\stackrel{x_{n+1}}{\longrightarrow}\cdots$ as above.  Then,  by step a), one also gets that $\text{Hom}_\mathcal{D}(T,?)$ vanishes on $\mathcal{U}_\mathcal{T}[1]$, for each $T\in\mathcal{T}$. Therefore $\mathcal{T}$ is a partial silting set.

Let us then fulfil the mentioned steps.

{\it Step a)}: We will actually prove that the canonical morphism $\varinjlim\text{Hom}_\mathcal{D}(T[k],X_n)\longrightarrow\text{Hom}_\mathcal{D}(T[k],X_\infty)$ is an epimorphism, for each $k\in\mathbb{Z}$, which immediately leads to the desired equality $\text{Hom}_\mathcal{D}(T,X_\infty [1])=0$ since, due to the strong nonpositivity of $\mathcal{T}$, we have that $\text{Hom}_\mathcal{D}(T[-1],X_n)=0$ for each $n\in\mathbb{N}$.

  Let us fix $k$ and consider $q$ as in condition 3.b. Due to the inclusion $\mathcal{T}\subset\mathcal{V}$, we know that  $\text{cone}(x_n)\in\mathcal{V}[q+k+1]$, for all $n>q+k$. 
  We  put $m(k)=q+k$ in the sequel, and we also  put $X'_n=X_n$, for $n\leq m(k)$,  and $X'_n=X_{m(k)}$, for all $n>m(k)$. Note that we get a new sequence $$X'_0\stackrel{x'_1}{\longrightarrow}X'_1\stackrel{x'_2}{\longrightarrow}\cdots\stackrel{x'_n}{\longrightarrow}X'_n\stackrel{x'_{n+1}}{\longrightarrow}\cdots,$$ where $x'_n$ is the identity map, for each $n>m(k)$. In particular, by  \cite[Lemma 1.6.6]{N}, we know that $X'_\infty:=\text{Mcolim}X'_n$ is isomorphic to $X'_{m(k)}$, with the canonical composition map $\mu_{m(k)}:X'_{m(k)}\longrightarrow\coprod_{n\in\mathbb{N}}X'_n\stackrel{p'}{\longrightarrow} X'_\infty$ being an isomorphism.  Moreover, we clearly have a `morphism of sequences'  $(X'_n,x'_n)\longrightarrow (X_n,x_n)$. 

Fix any object $T\in\mathcal{T}$. For each $n\in\mathbb{N}$, we have a triangle $X'_n\longrightarrow X_n\longrightarrow X''_n\stackrel{+}{\longrightarrow}$, where $X''_n\in\mathcal{V}[q+k+1]$ (note that $X''_n=0$ for $n\leq m(k)$). As a consequence, we get that $\coprod_{n\in\mathbb{N}}X_n''\in\mathcal{V}[q+k+1]$  and we put $\coprod_{n\in\mathbb{N}}X_n''=V[q+k+1]$, with $V\in\mathcal{V}$. We then have $\text{Hom}_\mathcal{D}(T[k],\coprod_{n\in\mathbb{N}}X_n''[j])=
\text{Hom}_\mathcal{D}(T[k],V[q+k+1+j])\cong\text{Hom}_\mathcal{D}(T,V[q+1+j])=0$, for all $j\geq -1$, due to the choice of $q$. In particular, the canonical morphism $\text{Hom}_\mathcal{D}(T[k],\coprod_{n\in\mathbb{N}}X_n')\longrightarrow\text{Hom}_\mathcal{D}(T[k],\coprod_{n\in\mathbb{N}}X_n)$ is an isomorphism. By applying the 3x3 lemma (see \cite[Lemma 1.7]{May}), we can form the following commutative diagram whose rows and columns are triangles:

$$\begin{xymatrix}{\coprod X'_n \ar[d] \ar[r]^{1-\sigma'}&\coprod X'_n \ar[d] \ar[r]^{p'}& X'_{\infty} \ar@{-->}[d]^{h} \\
\coprod X_n \ar[d] \ar[r]^{1-\sigma}&\coprod X_n \ar[d] \ar[r]^{p}& X_{\infty}\ar@{-->}[d] \\
\coprod X''_n  \ar[r]&\coprod X''_n  \ar@{-->}[r]&Z }
\end{xymatrix}$$

We then have that $\text{Hom}_\mathcal{D}(T[k],Z[j])=0$, for all $j\geq -1$, which in turn implies that the morphism $h_*:\text{Hom}_\mathcal{D}(T[k],X'_\infty[j])\longrightarrow\text{Hom}_\mathcal{D}(T[k],X_\infty[j])$ is
an isomorphism, for all $j\geq 0$, and an epimorphism for $j=-1$. Denote by $\iota'_{m(k)}:X'_{m(k)}=X_{m(k)}\longrightarrow\coprod_{n\in\mathbb{N}}X'_n$ and  $\iota_{m(k)}:X_{m(k)}\longrightarrow\coprod_{n\in\mathbb{N}}X_n$ the  injections into the respective coproducts and put $\mu'_{m(k)}=p'\circ\iota'_{m(k)}$ and $\mu_{m(k)}=p\circ\iota_{m(k)}$. By the last diagram, we know that $h\circ\mu'_{m(k)}=\mu_{m(k)}$. And, as mentioned above, we know that $\mu'_{m(k)}$ is an isomorphism. It then follows that $(\mu_{m(k)}[j])_*:\text{Hom}_\mathcal{D}(T[k],X_{m(k)}[j])\longrightarrow\text{Hom}_\mathcal{D}(T[k],X_\infty [j])$ is an isomorphism, for all $j\geq 0$, and an epimorphism for $j=-1$. But $(\mu_{m(k)}[j])_*$ factors in the form $$\text{Hom}_\mathcal{D}(T[k],X_{m(k)}[j])\longrightarrow\varinjlim\text{Hom}_\mathcal{D}(T[k],X_n[j])\longrightarrow\text{Hom}_\mathcal{D}(T[k],X_\infty[j]).$$ Then the canonical map $\varinjlim\text{Hom}_\mathcal{D}(T[k],X_n[j])\longrightarrow\text{Hom}_\mathcal{D}(T[k],X_\infty [j])$ is an epimorphism, for all $j\geq -1$.

{\it Step b)}:   The proof is inspired by that of \cite[Theorem 12.2]{KN}. Let $M\in\mathcal{D}$ be any object. We construct a direct system of triangles $(X_n\stackrel{f_n}{\longrightarrow} M\stackrel{g_n}{\longrightarrow} Y_n\stackrel{+}{\longrightarrow})_{n\in\mathbb{N}}$, with the property that $\text{Hom}_\mathcal{D}(T[k],Y_n)=0$, for all $0\leq k\leq n$. For $n=0$, the map $f_0:X_0\longrightarrow M$ is a $\text{Sum}(\mathcal{T})$-precover of $M$ and $g_0$ and $Y_0$ are obtained by choosing a (fixed) completion to a triangle $X_0\stackrel{f_0}{\longrightarrow}M\stackrel{g_0}{\longrightarrow} Y_0\stackrel{+}{\longrightarrow}$. If $n>0$ and we already have defined the direct system up to step $n-1$, then we choose a $\text{Sum}(T)[n]$-precover $p_n:T_{n}[n]\longrightarrow Y_{n-1}$. Note that $p_n$ is also an $\text{Sum}(\coprod_{0\leq k\leq n}T[k])$-precover since $\text{Hom}_\mathcal{D}(T[k],Y_{n-1})=0$, for $0\leq k\leq n-1$. When completing to a triangle $T_{n}[n]\stackrel{p_n}{\longrightarrow}Y_{n-1}\stackrel{y_n}{\longrightarrow}Y_n\stackrel{+}{\longrightarrow}$, from the precovering condition of $p_n$ and the fact that $\text{Hom}_\mathcal{D}(T,T[j])=0$ for $j>0$,  we easily deduce that $\text{Hom}_\mathcal{D}(T[k],Y_n)=0$, for all $0\leq k\leq n$. Now, applying the octahedral axiom, we get the following commutative diagram, which gives the triangle $X_n\stackrel{f_n}{\longrightarrow}M\stackrel{g_n}{\longrightarrow}Y_n\stackrel{+}{\longrightarrow}$ at step $n$, together with the connecting morphisms $x_n:X_{n-1}\longrightarrow X_n$ and $y_n:Y_{n-1}\longrightarrow Y_n$:

$$
\begin{xymatrix}{ X_{n-1} \ar@{=}[d] \ar[r]^{x_n}& X_n \ar[d]^{f_n} \ar[r]& T_n[n] \ar[d]^{p_n} \\
 X_{n-1}  \ar[r]^{f_{n-1}}& M \ar[d]^{g_n} \ar[r]^{g_{n-1}}& Y_{n-1}\ar[d]^{y_n} \\
 &Y_n  \ar@{=}[r]&Y_n }
\end{xymatrix}
$$

 Denote by $\mu_r$ the composition $X_r\stackrel{\iota_r}{\longrightarrow}\coprod_{n\in\mathbb{N}}X_n\stackrel{p}{\longrightarrow}\text{Mcolim}X_n=X_\infty$.  If $\hat{f}:\coprod_{n\in\mathbb{N}}X_n\longrightarrow M$ is the unique morphism such that $\hat{f}\circ\iota_r=f_r$, for all $r\in\mathbb{N}$, then the composition $\coprod_{n\in\mathbb{N}}X_n\stackrel{1-\sigma}{\longrightarrow}\coprod_{n\in\mathbb{N}}X_n\stackrel{\hat{f}}{\longrightarrow}M$ is the zero morphism. 
We then get a morphism $f:X_\infty\longrightarrow M$ such that $f\circ p=\hat{f}$, and hence $f\circ\mu_r=f_r$, for each $r\in\mathbb{N}$,  together with a triangle $X_\infty\stackrel{f}{\longrightarrow}M\stackrel{g}{\longrightarrow} Y\stackrel{+}{\longrightarrow}$. 

It remains to prove that $Y\in\mathcal{T}^{\perp_{\leq 0}}$.  Fix  $k\in\mathbb{N}$.  The map  $(f_r)_*:\text{Hom}_\mathcal{D}(T[k],X_r)\longrightarrow\text{Hom}_\mathcal{D}(T[k],M)$ is an epimorphism for $r>k$ and $T\in\mathcal{T}$, because $\text{Hom}_\mathcal{D}(T[k],Y_r)=0$ for $r>k$. 
Since we have $(f_r)_*=f_*\circ (\mu_r)_*$, we conclude that $f_*$ is an epimorphism, for all $k\in\mathbb{N}$. This implies in particular that $\text{Hom}_\mathcal{D}(T,Y)=0$, because $\text{Hom}_\mathcal{D}(T,X_\infty [1])=0$.

We now prove that  $f_*:\text{Hom}_\mathcal{D}(T[k],X_\infty)\longrightarrow\text{Hom}_\mathcal{D}(T[k],M)$ is a monomorphism (and hence an isomorphism), for all $k\geq 0$ and all $T\in\mathcal{T}$. Take any $\varphi\in\text{Ker}(f_*)$. Since, by step a),  the canonical morphism $\varinjlim\text{Hom}_\mathcal{D}(T[k],X_n)\longrightarrow\text{Hom}_\mathcal{D}(T[k],X_\infty)$ is surjective, we get that $\text{Hom}_\mathcal{D}(T[k],X_\infty)$ is the union of all the images of the maps $(\mu_r)_*:\text{Hom}_\mathcal{D}(T[k],X_r)\longrightarrow\text{Hom}_\mathcal{D}(T[k],X_\infty)$. In particular, we have that 
  $\varphi =(\mu_r)_*(\psi)$, for some $r\in\mathbb{N}$ and some $\psi\in\text{Hom}_\mathcal{D}(T[k],X_r)$. We can assume without loss of generality that $r>k+1$. We then have $$0=f_*(\varphi )=(f_*\circ (\mu_r)_*)(\psi )=f\circ\mu_r\circ\psi =f_r\circ\psi .$$
Using now the triangle $Y_r[-1]\stackrel{w}{\longrightarrow} X_r\stackrel{f_r}{\longrightarrow X}\stackrel{+}{\longrightarrow}$, we conclude  that $\psi$ factors in the form $\psi: T[k]\longrightarrow Y_r[-1]\stackrel{w}{\longrightarrow} X_r$. But $\text{Hom}_\mathcal{D}(T[k],Y_r[-1])\cong\text{Hom}_\mathcal{D}(T[k+1],Y_r)$ is zero, because $r>k+1$. We then get $\psi =0$, and also $\varphi =0$.

The fact that $f_*$ is an isomorphism, for all $k\geq 0$ , and that $\text{Hom}_\mathcal{D}(T,X_\infty [1])=0$ imply that $\text{Hom}_\mathcal{D}(T[k],Y)=0$, for all $k\geq 0$ and all $T\in\mathcal{T}$, so that $Y\in\mathcal{T}^{\perp_{\leq 0}}$ as desired. 
 
\end{proof}

We will frequently use the following two auxiliary results. The first one is a slight improvement of  \cite[Lemma 4.10 (i,ii)]{PV}.

\begin{lem} \label{lem.t-structure class generators}
Let $\mathcal{A}$ be any abelian category such that $\mathcal{D}(\mathcal{A})$ has $\text{Hom}$ sets, and let $\mathcal{G}$ be a class of generators of $\mathcal{A}$. If $k\in\mathbb{Z}$ and  $X\in\mathcal{D}(\mathcal{A})$ are such that $\text{Hom}_{\mathcal{D}(\mathcal{A})}(?,X[k])$ vanishes on $\mathcal{G}$, then $H^k(X)=0$. In particular, 
we have an equality $({}^\perp (\mathcal{G}^{\perp_{\leq 0}}),\mathcal{G}^{\perp_{<0}})=(\mathcal{D}^{\leq 0}(\mathcal{A}),\mathcal{D}^{\geq 0}(\mathcal{A}))$. 
\end{lem}
\begin{proof}
If we assume that $H^k(X)\neq 0$ and $p:Z^k(X)\longrightarrow H^k(X)$ is the epimorphism from $k$-cycles to $k$-homology, then there is a morphism $\alpha :G\longrightarrow Z^k(X)$, for some $G\in\mathcal{G}$, such that $p\circ\alpha\neq 0$. If now $s:X\longrightarrow\hat{X}$ is any quasi-isomorphism, then $p$ factors in the form $Z^k(X)\stackrel{s}{\longrightarrow}Z^k(\hat{X})\stackrel{\hat{p}}{\longrightarrow}H^k(X)$, which implies that the induced chain map $\tilde{\alpha}:G\longrightarrow X[k]$ has the property that $s[k]\circ\tilde{\alpha}$ is a nonzero morphism in $\mathcal{K}(\mathcal{A})$, for all quasi-isomorphisms $s$ with domain $X$. This implies that $\tilde{\alpha}$ is a nonzero morphism in $\mathcal{D}(\mathcal{A})$, which contradicts the fact that $\text{Hom}_{\mathcal{D}(\mathcal{A})}(?,X[k])$ vanishes on $\mathcal{G}$.
\end{proof}

\begin{lem} \label{lem.finite projective dimension}
Let $\mathcal{A}$ be any abelian category such that $\mathcal{D}(\mathcal{A})$ has $\text{Hom}$ sets, let $M$ be an object of $\mathcal{A}$ and let $n$ be a natural number. The following assertions are equivalent:

\begin{enumerate}
\item $\text{Ext}_\mathcal{A}^k(M,N)=0$, for all integers $k>n$ and all objects $N$ of $\mathcal{A}$.
\item The functor $\text{Hom}_{\mathcal{D}(\mathcal{A})}(M,?):\mathcal{D}(\mathcal{A})\longrightarrow Ab$ vanishes on $\mathcal{D}^{<-n}(\mathcal{A})$.

When $\mathcal{A}$ has enough projectives, the above conditions are equivalent to:

\item There exists an exact sequence $0\longrightarrow P^{-n}\longrightarrow \cdots\longrightarrow P^{-1}\longrightarrow P^0\longrightarrow M\longrightarrow 0$, where all the $P^{-k}$ are projective objects of $\mathcal{A}$. 
\end{enumerate}
\end{lem}
\begin{proof}
The equivalence of assertions 1 and 3 when $\mathcal{A}$ has enough projectives is standard, and the implication $2)\Longrightarrow 1)$ is clear. As for the implication $1)\Longrightarrow 2)$, consider a complex $Y^\bullet\in\mathcal{D}^{<-n}(\mathcal{A})$  and suppose that $\text{Hom}_{\mathcal{D}(\mathcal{A})}(M,Y^\bullet )\neq 0$. Then, up to replacement of $Y^\bullet$ by a quasi-isomorphic complex, we can assume that we have a chain map $f:M\longrightarrow Y^\bullet$ which represents a nonzero morphism in $\mathcal{D}(\mathcal{A})$. Then this map factors in the form $f:M\stackrel{\tilde{f}}{\longrightarrow}\sigma_{\geq -n-1}Y^\bullet\stackrel{can}{\longrightarrow}Y^\bullet$, where $\sigma_{\geq -n-1}$ denotes the stupid truncation at $-n-1$. But $\sigma_{ \geq -n-1}Y^\bullet$ has homology concentrated in degree $-n-1$ since $Y^\bullet\in\mathcal{D}^{<-n}(\mathcal{A})$. It follows that $\sigma_{ \geq -n-1}Y^\bullet$ is isomorphic in $\mathcal{D}(\mathcal{A})$ to a stalk complex $N[n+1]$, which implies that $\tilde{f}=0$ in $\mathcal{D}(\mathcal{A})$ since $\text{Ext}_\mathcal{A}^{n+1}(M,N)=0$. Therefore we have $f=0$ in $\mathcal{D}(\mathcal{A})$, which is a contradiction. 
\end{proof}

\begin{opr} \label{def.finite projective dimension}
An object $M$ of the abelian category $\mathcal{A}$ will be said to have \emph{projective dimension $\leq n$}, written $\text{pd}_\mathcal{A}(M)\leq n$,  when it satisfies condition 1 of Lemma \ref{lem.finite projective dimension}. The concept of \emph{injective dimension $\leq n$}, written $id_\mathcal{A}(M)\leq n$, is the dual. The category $\mathcal{A}$ is said to have \emph{global dimension $\leq n$}, written $\text{gldim}(\mathcal{A})\leq n$ when each object has projective (equivalently, injective) dimension $\leq n$. We will say that $\mathcal{A}$ has \emph{finite global dimension} when there is a $n\in\mathbb{N}$ such that $\text{gldim}(\mathcal{A})\leq n$. 
\end{opr}

In the rest of the section, we consider the following situation. 

\begin{setup} \label{setup}
 $\mathcal{A}$ is an abelian category with the property that its derived category $\mathcal{D}(\mathcal{A})$ has $\text{Hom}$ sets and arbitrary coproducts.
\end{setup}

\begin{lem} \label{lem.setup implies AB3}
Let $\mathcal{A}$ be an abelian category as in Setup \ref{setup}. Then $\mathcal{A}$ is AB3 and the restriction of the $0$-homology functor $H^0_{| \mathcal{D}^{\leq 0}(\mathcal{A})}:\mathcal{D}^{\leq 0}(\mathcal{A})\longrightarrow\mathcal{A}$ preserves coproducts. 
\end{lem}
\begin{proof}
We can identify $\mathcal{A}$ with the heart and $H^0$ with the cohomological functor provided by the canonical t-structure. The result is then a particular case of \cite[Proposition 3.2 and Lemma 3.1]{PS1}.
\end{proof}
 
\begin{ex}
Each AB4 abelian category with enough projectives, each Grothendieck category and the dual category of any Grothendieck category are abelian categories as in Setup \ref{setup}.
\end{ex}
\begin{proof}
If $\mathcal{A}$ is AB4, then $\mathcal{D}(\mathcal{A})$ has coproducts and they are calculated `pointwise', i.e., as in $\mathcal{C}(\mathcal{A})$ (see \cite{N2}). In such case, when either $\mathcal{A}$ has enough projectives (see  \cite[Theorem 1]{S}) or $\mathcal{A}$ is a Grothendieck category, we know that $\mathcal{D}(\mathcal{A})$ has $\text{Hom}$ sets.

Suppose finally $\mathcal{A}=\mathcal{G}^{op}$, where $\mathcal{G}$ is a Grothendieck category. We have an induced equivalence of categories $\mathcal{D}(\mathcal{A})\cong\mathcal{D}(\mathcal{G})^{op}$. The result in this case is just a consequence of the fact that $\mathcal{D}(\mathcal{G})$ has products. 
\end{proof}

Theorem \ref{teor.weakly equivalent to presilting} has now the following (nondirect) consequence. 

\begin{prop} \label{prop.aisle of partial tilting object}
Let $\mathcal{A}$ be as in Setup \ref{setup} and let $T\in\mathcal{A}$ be an object satisfying both of the following conditions:
\begin{enumerate}
\item[i)] The coproduct  of $I$ copies of $T$, denoted $T^{(I)}$, is the same  in $\mathcal{A}$ and $\mathcal{D}(\mathcal{A})$, for all sets $I$;
\item[ii)] $\text{Ext}_\mathcal{A}^k(T,T^{(I)})=0$, for all integers $k>0$ and all sets $I$. 
\end{enumerate}

 The following assertions hold:

\begin{enumerate}
\item $\text{thick}_{\mathcal{D}(\mathcal{A})}(\text{Sum(T)})$ consists of the complexes isomorphic in $\mathcal{D}(\mathcal{A})$ to bounded complexes of objects in $\text{Sum}(T)$ (or $\text{Add}(T)$).
\item If $m\leq n$ are integers, then the subcategory $\text{Add}(T)[m]\star\text{Add}(T)[m+1]\star \cdots\star\text{Add}(T)[n]$ consists of those complexes isomorphic in $\mathcal{D}(\mathcal{A})$ to complexes of $\mathcal{K}^{[-n,-m]}(\text{Add}(T))$. 
\item If $T$ is partial silting  and $\tau_T=({}^\perp(T^{\perp_{\leq 0}}),T^{\perp_{<0}})$ is the associated t-structure in $\mathcal{D}(\mathcal{A})$, then ${}^\perp(T^{\perp_{\leq 0}})$ consists of the complexes isomorphic in $\mathcal{D}(\mathcal{A})$ to complexes in $\mathcal{K}^{\leq 0}(\text{Sum}(T))=\mathcal{K}^{\leq 0}(\text{Add}(T))$. 
\end{enumerate}
\end{prop}
\begin{proof}
1) Consider the composition $F:\mathcal{K}^b(\text{Add}(T))\stackrel{\iota}{\longrightarrow}\mathcal{K}(\mathcal{A})\stackrel{q}{\longrightarrow}\mathcal{D}(\mathcal{A})$, where $\iota$ and $q$ are the inclusion and the localization functor, respectively. We will prove that $F$ is fully faithful by slightly modifying the proof of \cite[Proposition 7.3]{NS2}. First, it is clear that $\text{thick}_{\mathcal{K}(\mathcal{A})}(\text{Sum}(T))=\mathcal{K}^b(\text{Add}(T))$. Consider the subcategory $\mathcal{C}$ of $\mathcal{K}^b(\text{Add}(T))$ consisting of the $M$ such that the  map $$\text{Hom}_{\mathcal{K}(\mathcal{A})}(M,T^{(I)}[p])\longrightarrow\text{Hom}_{\mathcal{D}(\mathcal{A})}(F(M),F(T)^{(I)}[p]),$$ induced by $F$, is bijective, for all $p\in\mathbb{Z}$ and all sets $I$. We clearly have that $T^{(J)}\in\mathcal{C}$, for all sets $J$, and that $\mathcal{C}$ is closed under extensions, all shifts and direct summands. That is, $\mathcal{C}$ is a thick subcategory of $\mathcal{K}^b(\text{Add}(T))$ containing $\text{Sum}(T)$. It follows that $\mathcal{C}=\mathcal{K}^b(\text{Add}(T))$.

Fix now $M\in\mathcal{K}^b(\text{Add}(T))$ and consider the subcategory $\mathcal{C}'_M$ of $\mathcal{K}^b(\text{Add}(T))$ consisting of those $N$ such the induced map $$\text{Hom}_{\mathcal{K}(\mathcal{A})}(M,N[p])\longrightarrow\text{Hom}_{\mathcal{D}(\mathcal{A})}(F(M),F(N)[p])$$ is bijective, for all $p\in\mathbb{Z}$. Again, we have that $\mathcal{C}'_M$ is a thick subcategory of $\mathcal{K}^b(\text{Add}(T))$ which, due to the previous paragraph, contains $\text{Sum}(T)$. It follows that $\mathcal{C}'_M=\mathcal{K}^b(\text{Add}(T))$, for each $M\in\mathcal{K}^b(\text{Add}(T))$. Therefore $F$ is a fully faithful functor. 

Due to the fully faithful condition of $F$, we know that $F(\mathcal{K}^b(\text{Add}(T)))$ is a thick subcategory of $\mathcal{D}(\mathcal{A})$ such that $F(\mathcal{K}^b(\text{Add}(T)))=F(\text{thick}_{\mathcal{K}(\mathcal{A})}(\text{Sum}(T)))\subseteq\text{thick}_{\mathcal{D}(\mathcal{A})}(\text{Sum}(T))$. But the reverse inclusion also holds, because $T^{(I)}=F(T^{(I)})$, for each set $I$. 

2) The proof of assertion 1 gives that the functor $F$ induces an equivalence of triangulated categories $\mathcal{K}^b(\text{Add}(T))\stackrel{\sim}{\longrightarrow}\text{thick}_{\mathcal{D}(\mathcal{A})}(\text{Sum}(T))$. Assertion 2 is then a consequence of the fact that, as a subcategory of $\mathcal{K}^b(\text{Add}(T))$, the category $\text{Add}(T)[m]\star\text{Add}(T)[m+1]\star \cdots\star\text{Add}(T)[n]$ consists precisely of the complexes in $\mathcal{K}^{[-n,-m]}(\text{Add}(T))$. 

3) Put $\mathcal{U}_T:={}^\perp (T^{\perp_{\leq 0}})$ in the sequel. By Theorem \ref{teor.weakly equivalent to presilting}, each object $X\in\mathcal{U}_T$ is  the Milnor colimit of a sequence $$X_0\stackrel{x_1}{\longrightarrow}X_1\stackrel{x_2}{\longrightarrow}\cdots\stackrel{x_{n-1}}{\longrightarrow}X_{n-1}\stackrel{x_n}{\longrightarrow}X_n\stackrel{x_{n+1}}{\longrightarrow}\cdots, $$ where $X_0=T_0\in\text{Sum}(T)$ and $\text{cone}(x_n)=T_n[n]$, for some $T_n\in\text{Sum}(T)$. The proof of assertion 1 tells us that each morphism $x_n$ `is' a chain map and that the induced triangle $X_{n-1}\stackrel{x_n}{\longrightarrow}X_n\longrightarrow T_n[n]\stackrel{+}{\longrightarrow}$ may be viewed as a triangle in $\mathcal{K}^b(\mathcal{A})$, for each $n>0$.  This will allow us to construct a complex $T^\bullet:\cdots\longrightarrow T_n\longrightarrow T_{n-1}\longrightarrow \cdots\longrightarrow T_1\longrightarrow T_0\longrightarrow 0\longrightarrow\cdots$ which is isomorphic to $X$ in $\mathcal{D}(\mathcal{A})$. We will construct $T^\bullet$ inductively. Namely, for each $n>0$, we will give a morphism $f_n:T_n\longrightarrow T_{n-1}$ in $\mathcal{A}$ satisfying the following properties:

\begin{enumerate}
\item[a)] The composition of two consecutive maps in the sequence $T_n\stackrel{f_n}{\longrightarrow}T_{n-1}\longrightarrow \cdots\stackrel{f_2}{\longrightarrow} T_1\stackrel{f_1}{\longrightarrow} T_0$ is the zero map;
\item[b)] The complex $$\cdots\longrightarrow 0\longrightarrow T_n\stackrel{f_n}{\longrightarrow}T_{n-1}\longrightarrow \cdots\stackrel{f_2}{\longrightarrow} T_1\stackrel{f_1}{\longrightarrow} T_0\longrightarrow 0\longrightarrow \cdots$$ is isomorphic to $X_n$ in $\mathcal{D}(\mathcal{A})$;
\item[c)] Under the isomorphism of b) (and the ones from the preceding steps), the morphism $x_n:X_{n-1}\longrightarrow X_n$ is identified with the chain map given by the vertical arrows of the following diagram:

$$
\begin{xymatrix}{\cdots \ar[r]&0 \ar[r]&0 \ar[r]&T_{n-1} \ar@{=}[d] \ar[r]^{f_{n-1}}&\cdots \ar[r]^{f_{2}}&T_1 \ar@{=}[d] \ar[r]^{f_{1}}&T_0 \ar@{=}[d] \ar[r]&0 \ar[r]&\cdots \\
 \cdots \ar[r]&0 \ar[r]&T_n \ar[r]^{f_{n}}&T_{n-1} \ar[r]^{f_{n-1}}&\cdots \ar[r]^{f_{2}}&T_1 \ar[r]^{f_{1}}&T_0 \ar[r]&0 \ar[r]&\cdots }
\end{xymatrix}
$$

\end{enumerate}
Once we will have proved this, the complex $$T^\bullet:\cdots\stackrel{f_{n+1}}{\longrightarrow}T_n\stackrel{f_n}{\longrightarrow}\cdots\stackrel{f_2}{\longrightarrow}T_1\stackrel{f_1}{\longrightarrow}T_0\longrightarrow 0\longrightarrow \cdots$$ will be the desired one and the proof will be finished. Indeed, by taking stupid truncations, we have that $\sigma_{\geq -n}T^\bullet $ is the complex of condition b) above, and the canonical morphism $\sigma_{\geq -n+1}T^\bullet\longrightarrow\sigma_{\geq -n}T^\bullet$ is precisely the map of condition c) above. Therefore we will have an isomorphism $T^\bullet\cong\text{Mcolim}(\sigma_{\geq -n}T^\bullet )\cong\text{Mcolim}X_n=X$ in $\mathcal{D}(\mathcal{A})$. 

By definition of the sequence $(X_n,x_n)$, we have a triangle $T_1[0]\longrightarrow X_0=T_0[0]\stackrel{x_1}{\longrightarrow}X_1\stackrel{+}{\longrightarrow}$, and the morphism $T_1[0]\longrightarrow T_0[0]$ is of the form $f_1[0]$, for some morphism $f_1:T_1\longrightarrow T_0$ in $\mathcal{A}$. Suppose now that $n>1$. The induction hypothesis gives a complex $T_{n-1}^\bullet:\cdots\longrightarrow 0\longrightarrow T_{n-1}\stackrel{f_{n-1}}{\longrightarrow}\cdots\stackrel{f_1}{\longrightarrow}T_0\longrightarrow 0\longrightarrow \cdots$, which is isomorphic to $X_{n-1}$ in $\mathcal{D}(\mathcal{A})$. Since we have a triangle $X_{n-1}\stackrel{x_n}{\longrightarrow}X_n\longrightarrow T_n[n]\stackrel{+}{\longrightarrow}$, we also get a triangle $T_n[n-1]\stackrel{\alpha_n}{\longrightarrow}T_{n-1}^\bullet\stackrel{\beta_n}{\longrightarrow}X_n\stackrel{+}{\longrightarrow}$, where $\beta_n$ is identified with $x_n$ using the isomorphism $T^\bullet_{n-1}\cong X_{n-1}$. But assertion 1 tells us that $\alpha_n$ `is' a chain map. It is then given by the vertical arrows of the following commutative diagram, for some morphism $f_n:T_n\longrightarrow T_{n-1}$ in $\mathcal{A}$ such that $f_{n-1}\circ f_n=0$:

$$
\begin{xymatrix}{\cdots \ar[r]&0 \ar[r]&T_{n} \ar[d]^{f_{n}} \ar[r]& 0\ar[r] \ar[d] &\cdots \ar[r]&0 \ar[d] \ar[r]&0 \ar[d] \ar[r]& 0  \ar[r]&\cdots \\
 \cdots \ar[r]&0 \ar[r]&T_{n-1} \ar[r]^{f_{n-1}}&T_{n-2} \ar[r]^{f_{n-2}}&\cdots \ar[r]^{f_{2}}&T_1 \ar[r]^{f_{1}}&T_0 \ar[r]&0 \ar[r]&\cdots }
\end{xymatrix}
$$

Note that the cone of the last mentioned chain map is isomorphic to the complex $T^\bullet_n:\cdots\longrightarrow 0\longrightarrow T_n\stackrel{f_n}{\longrightarrow}T_{n-1}\stackrel{f_{n-1}}{\longrightarrow}\cdots\stackrel{f_2}{\longrightarrow}T_1\stackrel{f_1}{\longrightarrow}T_0\longrightarrow 0\longrightarrow\cdots$. Then all needed conditions a)-c) are satisfied. 
\end{proof}

\section{A tilting theory for objects in  AB3 abelian categories} \label{sect.tilting theory abelian}

The goal of this section is to show that the results in the previous section allow to extend the well-established theory of infinitely generated $n$-tilting modules, for $n\in\mathbb{N}$, to any abelian category as in Setup \ref{setup} (see \cite{C} for a similar attempt, when $n=1$ and $\mathcal{A}$ is a Grothendieck category). 

\begin{opr} \label{def.(partial) tilting object}
An object $T$ of $\mathcal{A}$ will be called \emph{partial $n$-tilting} when the following conditions hold:

\begin{enumerate}
\item[T0] The coproduct of $I$ copies of $T$, denoted  by $T^{(I)}$, is the same in $\mathcal{A}$ and $\mathcal{D}(\mathcal{A})$, for each set $I$;
\item[T1] $\text{Ext}_\mathcal{A}^k(T,T^{(I)})=0$, for all integers $k>0$ and all sets $I$;
\item[T2] The projective dimension of $T$ is $\leq n$;
\end{enumerate}
We will say that $T$ is a \emph{$n$-tilting object} if it is partial $n$-tilting and, in addition, the following condition holds:

\begin{enumerate}
\item[T3] There is a generating class $\mathcal{G}$   of $\mathcal{A}$ such that, for each $G\in\mathcal{G}$, there is an exact sequence $0\longrightarrow G\longrightarrow T^0\longrightarrow T^1\longrightarrow \cdots\longrightarrow T^n\longrightarrow 0$, where all the $T^k$ are in $\text{Add}(T)$. 
\end{enumerate}
We will say that $T$ is a \emph{(partial) tilting object} of $\mathcal{A}$ when it is (partial) $n$-tilting, for some $n\in\mathbb{N}$. Finally,  a \emph{classical (partial) tilting} object of $\mathcal{A}$ will be a (partial) tilting object which is compact as an object of $\mathcal{D}(\mathcal{A})$.
\end{opr}

\begin{rems} \label{rem.tilting object}
\begin{enumerate}
\item We will see in the proof of Theorem \ref{teor.characterization of tilting objects} that we could have chosen any $m\in\mathbb{N}$ in condition T3.  That is, `$n$-tilting object' is synonymous with `tilting object of projective dimension $\leq n$'.
\item Condition T0 is always satisfied when $\mathcal{A}$ is AB4 (e.g. a Grothendieck category). We  could have chosen  to define a notion of partial tilting object in $\mathcal{A}$ by replacing conditions T0 and T1 by the condition that $\text{Hom}_{\mathcal{D}(\mathcal{A})}(T,T^{*(I)}[k])=0$, for all integers $k>0$ and all sets $I$, where $T^{*(I)}$ denotes the coproduct of $I$ copies of $T$ in $\mathcal{D}(\mathcal{A}$). But some of the nice properties would disappear.  For instance, the description of the aisle of the associated t-structure given in Corollary \ref{cor.partial tilting are presilting} below would not be necessarily true. 
\item The reader is invited to define the dual notions of \emph{(partial) $n$-cotilting object} and \emph{(partial) cotilting object}, which make sense in any  abelian category $\mathcal{A}$ such that $\mathcal{D}(\mathcal{A})$ has $\text{Hom}$ sets and arbitrary products. We also leave to him/her the statements of the results dual to those which will be proved in the rest of the section for (partial) tilting objects. 
\end{enumerate}
\end{rems}

Recall that a Grothendieck category $\mathcal{A}$ is \emph{locally noetherian} when it has a set $\mathcal{S}$ of noetherian generators, i.e. all the objects in $\mathcal{S}$ satisfy ACC on subobjects. 
\begin{ex} \label{exem.tilting object}
Let $\mathcal{A}$ be a locally noetherian Grothendieck category of finite global dimension (e.g. $\mathcal{A}=\text{Qcoh}(\mathbb{X})$, where $\mathbb{X}$ is a smooth algebraic variety or $\mathcal{A}=\text{Mod}-R$ for a right noetherian ring $R$ of finite global dimension). Let $G$ be a generator of $\mathcal{A}$ and let $0\longrightarrow G\longrightarrow E^0\longrightarrow E^1\longrightarrow \cdots\longrightarrow E^m\longrightarrow 0$ be its minimal injective resolution. Then $T=\oplus_{0\leq i\leq m}E^i$ is a tilting object of $\mathcal{A}$. 
\end{ex}
\begin{proof}
We check the conditions of Definition \ref{def.(partial) tilting object}. Since $\mathcal{A}$ is AB4 condition T0 holds. By the locally noetherian condition, we know that each coproduct of injective objects is injective (see \cite[Proposition V.4.3]{St}), so that $T^{(I)}$ is an injective object of $\mathcal{A}$, for each set $I$. This gives condition T1, while condition T2 holds for some $n\in\mathbb{N}$ due to the finite global dimension of $\mathcal{A}$. Finally, taking as generating class $\mathcal{G}=\text{Sum}(G)$, we immediately get condition T3 using the exactness of coproducts and the fact that coproducts of injective objects are injective.
\end{proof}

\begin{cor} \label{cor.partial tilting are presilting}
Each partial tilting object $T$ of $\mathcal{A}$ is a partial silting object of $\mathcal{D}(\mathcal{A})$ whose associated t-structure is $(\mathcal{K}^{\leq 0}(\text{Sum}(T)),T^{\perp_{<0}})$. 
\end{cor}
\begin{proof}
The set $\mathcal{T}=\{T\}$ of $\mathcal{D}(\mathcal{A})$ satisfies assertion 3 of Theorem \ref{teor.weakly equivalent to presilting}, by taking $(\mathcal{V},\mathcal{V}^\perp [1])=(\mathcal{D}^{\leq 0}(\mathcal{A}),\mathcal{D}^{\geq 0}(\mathcal{A}))$ (see Lemma \ref{lem.finite projective dimension}). That the associated t-structure is as indicated follows from Proposition \ref{prop.aisle of partial tilting object}. 
\end{proof}

\begin{lem} \label{lem.Pres}
Let $T$ be an object of $\mathcal{A}$ satisfying properties T0 and T1 of Definition \ref{def.(partial) tilting object} and let $\mathcal{Y}:=\bigcap_{k>0}\text{Ker}(\text{Ext}_\mathcal{A}^k(T,?))$. If there is an $m\in\mathbb{N}$ such that $\text{Pres}^m(T)=\mathcal{Y}$, then $\text{Pres}^m(\mathcal{Y})\subseteq\mathcal{Y}$.
\end{lem}
\begin{proof}
Let $Y^0\stackrel{f^0}{\longrightarrow}Y^1\stackrel{f^1}{\longrightarrow}\cdots\stackrel{f^{m-1}}{\longrightarrow}Y^m$ be an exact sequence in $\mathcal{A}$, with all the $Y^k$ in $\mathcal{Y}$, and let us prove that $\text{Coker}(f^{m-1})\in\mathcal{Y}$. We put $B^k=\text{Im}(f^{k-1})$, for $k=1,\dots,m$ and put $B^{m+1}=\text{Coker}(f^{m-1})$. We shall prove by induction on $k=1,\dots,m+1$ that $B^k\in\text{Pres}^{k-1}(T)$. For $k=m+1$, this will give that $B^{m+1}\in\text{Pres}^m(T)=\mathcal{Y}$ and will end the proof. 

For $k=1$ is clear. Let us take $k>1$ (and $k\leq m+1$). We consider the induced exact sequence $0\longrightarrow B^{k-1}\stackrel{u}{\longrightarrow} Y^{k-1}\stackrel{\bar{f}^{k-1}}{\longrightarrow} B^k\longrightarrow 0$ and fix a $\text{Sum}(T)$-precover $p:T'\longrightarrow Y^{k-1}$, which is necessarily an epimorphism (e.g. one can take the canonical (epi)morphism $T^{(\text{Hom}_\mathcal{A}(T,Y^{k-1}))}\twoheadrightarrow Y^{k-1}$). Note that, since the induced map $p_*:\text{Hom}_\mathcal{A}(T,T')\longrightarrow\text{Hom}_\mathcal{A}(T,Y^{k-1})$ is surjective,   the long exact sequence of $\text{Ext}$ applied to $0\longrightarrow\text{Ker}(p)\longrightarrow T'\stackrel{p}{\longrightarrow}Y^{k-1}$ gives that $\text{Ker}(p)\in\bigcap_{k>0}\text{Ker}(\text{Ext}_\mathcal{A}^k(T,?))=\mathcal{Y}$. On the other hand, since $u$ is a monomorphism, the upper left corner of the pullback of $p$ and $u$ is $\text{Ker}(\bar{f}^{k-1}\circ p)$. But two parallel arrows in a pullback have isomorphic kernels. 
We  then get an exact sequence $0\longrightarrow\text{Ker}(p)\longrightarrow\text{Ker}(\bar{f}^{k-1}\circ p)\longrightarrow B^{k-1}\longrightarrow 0$. Now the proof of \cite[Lemma 3.8]{B} is valid on any AB3* abelian category, and the dual of this proof applies to our case. It follows that $\text{Ker}(\bar{f}^{k-1}\circ p)\in\text{Pres}^{k-2}(T)$. By considering the exact sequence $0\longrightarrow\text{Ker}(\bar{f}^{k-1}\circ p) \hookrightarrow T'\stackrel{\bar{f}^{k-1}\circ p}{\longrightarrow}B^k\longrightarrow 0$, we conclude that $B^k\in\text{Pres}^{k-1}(T)$. 
\end{proof}

The following lemma  was pointed out to us by Luisa Fiorot. We here reproduce the essential idea of her proof and thank her for the help. 

\begin{lem}\label{lem.YcogeneratesA}
Let   $T$ be a $n$-tilting object in $\mathcal{A}$ and let $\mathcal{H}_T=T^{\perp_{>0}}\cap T^{\perp_{<0}}$ be the heart of the associated t-structure in $\mathcal{D}(\mathcal{A})$. Then the class
$\mathcal{Y}:=\bigcap_{k>0}{\rm Ker}(\text{Ext}^k_\mathcal{A}(T,?))$ coincides with $\mathcal{A}\cap\mathcal{H}_T$ and is a cogenerating class in  $\mathcal{A}$. Moreover,  any object
$A\in\mathcal{A}$ admits an exact sequence 
$0\to A\to Y\to T^1\to\cdots T^n\to 0$,  where $Y\in \mathcal{Y}$ and $T^k\in\Add(T)$ for any $k=1,\dots n$.
\end{lem}
\begin{proof}
The aisle of the associated t-structure in $\mathcal{D}(\mathcal{A})$ is $\mathcal{U}_T=T^{\perp_{>0}}=\mathcal{K}^{\leq 0}(\text{Add}(T))$ (see Corollary \ref{cor.partial tilting are presilting}). 
We then have an equality of subcategories $\mathcal{Y}=\mathcal{A}\cap T^{\perp_{>0}}=\mathcal{A}\cap T^{\perp_{>0}}\cap T^{\perp_{<0}}=\mathcal{A}\cap\mathcal{H}_T$ since there are not negative extensions between objects of $\mathcal{A}$.  Moreover, by Lemma \ref{lem.finite projective dimension},  we have that $\mathcal{D}^{\leq -n}(\mathcal{A})\subseteq T^{\perp_{>0}}$, and so $\mathcal{A}\subseteq\mathcal{U}_T[-n]=\mathcal{K}^{\leq n}(\text{Add}(T))$. 

Given  any object $A\in\mathcal{A}$, we then have a complex $$T^\bullet : \hspace*{0.5cm} \cdots\longrightarrow T^{-k}\stackrel{d^{-k}}{\longrightarrow} T^{-k+1}\longrightarrow \cdots\stackrel{d^{-1}}{\longrightarrow} T^0\stackrel{d^0}{\longrightarrow} \cdots\longrightarrow T^{n-1}\stackrel{d^{n-1}}{\longrightarrow} T^n\longrightarrow 0\longrightarrow \cdots,$$
with all the $T^j$ in $\text{Add}(T)$, which is isomorphic to $A[0]$ in $\mathcal{D}(\mathcal{A})$. Then $Y:=\text{Coker}(d^{-1})$ is in $\mathcal{Y}$ because, using stupid truncations,  we have an isomorphism $\sigma_{\leq 0}T^\bullet\cong Y[0]$ in $\mathcal{D}(\mathcal{A})$ and $\sigma_{\leq 0}T^\bullet\in\mathcal{K}^{\leq 0}(\text{Add}(T))=T^{\perp_{>0}}$. But we have a monomorphism $A\cong H^0(T^\bullet )\rightarrowtail Y$ and, hence, $\mathcal{Y}$ is a cogenerating class of $\mathcal{A}$. On the other hand, using intelligent truncation, we have an isomorphism $A[0]\cong\tau^{\geq 0}T^\bullet$ in $\mathcal{D}(\mathcal{A})$, and $\tau^{\geq 0}T^\bullet$ is identified with the induced complex $$\cdots\longrightarrow 0\longrightarrow Y\longrightarrow T^1\longrightarrow \cdots\longrightarrow T^{n}\longrightarrow 0\longrightarrow \cdots $$
\end{proof}

The following main result of the section shows that several common characterizations of tilting modules pass naturally to our general setting.

\begin{thm} \label{teor.characterization of tilting objects}
Let $\mathcal{A}$ be an abelian category such that $\mathcal{D}(\mathcal{A})$ has $\text{Hom}$ sets and arbitrary coproducts, and let $T$ be an object of $\mathcal{A}$ such that the coproduct $T^{(I)}$ is the same in $\mathcal{A}$ and $\mathcal{D}(\mathcal{A})$. The following assertions are equivalent:

\begin{enumerate}
\item $T$ is a tilting object of $\mathcal{A}$.
\item $T$ has finite projective dimension and is a silting (or tilting) object of $\mathcal{D}(\mathcal{A})$.
\item $T$ is a (partial) silting object of $\mathcal{D}(\mathcal{A})$ such that, for some generating class $\mathcal{G}$  of $\mathcal{A}$, there is an $n\in\mathbb{N}$ such that the  inclusions $\mathcal{G}\subset\text{Add}(T)[-n]\star\text{Add}(T)[-n+1]\star \cdots\star\text{Add}(T)[0]\subseteq\text{thick}_{\mathcal{D}(\mathcal{A})}(\mathcal{G})$ hold.
\item $\text{Ext}_\mathcal{A}^k(T,T^{(I)})=0$, for all integers $k>0$ and all sets $I$, and  there is a generating class $\mathcal{G}$  of $\mathcal{A}$ such that:

\begin{enumerate}
\item There is a common finite upper bound on the projective dimensions of the objects of $\mathcal{G}$;
\item $\text{thick}_{\mathcal{D}(\mathcal{A})}(\mathcal{G})=\text{thick}_{\mathcal{D}(\mathcal{A})}(\text{Sum}(T))$.

\end{enumerate}

\item If $\mathcal{Y}=\bigcap_{k>0}\text{Ker}(\text{Ext}_\mathcal{A}^k(T,?))$, then:

\begin{enumerate}
\item  $\mathcal{Y}=\text{Pres}^{m}(T),$ for some integer $m\in\mathbb{N}$;
\item $\mathcal{Y}$ is a cogenerating class of $\mathcal{A}$.
\end{enumerate}
\end{enumerate}
If any of the equivalent conditions hold, the associated t-structure in $\mathcal{D}(\mathcal{A})$ is $\tau_T=(\mathcal{K}^{\leq 0}(\text{Sum}(T)),T^{\perp_{<0}})=(T^{\perp_{>0}},T^{\perp_{<0}})$.
\end{thm}
\begin{proof}
$1)\Longrightarrow 2)$ By Corollary \ref{cor.partial tilting are presilting}, we know that $T$ is a partial silting object of $\mathcal{D}(\mathcal{A})$. On the other hand, by property T3 in the definition of tilting object and Lemma \ref{lem.t-structure class generators}, we easily get that $T$  is a generator of $\mathcal{D}(\mathcal{A})$. 

$2)\Longrightarrow 1)$ It immediately follows that  $T$ is a partial tilting object, and only condition T3 in Definition \ref{def.(partial) tilting object} needs to be checked.   Let us put $\text{pd}_\mathcal{A}(T)=n$.  If  $M\in\mathcal{A}$ is any object, then we have that $M[n]\in T^{\perp_{>0}}=\mathcal{U}_T=\mathcal{K}^{\leq 0}(\text{Add}(T))$, due to Lemma \ref{lem.finite projective dimension}. Then we have a complex $$T_M^\bullet :\cdots\longrightarrow T_M^{-n-1}\stackrel{d^{-n-1}}{\longrightarrow} T_M^{-n}\stackrel{d^{-n}}{\longrightarrow} \cdots\stackrel{d^{-2}}{\longrightarrow} T_M^{-1}\stackrel{d^{-1}}{\longrightarrow} T_M^0\longrightarrow 0\longrightarrow\cdots $$ with homology concentrated in degree $-n$ and $H^{-n}(T_M^\bullet )\cong M$, such that $T_M^{-k}\in\text{Add}(T)$ for all $k\in\mathbb{N}$. 
Putting $B_M^{-n}=\text{Im}(d^{-n-1})$ and $Z_M^{-n}=\text{Ker}(d^{-n})$, we get an exact sequence $0\longrightarrow B_M^{-n}\longrightarrow Z_M^{-n}\longrightarrow M\longrightarrow 0$. Putting $\mathcal{G}=\{Z_M^{-n}\text{: }M\in\mathcal{A}\}$, we get a generating class of $\mathcal{A}$ satisfying the mentioned property T3. 

$1)\Longrightarrow 3)$ By Corollary \ref{cor.partial tilting are presilting}, we know that $T$ is a partial silting object of $\mathcal{D}(\mathcal{A})$ whose associated t-structure is the one of the final statement of this theorem. Let $n\in\mathbb{N}$ be such that $T$ is $n$-tilting. By assertion 2 of Proposition \ref{prop.aisle of partial tilting object} and condition T3, we get that 
 $\mathcal{G}\subset\text{Add}(T)[-n]\star\text{Add}(T)[-n+1]\star \cdots\star\text{Add}(T)[0]$. 

Consider now a $\mathcal{G}$-resolution of $T$  $$\cdots\longrightarrow G^{-m}\longrightarrow G^{-m+1}\longrightarrow \cdots\longrightarrow G^{-1}\longrightarrow G^0\longrightarrow T\longrightarrow 0 $$  and denote by $Z^{-m}$ the kernel of the differential $G^{-m}\longrightarrow G^{-m+1}$, for each $m\in\mathbb{N}$. We then have an induced complex $$X^\bullet :\cdots\longrightarrow 0\longrightarrow Z^{-n}\longrightarrow G^{-n}\longrightarrow G^{-n+1}\longrightarrow\cdots\longrightarrow G^1\longrightarrow G^0\longrightarrow 0\longrightarrow\cdots $$ which is quasi-isomorphic to the stalk complex $T[0]$.  If we consider the truncated complex $$G^\bullet :\cdots \longrightarrow 0\longrightarrow G^{-n}\longrightarrow G^{-n+1}\longrightarrow\cdots\longrightarrow G^1\longrightarrow G^0\longrightarrow 0 \longrightarrow\cdots ,$$ then $X^\bullet$ is just the cone in $\mathcal{K}(\mathcal{A})$ of the chain map $Z^{-n}[n]\longrightarrow G^\bullet$ determined by the inclusion $Z^{-n}\hookrightarrow G^{-n}$. We then get a triangle $$Z^{-n}[n]\longrightarrow G^\bullet\longrightarrow T[0]\stackrel{0}{\longrightarrow}Z^{-n}[n+1]$$ in $\mathcal{D}(\mathcal{A})$, which splits since $\text{Ext}_\mathcal{A}^{n+1}(T,Z^{-n})=0$. It follows that $T[0]$ is isomorphic to a direct summand of $G^\bullet$ in $\mathcal{D}(\mathcal{A})$, and we clearly have $G^\bullet\in\text{thick}_{\mathcal{D}(\mathcal{A})}(\mathcal{G})$.

$3)\Longrightarrow 4)$ is clear, except for condition 4.a. By assertion 2 of Proposition \ref{prop.aisle of partial tilting object}, we know that each $G\in\mathcal{G}$ admits an exact sequence $0\longrightarrow G\longrightarrow T^0\longrightarrow T^1\longrightarrow \cdots\longrightarrow T^n\longrightarrow 0$ (*), with all the $T^k$  in $\text{Add}(T)$. But, by hypothesis, we know that $T$ is a partial silting object of $\mathcal{D}(\mathcal{A})$, so that $\mathcal{U}_T={}^\perp (T^{\perp_{\leq 0}})$ is an aisle of $\mathcal{D}(\mathcal{A})$ such that $\text{Hom}_{\mathcal{D}(\mathcal{A})}(T,?)$ vanishes on $\mathcal{U}_T[1]$. 
Then the inclusion $\mathcal{G}\subset\text{Add}(T)[-n]\star\text{Add}(T)[-n+1]\star \cdots\star\text{Add}(T)[0]$ together with Lemma \ref{lem.t-structure class generators}  imply that $\mathcal{D}^{\leq 0}(\mathcal{A})={}^\perp(\mathcal{G}^{\perp_{\leq 0}})\subseteq\mathcal{U}_T[-n]$, so that $\text{Hom}_{\mathcal{D}(\mathcal{A})}(T,?)$ vanishes on $\mathcal{D}^{\leq 0}(\mathcal{A})[n+1]=\mathcal{D}^{<-n}(\mathcal{A})$. Therefore we get that $\text{pd}_\mathcal{A}(T)\leq n$ (see Lemma \ref{lem.finite projective dimension}) and, in particular, we have that $\text{pd}_\mathcal{A}(T^k)\leq n$ for all $T^k$ in the sequence (*). But, using the long exact sequence of $\text{Ext}$, one readily sees that the class $\mathcal{P}^{\leq n}(\mathcal{A})$ of objects of projective dimension $\leq n$ is closed under taking kernels of epimorphisms.  By iteration of this property, we then get that $\text{pd}_\mathcal{A}(G)\leq n$, for all $G\in\mathcal{G}$. 

$4)\Longrightarrow 1)$ Condition T1 of the definition of tilting object is automatic.  By condition 4.a and Lemma \ref{lem.finite projective dimension}, we know that there is a $m\in\mathbb{N}$ such that $\text{Hom}_{\mathcal{D}(\mathcal{A})}(G,?)$ vanishes on $\mathcal{D}^{<-m}(\mathcal{A})$, for all $G\in\mathcal{G}$. Since $T\in\text{Sum}(\mathcal{G})[r_1]\star \cdots\star\text{Sum}(\mathcal{G})[r_t]$, for some  $r_1,\dots,r_t\in\mathbb{Z}$, there is a large enough  $n\in\mathbb{N}$ such that $\text{Hom}_{\mathcal{D}(\mathcal{A})}(T,?)$ vanishes on $\mathcal{D}^{<-n}(\mathcal{A})$. Therefore we have $\text{pd}_\mathcal{A}(T)\leq n$, so that also condition T2 holds. 

Without loss of generality, put $n=\text{pd}_\mathcal{A}(T)$. Since $\mathcal{G}\subseteq\text{thick}_{\mathcal{D}(\mathcal{A})}(\text{Sum}(T))$, we know by Proposition \ref{prop.aisle of partial tilting object} that  each $G\in\mathcal{G}$ is isomorphic in $\mathcal{D}(\mathcal{A})$ to a complex of $\mathcal{K}^b(\text{Add}(T))$. This last complex will be then of the form $$T_G^\bullet:\cdots\longrightarrow 0\longrightarrow T^{-r}\longrightarrow \cdots\longrightarrow T^{-1}\stackrel{d^{-1}}{\longrightarrow} T^0\stackrel{d^0}{\longrightarrow}T^1\longrightarrow \cdots\longrightarrow T^p\longrightarrow 0\longrightarrow \cdots $$ Putting $Z_G=\text{Ker}(d^0)$ and $B_G=\text{Im}(d^{-1})$, we then get that $G\cong H^0(T^\bullet_G)=Z_G/B_G$ and $H^j(T_G^\bullet )=0$, for each $j\neq 0$. Then $\hat{\mathcal{G}}:=\{Z_G\text{: }G\in\mathcal{G}\}$ is a class of generators of $\mathcal{A}$ and, for each $\hat{G}=Z_G\in\hat{\mathcal{G}}$, we have an induced exact sequence $0\longrightarrow\hat{G}\longrightarrow T^0\longrightarrow T^1\longrightarrow \cdots\longrightarrow T^p\longrightarrow 0$ (**). 

To end the proof of this implication, we only need to check that we can choose $p=n$. If $p<n$ that is clear: we simply put $T^k=0$ for $k=p+1,\dots,n$. So we assume that $p>n$. 
 Let us consider the induced exact sequence $$0\longrightarrow K'\longrightarrow T^{p-n-1}\longrightarrow T^{p-n}\longrightarrow\cdots\longrightarrow T^p\longrightarrow 0, $$   which is an $\text{Add}(T)$-coresolution of $K'$. Since $\text{Ext}_\mathcal{A}^k(T,?)$ vanishes on all the $T^k$, $\text{Ext}_\mathcal{A}^j(T,K')$ is the $j$-th cohomology group of the induced complex of abelian groups $$\cdots\longrightarrow 0\longrightarrow\text{Hom}_\mathcal{A}(T,T^{p-n-1})\longrightarrow\text{Hom}_\mathcal{A}(T,T^{p-n})\longrightarrow\cdots\longrightarrow\text{Hom}_\mathcal{A}(T,T^p)\longrightarrow 0\longrightarrow \cdots,$$ where $\text{Hom}_\mathcal{A}(T,T^k)$ is in degree $k+n+1-p$ for each $k=p-n-1,p-n,\dots,p$. In particular, we have that $0=\text{Ext}_\mathcal{A}^{n+1}(T,K')$ is the cokernel of the map $\text{Hom}_\mathcal{A}(T,T^{p-1})\longrightarrow\text{Hom}_\mathcal{A}(T,T^p)$, so that this map is surjective. It follows that $d^{p-1}:T^{p-1}\longrightarrow T^p$ is a retraction and, hence, that  $\text{Ker}(d^{m-1})$ is also in $\text{Add}(T)$. That is,  if there exists a sequence as (**) of length $p$, then there also exists one of length $p-1$. By iterating the process, we arrive at an exact sequence like (**) of length exactly $n$. 

$1)=3)\Longrightarrow 5)$ Due to Lemma \ref{lem.YcogeneratesA}, we know that condition 5.b holds.  Let us put $n=\text{pd}_\mathcal{A}(T)$.  We have already seen in the proof of that lemma that  $\mathcal{Y}=\mathcal{U}_T\cap\mathcal{A}$, where  $\mathcal{U}_T=T^{\perp_{>0}}={}^\perp (T^{\perp_{\leq 0}})$ is the aisle of the associated t-structure in $\mathcal{D}(\mathcal{A})$. Moreover, by Proposition \ref{prop.aisle of partial tilting object}(3), we also know that $\mathcal{U}_T=\mathcal{K}^{\leq 0}(\text{Add}(T))$. We then get that $\mathcal{Y}=\mathcal{K}^{\leq 0}(\text{Add}(T))\cap\mathcal{A}\subseteq\text{Pres}^{n-1}(T)$. 

We will also prove the converse inclusion. Let $Y\in\text{Pres}^{n-1}(T)$ be any object and consider an exact sequence in $\mathcal{A}$ $$ T^{-n+1}\stackrel{d^{-n+1}}{\longrightarrow}T^{-n+2}\longrightarrow \cdots\longrightarrow T^{-1}\longrightarrow T^0\longrightarrow Y\longrightarrow 0. \hspace{1cm} (***)$$ Let us put $Z_n=\text{Ker}(d^{-n+1})$, consider a $\mathcal{G}$-resolution of $Z_n$ and patch it with the sequence (***). We then obtain a complex $$X^\bullet :\cdots\longrightarrow G^{-n-1}\longrightarrow G^{-n}\longrightarrow T^{-n+1}\longrightarrow \cdots\longrightarrow T^{-1}\longrightarrow T^0\longrightarrow 0\longrightarrow\cdots $$ which is quasi-isomorphic to the stalk complex $Y[0]$. By taking the stupid truncation at $-n+1$, we get a triangle $\sigma_{\geq -n+1}X^\bullet\longrightarrow X^\bullet\longrightarrow\sigma_{\leq -n}X^\bullet\stackrel{+}{\longrightarrow}$ in $\mathcal{K}(\mathcal{A})$. But $\sigma_{\geq -n+1}X^\bullet$ is a complex of objects in $\text{Add}(T)$ concentrated in degrees $-n+1,\dots,-1,0$. It follows that $\sigma_{\geq -n+1}X^\bullet\in\text{Add}(T)[0]\star\text{Add}(T)[1]\star \cdots\star\text{Add}(T)[n-1]$ and $\sigma_{\leq -n}X^\bullet\in\mathcal{D}^{\leq -n}(\mathcal{A})$. Therefore the functor $\text{Hom}_{\mathcal{D}(\mathcal{A})}(T,?[k]):\mathcal{D}(\mathcal{A})\longrightarrow\text{Ab}$ vanishes both on $\sigma_{\geq -n+1}X^\bullet$ and $\sigma_{\leq -n}X^\bullet$, for all $k>0$, due to the silting condition of $T$ and the fact that $\text{pd}_\mathcal{A}(T)\leq n$. We then get that $$\text{Ext}_\mathcal{A}^k(T,Y)\cong\text{Hom}_{\mathcal{D}(\mathcal{A})}(T,Y[k])\cong\text{Hom}_{\mathcal{D}(\mathcal{A})}(T,X^\bullet [k])=0,$$ for all $k>0$.

$5)\Longrightarrow 1)$  We have $T^{(I)}\in\text{Pres}^{m}(T)$, so that $\text{Ext}_\mathcal{A}^k(T,T^{(I)})=0$, for all sets $I$. Moreover, the cogenerating condition of $\mathcal{Y}$ together with Lemma \ref{lem.Pres} imply that each $M\in\text{Ob}(\mathcal{A})$ admits an exact sequence $0\longrightarrow M\longrightarrow Y^0\longrightarrow Y^1\longrightarrow \cdots\longrightarrow Y^m\longrightarrow Y^{m+1}\longrightarrow 0$, where the $Y^k$ are in $\mathcal{Y}$. Consider now the complex $$Y^\bullet :\hspace*{0.5cm}\cdots \longrightarrow 0\longrightarrow Y^0\longrightarrow Y^1\longrightarrow \cdots\longrightarrow Y^m\longrightarrow Y^{m+1}\longrightarrow 0 \longrightarrow\cdots $$
 Since $\text{Ext}_\mathcal{A}^k(T,?)$ vanishes on all $Y^j$ for all $k>0$, we get that $\text{Ext}_\mathcal{A}^j(T,M)$ is the $j$-th cohomology group of the complex $\text{Hom}_\mathcal{A}(T,Y^\bullet )$. It follows that $\text{Ext}_\mathcal{A}^k(T,M)=0$, for all $k>m+1$, and so $\text{pd}_\mathcal{A}(T)\leq m+1$. Therefore $T$ is a partial $(m+1)$-tilting module.  

In order to check  property T3 of Definition \ref{def.(partial) tilting object}, note that we can apply the dual of  \cite[Theorem 1.1]{AB}, with $\mathbf{X}$, $\omega$, $\hat{\mathbf{X}}$ and $\hat{\omega}$ replaced by $\mathcal{Y}$, $\text{Add}(T)$, $\mathcal{A}$ and $\text{Add}(T)^\vee$, respectively. Here $\text{Add}(T)^\vee$ denotes the subcategory consisting of those objects $M$ which admit an exact sequence $0\longrightarrow M\longrightarrow T^0\longrightarrow T^1\longrightarrow \cdots\longrightarrow T^p\longrightarrow 0$, for some $p\in\mathbb{N}$,  with all the $T^k$ in $\text{Add}(T)$. It follows that $\mathcal{G}:=\text{Add}(T)^\vee$ is a generating class. It remains to see that if $G\in\mathcal{G}$ and we fix an exact sequence $0\longrightarrow G\longrightarrow T^0\longrightarrow T^1\longrightarrow \cdots\longrightarrow T^p\longrightarrow 0$, with the $T^k$ in $\text{Add}(T)$, then we can choose one such  $\text{Add}(T)$-coresolution  with $p=\text{pd}_\mathcal{A}(T)$ ($=m+1$). This has been done in the proof of $4)\Longrightarrow 1)$.
\end{proof}

Note that if $\mathcal{A}$ has enough projectives, then $\text{Proj}\mathcal{A}\subset\text{add}(\mathcal{G})$, for any generating class $\mathcal{G}$. Moreover,  if there is a common finite upper bound in the projective dimensions of the objects of $\mathcal{G}$, then    one has $\text{thick}_{\mathcal{D}(\mathcal{A})}(\mathcal{G})=\text{thick}_{\mathcal{D}(\mathcal{A})}(\text{Sum}(\mathcal{P}))=\mathcal{K}^b(\text{Proj}\mathcal{A})$, for any class $\mathcal{P}$  of projective generators. Therefore  we immediately get from Theorem \ref{teor.characterization of tilting objects} and its proof the following result, which is well-known for module categories (see  \cite[Theorem 3.11]{B}, \cite[Corollary 3.7]{W}). 

\begin{cor} \label{cor.tilting object when enough projectives}
Let $\mathcal{A}$  as in Setup \ref{setup} have  enough projectives and  let  $T$ be an object of $\mathcal{A}$ such that the coproduct $T^{(I)}$ is the same in $\mathcal{A}$ and $\mathcal{D}(\mathcal{A})$, for every set $I$.  The following assertions are equivalent:

\begin{enumerate}
\item $T$ is a tilting object of $\mathcal{A}$.
\item The following assertions hold:

\begin{enumerate}
\item $\text{Ext}_\mathcal{A}^k(T,T^{(I)})=0$, for all integers $k>0$ and all sets $I$, 
\item There exists an exact sequence $0\longrightarrow P^{-n}\longrightarrow \cdots\longrightarrow P^{-1}\longrightarrow P^0\longrightarrow T\longrightarrow 0$, where all the $P^{-k}$ are projective objects;
\item For some (resp. every) class of projective generators $\mathcal{P}$ of $\mathcal{A}$, there is  $m\in\mathbb{N}$ such that each object $P\in\mathcal{P}$ admits an exact sequence $0\longrightarrow P\longrightarrow T^0\longrightarrow T^1\longrightarrow \cdots\longrightarrow T^m\longrightarrow 0$, where all the $T^k$ are in $\text{Add}(T)$. 
\end{enumerate}
\item $\text{Ext}_\mathcal{A}^k(T,T^{(I)})=0$, for all integers $k>0$ and all sets $I$, and $\text{thick}_{\mathcal{D}(\mathcal{A})}(\text{Add}(T))=\mathcal{K}^b(\text{Proj}\mathcal{A})$.
\item If $\mathcal{Y}:=\bigcap_{k>0}\text{Ker}(\text{Ext}_\mathcal{A}^k(T,?))$, then $\mathcal{Y}=\text{Pres}^{m}(T)$, for some $m\in\mathbb{N}$, and $\mathcal{Y}$ is a cogenerating class of $\mathcal{A}$.
\end{enumerate}
If $T$ satisfies any of these equivalent conditions, then $\mathcal{A}$ has a  projective generator and  $\tau_T=(T^{\perp_{>0}},T^{\perp_{<0}})=(\mathcal{K}^{\leq 0}(\text{Sum}(T)),T^{\perp_{<0}})$ is a t-structure in $\mathcal{D}(\mathcal{A})$. 
\end{cor} 
\begin{proof}
By our previous comments, the only thing that does not follow immediately from Theorem \ref{teor.characterization of tilting objects} and its proof is the fact that $\mathcal{A}$ has a projective generator. To see this, consider condition 2.b and  take $P=\oplus_{0\leq k\leq n}P^{-k}$. Since $T$ is a silting object, whence a generator,  of $\mathcal{D}(\mathcal{A})$ we get that also $P$ is a generator of $\mathcal{D}(\mathcal{A})$. Therefore, for each $0\neq M\in\mathcal{A}$, we have that $\text{Hom}_\mathcal{A}(P,M)=\text{Hom}_{\mathcal{D}(\mathcal{A})}(P,M)\neq 0$ since $\text{Hom}_{\mathcal{D}(\mathcal{A})}(P[j],M)=0$ for $j\neq 0$. Then $P$ is a generator of $\mathcal{A}$. 
\end{proof}

\begin{rems}
\begin{enumerate}
\item It is evident from the proof of Theorem \ref{teor.characterization of tilting objects} that the $n$ of $n$-tilting is the same natural number that appears in assertion 3 of the theorem and equals $m+1$, for the $m$ of assertion 5 in the theorem. This and its dual generalize \cite[Theorem 3.11]{B}.  
\item From  the dual of Theorem \ref{teor.characterization of tilting objects} one derives that  \cite[Theorem 4.5]{Sto2} is  valid for any cotilting object in an abelian category $\mathcal{A}$  such that $\mathcal{D}(\mathcal{A})$ has products and $\text{Hom}$ sets (e.g. for any cotilting object in a Grothendieck category). 
\end{enumerate} 
\end{rems}

\section{Triangulated equivalences induced by  tilting objects} \label{sect.triangulated equivalences}

Given a triangulated category with coproducts $\mathcal{D}$ and a t-structure $\tau$ in it with heart $\mathcal{H}_\tau$, one can ask naturally  the following:

\begin{ques}
Under what conditions does the inclusion $\mathcal{H}_\tau\hookrightarrow\mathcal{D}$ extend to a triangulated equivalence  $\mathcal{D}(\mathcal{H}_\tau)\stackrel{\sim}{\longrightarrow}\mathcal{D}$? At least, when does it extend to a fully faithful functor $\mathcal{D}^b(\mathcal{H}_\tau)\rightarrowtail\mathcal{D}$?
\end{ques}

Very recently, Psaroudakis and Vitoria (see \cite[Section 5]{PV}) have used filtered categories and the realisation functor introduced in \cite{BBD} to study in depth the second part of the question, and have obtained precise answers in case $\tau$ is induced by a tilting object of $\mathcal{D}$, for some choices of $\mathcal{D}$. Also very recently, Fiorot, Mattiello and the second named author have used the tilting theory of last section to show in a very simple way that if $\mathcal{A}$ is an abelian category as in Setup \ref{setup}, $T$ is a tilting object of $\mathcal{A}$ and $\mathcal{H}_T$ denotes the heart of the associated t-structure in $\mathcal{D}(\mathcal{A})$, then the inclusion $\mathcal{H}_T\hookrightarrow\mathcal{D}(\mathcal{A})$ extends to a triangulated equivalence  $\mathcal{D}(\mathcal{H}_T)\stackrel{\sim}{\longrightarrow}\mathcal{D}(\mathcal{A})$  (see \cite[Proposition 2.3]{FMS}). The main result of this section, Theorem \ref{teor.tilting derived equivalence}  will show that the analogous of the latter result holds when we replace $\mathcal{D}(\mathcal{A})$ by any compactly generated algebraic triangulated category and $T$ by any \text{bounded tilting object} in $\mathcal{D}$ (see definition below).

The sets of objects with which we shall be dealing here are the following. 
 
\begin{opr} \label{def.big tilting object}
Let $\mathcal{D}$ be a triangulated category with arbitrary coproducts and let $\mathcal{T}$ be a set of objects of $\mathcal{D}$. We will say that $\mathcal{T}$ is a \emph{bounded (partial)  tilting set} when it is  (partial) tilting  (see Definition \ref{def.strong presilting}) and weakly equivalent to a classical (partial)  silting set (see Definitions \ref{def.weakly equivalent silting sets} and \ref{def.classical silting objects}). Of course the same adjectives are applied to an object $T$ when the set $\{T\}$ is so. 

A t-structure $\tau =(\mathcal{U},\mathcal{U}^\perp [1])$ will be called a \emph{(bounded, resp. classical) (partial) tilting t-structure} when there is a (bounded, resp. classical ) (partial) tilting  set $\mathcal{T}$ of $\mathcal{D}$ such that $\tau =({}^\perp (\mathcal{T}^{\perp_{\leq 0}}),\mathcal{T}^{\perp_{<0}})$. 
\end{opr}

\begin{rem}
Although, we will not use it in this paper, one can define \emph{bounded (partial) silting sets} and \emph{bounded (partial) silting t-structures} just replacing `tilting' by `silting' in the last definition.   When $A$ is an ordinary algebra, one can easily see that a bounded silting  object of $\mathcal{D}(A)$ is exactly what is called a semi-tilting complex in \cite{W} and a big silting complex in \cite{AMV}.  
\end{rem}

In the proof of our next result and throughout the rest of this section, we will assume that the reader is acquainted with  (small) dg categories and their derived categories. Our terminology is mainly taken from \cite{K2}, but we will use several results from \cite{K}. However, we will still keep the notation $\mathcal{K}(\mathcal{A})$ (instead of $\mathcal{H}\mathcal{A}$ as in Keller's papers) for the homotopy category of $\mathcal{A}$. 

\begin{prop} \label{prop.abelian category in module category}
Let $\mathcal{D}$ be a compactly generated algebraic triangulated category and let $\tau$ be a bounded tilting t-structure in $\mathcal{D}$, whose heart is denoted by $\mathcal{H}$. There is a bounded  tilting object $T$ in $\mathcal{D}$ satisfying the following properties:

\begin{enumerate}
\item $\tau =(T^{\perp_{>0}},T^{\perp_{<0}})$ is the t-structure associated to $T$;
\item If $E:=\text{End}_{\mathcal{D}}(T)$, then the functor $\text{Hom}_\mathcal{H}(T,?):\mathcal{H}\longrightarrow\text{Mod}-E$ is fully faithful, exact and has a left adjoint which preserves products. 
\end{enumerate}
\end{prop}
\begin{proof}
Let us assume that $\mathcal{T}$ is a tilting set such that $\tau =(\mathcal{T}^{\perp_{>0}},\mathcal{T}^{\perp_{<0}})$. By taking $\tilde{T}:=\coprod_{T\in\mathcal{T}}T$, we can and shall assume that $\mathcal{T}=\{\tilde{T}\}$.  By definition, $\mathcal{D}$ has a silting set $\mathcal{S}$ of compact objects which is weakly equivalent  to $\mathcal{T}$. By \cite[Theorem 4.3]{K}, we have a dg category $\mathcal{A}$, which can be chosen to be $K$-flat,  and a triangulated equivalence $\mathcal{D}\stackrel{\sim}{\longrightarrow}\mathcal{D}(\mathcal{A})$ which takes $\mathcal{S}$ onto  the set  $\{A^\wedge\text{: }A\in\mathcal{A}\}$ of representable $\mathcal{A}$-modules. These form then a silting set of compact generators of $\mathcal{D}(\mathcal{A})$, which implies  that $H^n\mathcal{A}(A,A')=0$, for all $n>0$ and all $A,A'\in\mathcal{A}$. 
So in the rest of the proof, we assume that $\mathcal{D}=\mathcal{D}(\mathcal{A})$, where $\mathcal{A}$ is a $K$-flat dg category with this homology upper boundedness condition,  and that  $\text{thick}_{\mathcal{D}(\mathcal{A})}(\text{Sum}(\tilde{T}))=\text{thick}_{\mathcal{D}(\mathcal{A})}(\text{Sum}(A^\wedge\text{: }A\in\mathcal{A}))$. It follows that there is a set $I$ such that $A^\wedge\in\text{thick}_{\mathcal{D}(\mathcal{A})}(\tilde{T}^{(I)})$, for all $A\in\mathcal{A}$. We put $T=\tilde{T}^{(I)}$ in the sequel, and will check that it satisfies the requirements.  

Property 1 is straightforward. We can assume without loss of generality that $T$ is a homotopically projective dg $\mathcal{A}$-module (i.e. that $\text{Hom}_{\mathcal{K}(\mathcal{A})}(T,?)$ vanishes on acyclic complexes) and put $B=\text{End}_{\mathcal{C}_{dg}\mathcal{A}}(T)$. Then $B$ is a dg algebra such that $H^k(B)\cong\text{Hom}_{\mathcal{K}(\mathcal{A})}(T,T[k])\cong\text{Hom}_{\mathcal{D}(\mathcal{A})}(T,T[k])$, for all $k\in\mathbb{Z}$, so that $B$ has homology concentrated in degree $0$ and $H^0(B)\cong\text{End}_{\mathcal{D}(\mathcal{A})}(T)=:E$. Note that the classical truncation at zero of $B$ gives a dg subalgebra $\tau^{\leq 0}B$. The corresponding inclusion $j:\tau^{\leq 0}B\hookrightarrow B$ of dg algebras is a quasi-isomorphism, which implies that the restriction of scalars $j_*:\mathcal{D}(B)\longrightarrow\mathcal{D}(\tau^{\leq 0}B)$ is a triangulated equivalence. Without loss of generality, we can replace $B$ by $\tau^{\leq 0}B$ and assume that $B=\oplus_{n\geq 0}B^{-n}$ is concentrated in degrees $\leq 0$. Note that $T$ is canonically a dg $B-\mathcal{A}$-bimodule and we have now a canonical projection $p:B\longrightarrow H^0(B)=E$ which is also a quasi-isomorphism of dg algebras. This implies that if $p^*=E\otimes_B?:\mathcal{C}_{dg}(B^{op})\longrightarrow\mathcal{C}_{dg}(E^{op})$ is the extension of scalars, then its left derived functor $\mathbf{L}p^*:\mathcal{D}(B^{op})\longrightarrow\mathcal{D}(E^{op})$ is a triangulated equivalence.  If necessary, we replace $T$ by a quasi-isomorphic homotopically projective $B-\mathcal{A}$-bimodule. The $K$-flatness of $A$ implies that ${}_BT$ is homotopically flat (just adapt the proof of \cite[Lemma 3.6]{NS3} to the case when $A$ is a dg category instead of a dg algebra). This implies that $p\otimes_B1_T:T=B\otimes_BT\longrightarrow E\otimes_BT$ is a quasi-isomorphism of $B-\mathcal{A}$-modules and that  $\mathbf{L}p^*(T)=E\otimes_BT$. As a final  step of reduction, we replace $B$ by $E$ and $T$ by $E\otimes_BT$, thus assuming in the rest of the proof that $T$ is an $E-\mathcal{A}$-bimodule. 

Now  assertion 4 of \cite[Theorem 6.4]{NS2} holds (with $E$ instead of $\mathcal{B}$), so that the classical derived  functor $\text{RHom}_\mathcal{A}(T,?):\mathcal{D}(\mathcal{A})\longrightarrow\mathcal{D}(E)$ is fully faithful and its left adjoint $?\otimes_E^\mathbf{L}T:\mathcal{D}(E)\longrightarrow\mathcal{D}(\mathcal{A})$ has itself a (fully faithful) left adjoint $\Lambda :\mathcal{D}(\mathcal{A})\longrightarrow\mathcal{D}(E)$. We now put $G:=\text{Hom}_\mathcal{H}(T,?):\mathcal{H}\longrightarrow\text{Mod}-E$. We claim that the functor $F:\text{Mod}-E\longrightarrow\mathcal{H}$ defined as the following composition $$\text{Mod}-E\stackrel{can}{\longrightarrow}\mathcal{D}(E)\stackrel{?\otimes_E^\mathbf{L}T}{\longrightarrow}\mathcal{D}(\mathcal{A})\stackrel{\tilde{H}}{\longrightarrow}\mathcal{H} \hspace*{1cm} (*) $$ is a left adjoint to $G$, where $\tilde{H}$ is the cohomological functor associated to the t-structure $\tau$.  Note that if $M$ is an $E$-module, then $M[0]\in\mathcal{D}^{\leq 0}(E)$, so that $M[0]$ is isomorphic in $\mathcal{D}(E)$ to the Milnor colimit of a sequence $M_0\stackrel{\alpha_1}{\longrightarrow} M_1\stackrel{\alpha_2}{\longrightarrow} \cdots\stackrel{\alpha_n}{\longrightarrow} M_n\longrightarrow \cdots$, where $M_0\in\text{Sum}(E)[0]$ and  $\text{cone}(\alpha_n)\in\text{Sum}(E)[n]$, for all  $n\geq 0$. We then have $M[0]\otimes_E^\mathbf{L}T\cong\text{Mcolim}(M_n\otimes_E^\mathbf{L}T)$, which is an object of the aisle $\mathcal{U}=T^{\perp_{>0}}$ (see Theorem \ref{teor.weakly equivalent to presilting}). Bearing in mind that $\tilde{H}_{| \mathcal{U}}:\mathcal{U}\longrightarrow\mathcal{H}$  is left adjoint to the inclusion functor $\mathcal{H}\hookrightarrow\mathcal{U}$ (see \cite[Lemma 3.1]{PS1}), we get a sequence of isomorphisms, for $M\in\text{Mod}-E$ and $Y\in\mathcal{H}$

\begin{center}
$\text{Hom}_\mathcal{H}(F(M),Y)=\text{Hom}_\mathcal{H}(\tilde{H}(M[0]\otimes_E^\mathbf{L}T),Y)\cong\text{Hom}_\mathcal{U}(M[0]\otimes_E^\mathbf{L}T,Y)=\text{Hom}_{\mathcal{D}(\mathcal{A})}(M[0]\otimes_E^\mathbf{L}T,Y)\cong\text{Hom}_{\mathcal{D}(E)}(M[0],\text{RHom}_\mathcal{A}(T,Y)). $
\end{center}
But we have that $Y\in\mathcal{H}=T^{\perp_{>0}}\cap T^{\perp_{<0}}$, which implies that $H^k(\text{RHom}_\mathcal{A}(T,Y))=\text{Hom}_{\mathcal{D}(\mathcal{A})}(T,Y[k])=0$, for all $k\neq 0$. It follows that we have an isomorphism $\text{RHom}_\mathcal{A}(T,Y)\cong\text{Hom}_{\mathcal{D}(\mathcal{A})}(T,Y)[0]=G(Y)[0]$ in $\mathcal{D}(E)$. We then have an isomorphism $\text{Hom}_\mathcal{H}(F(M),Y)\cong \text{Hom}_{\mathcal{D}(E)}(M[0],G(Y)[0])\cong\text{Hom}_E(M,G(Y))$. It is routine to check that this isomorphism is natural on $M$ and $Y$, so that $(F,G)$ is an adjoint pair. 

In order to prove that $G$ is fully faithful, we will check that the counit $\epsilon :F\circ G\longrightarrow 1_\mathcal{H}$ is a natural isomorphism. Due to the fact that $(?\otimes_E^\mathbf{L}T,\text{RHom}_\mathcal{A}(T,?))$ is an adjoint pair with fully faithful right component, the counit of this adjunction $\delta :(?\otimes_E^\mathbf{L}T)\circ\text{RHom}_\mathcal{A}(T,?)\longrightarrow 1_{\mathcal{D}(\mathcal{A})}$ is a natural isomorphism. Then $\tilde{H}(\delta ):\tilde{H}(\text{RHom}_\mathcal{A}(T,Y)\otimes_E^\mathbf{L}T)\longrightarrow\tilde{H}(Y)=Y$ is an isomorphism in $\mathcal{H}$, for all $Y\in\mathcal{H}$. We have already seen that there is an isomorphism $\text{RHom}_\mathcal{A}(T,Y)\cong G(Y)[0]$ in $\mathcal{D}(E)$, which allows us to view $\tilde{H}(\delta )$ as a morphism $(F\circ G)(Y)=\tilde{H}(G(Y)[0]\otimes_E^\mathbf{L}T)\longrightarrow Y$. This morphism is easily identified with the counit morphism $\epsilon_Y$.  

The functor $?\otimes_E^\mathbf{L}T:\mathcal{D}(E)\longrightarrow\mathcal{D}(\mathcal{A})$ preserves products since it has a left adjoint. Moreover, by Theorem \ref{teor.presilting t-structures}, we know that $\tilde{H}:\mathcal{D}(\mathcal{A})\longrightarrow\mathcal{H}$ preserves products. It follows that the three functors in the composition $(*)$ which defines $F$ preserve products. Therefore $F$ preserves products. 
\end{proof}

Recall the following definition (see \cite[Subsection 1.2.1]{L} and \cite[Section 1]{N2}).

\begin{opr} \label{def.left complete}
Let $\mathcal{H}$ be an abelian category. We say that its derived category $\mathcal{D}(\mathcal{H})$ is \emph{left complete} when the induced (non-canonical) morphism $M\longrightarrow\text{Holim}(\tau^{\geq n}M)$ is an isomorphism, for each $M\in\mathcal{D}(\mathcal{H})$. 
\end{opr}

The following result is  a direct consequence of \cite[Proposition 1.2.1.19]{L} since any AB3* abelian category with a projective generator is AB4* (see  the dual of \cite[Corollary 3.2.9]{Po}). See also \cite[Lemma 6.4] {PV} for a related result.

\begin{lem} \label{lem.Lurie}
If $\mathcal{H}$ is an AB3* abelian category with a projective generator, then $\mathcal{D}(\mathcal{H})$ is left complete. 
\end{lem}

In the rest of the section, we will assume the following situation, as derived from the proof of Proposition \ref{prop.abelian category in module category}.

\begin{setup} \label{setup2}
$\mathcal{A}$ will be a nonpositively graded  $K$-flat dg category, $E$ will be an ordinary $K$-algebra and $T$ will be a dg $E-\mathcal{A}$-bimodule satisfying the following conditions:

\begin{enumerate}
\item $T$ is a tilting object of $\mathcal{D}(\mathcal{A})$ such that $A^\wedge\in\text{thick}_{\mathcal{D}(\mathcal{A})}(T)$, for all $A\in\mathcal{A}$;
\item The canonical map $E\longrightarrow\text{End}_{\mathcal{D}(\mathcal{A})}(T)$ is an isomorphism of algebras.
\end{enumerate}
By \cite[Theorem 6.4 and Proposition 6.9]{NS2}, we know that we have adjoint pairs $(?\otimes_E^\mathbf{L}T,\text{RHom}_\mathcal{A}(T,?))$ and $(\Lambda ,?\otimes_E^\mathbf{L}T)$ of triangulated functors, where $\text{RHom}_\mathcal{A}(T,?)),\Lambda :\mathcal{D}(\mathcal{A})\longrightarrow\mathcal{D}(E)$ are fully faithful and preserve compact objects. 
\end{setup}

\begin{lem} \label{lem.image of G}
 The functor $G=\text{Hom}_\mathcal{H}(T,?):\mathcal{H}\longrightarrow\text{Mod}-E$ induces a triangulated functor $G=\mathbf{R}G:\mathcal{D}(\mathcal{H})\longrightarrow\mathcal{D}(E)$ whose essential image is contained in the essential image of $\text{RHom}_\mathcal{A}(T,?):\mathcal{D}(\mathcal{A})\longrightarrow\mathcal{D}(E)$. Moreover, if $\mathbf{L}F:\mathcal{D}(E)\longrightarrow\mathcal{D}(\mathcal{H})$ is the left derived functor of $F$, then $(\mathbf{L}F,G)$ is an adjoint pair.
\end{lem}
\begin{proof}
Since $G:\mathcal{H}\longrightarrow\text{Mod}-E$ is an exact functor, we have $G=\mathbf{R}G$. That is, its right derived functor $\mathbf{R}G:\mathcal{D}(\mathcal{H})\longrightarrow\mathcal{D}(E)$ exists and takes the complex $$Y^\bullet :\cdots \longrightarrow Y^{k-1}\longrightarrow Y^k\longrightarrow Y^{k+1}\longrightarrow \cdots$$ to the complex  $$G(Y^\bullet) :\cdots \longrightarrow G(Y^{k-1})\longrightarrow G(Y^k)\longrightarrow G(Y^{k+1})\longrightarrow \cdots$$ 
The last sentence of the lemma is then a standard fact about derived functors.

We now consider the full subcategory $\mathcal{C}$ of $\mathcal{D}(\mathcal{H})$ consisting of those complexes $Y^\bullet\in\mathcal{D}(\mathcal{H})$ such that $G(Y^\bullet )\in\text{Im}(\text{RHom}_\mathcal{A}(T,?))=:\mathcal{Z}$. The goal is to prove that $\mathcal{C}=\mathcal{D}(\mathcal{H})$. Note that  $\mathcal{Z}$ is a full triangulated subcategory of $\mathcal{D}(E)$ closed under taking products and, hence, it is also closed under taking homotopy limits. Moreover, the category $\mathcal{H}$ is AB3* (see \cite[Proposition 3.2]{PS1}) and has $T$ as a projective generator. It then follows that $\mathcal{H}$ is AB4* and, hence, products in $\mathcal{D}(\mathcal{H})$ are calculated pointwise. In particular $G=\mathbf{R}G$ preserves products and, hence, also homotopy limits. We then have that $\mathcal{C}$ is closed under taking homotopy limits in $\mathcal{D}(\mathcal{H})$, and this is a left complete triangulated category by Lemma \ref{lem.Lurie}. Our task is then reduced to prove that $\mathcal{D}^+(\mathcal{H})\subseteq\mathcal{C}$.

Let now $Y^\bullet\in\mathcal{C}^+(\mathcal{H})$ be any bounded below complex of objects in $\mathcal{H}$ which, without loss of generality, we assume to be concentrated in degrees $\geq 0$. By taking stupid truncations, we have that $Y^\bullet =\text{Holim}\sigma_{\leq n}Y^\bullet$ and, since the functor $G=\mathbf{R}G$ acts pointwise, we get that $G(Y^\bullet )=\text{Holim}G(\sigma_{\leq n}Y^\bullet )=\text{Holim}\sigma_{\leq n}G(Y^\bullet )$. In this way, the proof is further reduced to prove that $\mathcal{D}^b(\mathcal{H})\subset\mathcal{C}$. But each object $Y^\bullet$ of $\mathcal{D}^b(\mathcal{H})$ is a finite iterated extension of the stalk complexes $H^k(Y^\bullet )[-k]$. This finally reduces the proof to check that if $Y\in\mathcal{H}$, then each stalk complex $G(Y)[k]$ is in $\mathcal{Z}$. But this is clear since we have seen in the proof of Proposition \ref{prop.abelian category in module category} that $G(Y)[0]\cong\text{RHom}_\mathcal{A}(T,Y)$.
\end{proof}

\begin{lem} \label{lem.vahishing of derived tensor}
Let $$C^\bullet :\cdots\longrightarrow C^{k-1}\longrightarrow C^k\longrightarrow C^{k+1}\longrightarrow\cdots$$ be a complex of $E$-modules such that $C^k[0]\otimes_E^\mathbf{L}T=0$, for all $k\in\mathbb{Z}$. Then we also have that $C^\bullet\otimes_E^\mathbf{L}T=0$. 
\end{lem}
\begin{proof}
By taking stupid truncations, we have an isomorphism $C^\bullet\cong\text{Mcolim}(\sigma_{\geq -n}C^\bullet )$ in $\mathcal{K}(E)$ (and, hence, also in $\mathcal{D}(E)$). We then get an isomorphism $C^\bullet\otimes_E^\mathbf{L}T\cong\text{Mcolim}((\sigma_{\geq -n}C^\bullet )\otimes_E^\mathbf{L}T)$ in $\mathcal{D}(\mathcal{A})$. This reduces the proof to the case when $C^\bullet\in\mathcal{C}^+(E)$. Without loss of generality, we assume that  $C^\bullet\in\mathcal{C}^{\geq 0}(E)$. Then we have an isomorphism $C^\bullet =\text{Holim}(\sigma_{\leq n}C^\bullet )$ in $\mathcal{D}(E)$. But the functor $?\otimes_E^\mathbf{L}T:\mathcal{D}(E)\longrightarrow\mathcal{D}(\mathcal{A})$ preserves products, and hence homotopy limits, because it has a left adjoint. It follows that $C^\bullet\otimes_E^\mathbf{L}T\cong\text{Holim}((\sigma_{\leq n}C^\bullet )\otimes_E^\mathbf{L}T)$, so that the proof is reduced to the case when $C^\bullet\in\mathcal{C}^b(E)$. But in this case $C^\bullet$ is a finite iterated extension of the stalk complexes $C^k[-k]$ and we have $C^k[-k]\otimes_E^\mathbf{L}T=(C^k[0]\otimes_E^\mathbf{L}T)[-k]=0$, for each $k\in\mathbb{Z}$. 
\end{proof}

In the proof of the following lemma, we will follow the terminology and notation of \cite{S} to which we refer for all definitions. We will denote by $\mathcal{K}_\mathcal{P}(\mathcal{H})$ the homotopy category of homological projective multicomplexes over $\mathcal{H}$ and will consider the functor $\kappa :\mathcal{K}_\mathcal{P}(\mathcal{H})\longrightarrow\mathcal{K}(\mathcal{H})$ which takes any $X^{\bullet\bullet}\in\mathcal{K}_\mathcal{P}(\mathcal{H})$ to its totalization complex $\text{Tot}(X^{\bullet\bullet})$. Note that then $\text{Im}(\kappa )\subseteq\mathcal{K}(\text{Proj}\mathcal{H})=\mathcal{K}(\text{Sum}(T))$. 
Recall that if $q:\mathcal{K}(\mathcal{H})\longrightarrow\mathcal{D}(\mathcal{H})$ is the usual localization functor, then the composition $q\circ\kappa :\mathcal{K}_\mathcal{P}(\mathcal{H})\longrightarrow\mathcal{D}(\mathcal{H})$ is a triangulated equivalence (see \cite[Theorem 1]{S}).

\begin{lem} \label{lem.coproducts and G}
Let $(X_i^\bullet )_{i\in I}$ be a family of objects of $\mathcal{D}(\mathcal{H})$ which has a coproduct in this category and consider the induced triangle in $\mathcal{D}(E)$

\begin{center}

$\coprod G(X_i^\bullet )\longrightarrow G(\coprod X_i^\bullet )\longrightarrow C\stackrel{+}{\longrightarrow}$,
\end{center}

Then we have $C\otimes_E^\mathbf{L}T=0$ in $\mathcal{D}(\mathcal{A})$.
\end{lem}
\begin{proof}
All throughout the proof, we will see the action of the canonical functor $q:\mathcal{K}(\mathcal{H})\longrightarrow\mathcal{D}(\mathcal{H})$ as the identity. That is, by abuse of notation,  we will put $X=q(X)$ and $f=q(f)$, for any object $X$ and morphism $f$ in $\mathcal{K}(\mathcal{H})$. Note that the restriction functor $q_{\text{Im}(\kappa )}:\text{Im}(\kappa )\longrightarrow\mathcal{D}(\mathcal{H})$ is full and dense since $q\circ\kappa$ is an equivalence of categories. This allows us to assume all throughout the proof that $X_i^\bullet\in\text{Im}(\kappa )\subset\mathcal{K}(\text{Proj}\mathcal{H})$, for each $i\in I$.

Consider now the coproducts $\coprod 'X_i^\bullet$ and $\coprod X_i^\bullet$ of the $X_i^\bullet$ in $\mathcal{K}(\mathcal{H})$ and $\mathcal{D}(\mathcal{H})$, respectively. Note that the first one  is calculated as in $\mathcal{C}(\mathcal{H})$ (i.e. pointwise). Let us denote by $\lambda_j:X_j^\bullet\longrightarrow\coprod X_i^\bullet$ and $\iota_j:X_j^\bullet\longrightarrow\coprod 'X_i^\bullet$ the respective injections into the coproduct, in $\mathcal{D}(\mathcal{H})$ and $\mathcal{K}(\mathcal{H})$, respectively. 
By the universal property of the coproduct $\coprod X_i^\bullet$, we have a unique morphism $f:\coprod X_i^\bullet\longrightarrow\coprod 'X_i^\bullet$ in $\mathcal{D}(\mathcal{H})$ such that $f\circ\lambda_j =\iota_j$, for all $j\in I$. It follows that, for each $Y^\bullet\in\text{Im}(\kappa )$, we have the following commutative diagram, where the horizontal arrows are the canonical isomorphisms coming from the universal property of coproducts:

$$
\begin{xymatrix}{ \text{Hom}_{\mathcal{K(H)}}(\coprod ' X_i^{\bullet}, Y^{\bullet}) \ar[d]^{q} \ar[r]^{\simeq} & \prod\limits_{i\in I} \text{Hom}_{\mathcal{K(H)}}( X_i^{\bullet}, Y^{\bullet}) \ar[dd]^{q}\\
\text{Hom}_{\mathcal{D(H)}}(\coprod ' X_i^{\bullet}, Y^{\bullet}) \ar[d]^{f^*} &\\
\text{Hom}_{\mathcal{D(H)}}(\coprod  X_i^{\bullet}, Y^{\bullet}) \ar[r]^{\simeq} & \prod\limits_{i\in I} \text{Hom}_{\mathcal{D(H)}}( X_i^{\bullet}, Y^{\bullet})}
\end{xymatrix}
$$

Due to the fullness of $q_{| \text{Im}(\kappa )}$,  the right vertical arrow of the last diagram is an epimorphism, which implies that $f^*$ is an epimorphism, for each $Y^\bullet\in\text{Im}(\kappa )$.

But each object of $\mathcal{D}(\mathcal{H})$, in particular $\coprod X_i^\bullet $, is in the image of $q_{| \text{Im}(\kappa )}$. This means that we can (and shall) assume that, as an object of $\mathcal{K}(\mathcal{H})$, one has $\coprod X_i^\bullet\in\text{Im}(\kappa )$. Putting then $Y^\bullet=\coprod X_i^\bullet$ in the last paragraph, we get that the map $f^*:\text{Hom}_{\mathcal{D}(\mathcal{H})}(\coprod 'X_i^\bullet ,\coprod X_i^\bullet )\longrightarrow\text{Hom}_{\mathcal{D}(\mathcal{H})}(\coprod X_i^\bullet ,\coprod X_i^\bullet )$ is an epimorphism. Therefore $f$ is a section in $\mathcal{D}(\mathcal{H})$.

The octahedral axiom gives the following commutative diagram in $\mathcal{D}(E)$, where the morphism $\coprod G(X_i^\bullet )\longrightarrow G(\coprod X_i^\bullet )$ (resp. $\coprod G(X_i^\bullet )\longrightarrow G(\coprod 'X_i^\bullet )$) is the canonical one in $\mathcal{D}(E)$ (resp. $\mathcal{K}(E)$) induced by the universal property of the coproduct  (here we use that the coproduct of complexes in $\mathcal{D}(E)$ can be calculated pointwise and, hence, `coincides' with the coproduct in $\mathcal{K}(E)$): 

$$
\begin{xymatrix}{ \coprod G(X_i^{\bullet}) \ar[r] \ar@{=}[d] &  G(\coprod X_i^{\bullet}) \ar[r] \ar[d]^{G(f)}& C \ar@{-->}[d]\\
\coprod G(X_i^{\bullet}) \ar[r]  &  G(\coprod' X_i^{\bullet})\ar[r] \ar[d]^{\beta} & C' \ar@{-->}[d]^{\alpha}\\
& N \ar@{=}[r]& N}
\end{xymatrix}
$$

The fact that $G(f)$ is a section implies that $\beta$ is a retraction, so that $\alpha$ is also a retraction in $\mathcal{D}(E)$ and the dotted triangle splits. Since $\mathcal{D}(E)$ has arbitrary coproducts, we know that idempotents split and, hence, we have $C'\cong C\oplus N$. 

The proof is whence reduced to check that $C'\otimes_E^\mathbf{L}T=0$. But the canonical morphism $u:\coprod G(X_i^\bullet )\longrightarrow G(\coprod 'X_i^\bullet )$ is the image of the `same' morphism in $\mathcal{C}(E)$, which is a monomorphism in this category since $G:\mathcal{H}\longrightarrow\text{Mod}-E$ is exact and preserves finite coproducts. Therefore we can assume that we have an exact sequence $0\longrightarrow\coprod G(X_i^\bullet )\stackrel{u}{\longrightarrow} G(\coprod 'X_i^\bullet )\longrightarrow C'\longrightarrow 0$ in the abelian category $\mathcal{C}(E)$ whose image by the usual functor $\mathcal{C}(E)\longrightarrow\mathcal{D}(E)$ is the second horizontal triangle  of the diagram above. We then have an exact sequence $$0\longrightarrow\coprod G(X_i^\bullet )^k\stackrel{u}{\longrightarrow} G({\coprod} 'X_i^{\bullet} )^k\longrightarrow C'^k\longrightarrow 0 \hspace*{1cm} (*)$$ in $\text{Mod}-E$,  for each $k\in\mathbb{Z}$. But, due to the exactness of $G:\mathcal{H}\longrightarrow\text{Mod}-E$, we know that $G(Y^\bullet )^k=G(Y^k)$, for each $Y^\bullet\in\mathcal{C}(\mathcal{H})$ and each $k\in\mathbb{Z}$. This together with the fact that coproducts in $\mathcal{K}(\mathcal{H})$ are calculated pointwise allows us to rewrite the sequence (*) as  

$$0\longrightarrow\coprod G(X_i^k)\stackrel{u}{\longrightarrow} G(\coprod X_i^k)\longrightarrow C'^k\longrightarrow 0 \hspace*{1cm} (*)$$

We have seen in the proof of Proposition 
\ref{prop.abelian category in module category} that $G(Y)[0]\cong\text{RHom}_\mathcal{A}(T,Y)$, for each $Y\in\mathcal{H}$. The last exact sequence then gives a triangle in $\mathcal{D}(E)$  $$\coprod_{i\in I}\text{RHom}_\mathcal{A}(T,X_i^k[0])\longrightarrow\text{RHom}_\mathcal{A}(T,(\coprod_{i\in I}X_i^k)[0])\longrightarrow C'^k[0]\stackrel{+}{\longrightarrow}. $$ In principle, the coproduct of the $X_i^k$ has  been taken in $\mathcal{H}$. But, by the choice of the complexes $X_i^\bullet$, we know that each $X_i^k$ is a projective object of $\mathcal{H}$, so that the coproduct $\coprod_{i\in I}X_i^k$ is the same in $\mathcal{H}$ and in $\mathcal{D}(\mathcal{H})$.  Applying now the functor $?\otimes_E^\mathbf{L}T:\mathcal{D}(E)\longrightarrow\mathcal{D}(\mathcal{A})$ to the last triangle and bearing in mind that the counit $(?\otimes_E^\mathbf{L}T)\circ\text{RHom}_\mathcal{A}(T,?)\longrightarrow 1_{\mathcal{D}(\mathcal{A})}$ is a natural isomorphism, we easily conclude that $C'^k[0]\otimes_E^\mathbf{L}T=0$, for all $k\in\mathbb{Z}$. Then  $C'\otimes_E^\mathbf{L}T=0$ in $\mathcal{D}(\mathcal{A})$, by Lemma \ref{lem.vahishing of derived tensor}.

\end{proof}

We are now ready for the main result of this section. We refer the reader to \cite{FMS} for a simple proof of the corresponding result when $\mathcal{D}=\mathcal{D}(\mathcal{A})$ is the derived category of an abelian category $\mathcal{A}$ as in Setup \ref{setup},  and $T$ is a tilting object of $\mathcal{A}$.  

\begin{thm} \label{teor.tilting derived equivalence}
Let $K$ be a commutative ring, let $\mathcal{D}$ be a compactly generated algebraic triangulated $K$-category, let $\mathcal{T}$ be a  bounded tilting set of $\mathcal{D}$ and let $\mathcal{H}$ be the heart of the associated t-structure. The inclusion $\mathcal{H}\hookrightarrow\mathcal{D}$ extends to a triangulated equivalence $\Psi :\mathcal{D}(\mathcal{H})\stackrel{\sim}{\longrightarrow}\mathcal{D}$.  
\end{thm}
\begin{proof}
As shown in the proof Proposition \ref{prop.abelian category in module category},  we can assume  that $\mathcal{T}=\{T\}$ and that we are in the situation of Setup \ref{setup2}, with $\mathcal{D}=\mathcal{D}(\mathcal{A})$. We  can then use the functors $G$ and $F$ given by such proposition. 
Using  the adjoint pairs $(\mathbf{L}F,G)$ and $(\Lambda ,?\otimes_E^\mathbf{L}T)$, we see that we have  
 two compositions of triangulated functors 

\begin{center}
$\Phi :\mathcal{D}=\mathcal{D}(\mathcal{A})\stackrel{\Lambda}{\longrightarrow}\mathcal{D}(E)\stackrel{\mathbf{L}F}{\longrightarrow}\mathcal{D}(\mathcal{H})$

$\Psi :\mathcal{D}(\mathcal{H})\stackrel{G}{\longrightarrow}\mathcal{D}(E)\stackrel{?\otimes_E^\mathbf{L}T}{\longrightarrow}\mathcal{D}(\mathcal{A})=\mathcal{D}$,
\end{center}
such  that $(\Phi,\Psi)$ is an adjoint pair of triangulated functors. 

Put $\mathcal{T}':=\text{thick}_{\mathcal{D}(\mathcal{H})}(T)$ in the sequel. We claim that the restriction $\Psi_{| \mathcal{T}'}:\mathcal{T}'\longrightarrow\mathcal{D}(\mathcal{A})$ is fully faithful. It is enough to check that the induced map $\text{Hom}_{\mathcal{D}(\mathcal{H})}(T,T[k])\longrightarrow\text{Hom}_{\mathcal{D}(\mathcal{A})}(\Psi (T),\Psi (T)[k])$ is an isomorphism, for all $k\in\mathbb{Z}$. But $\Psi (T)=E\otimes_E^\mathbf{L}T=T$ and the last map gets identified with the composition $$\text{Hom}_{\mathcal{D}(\mathcal{H})}(T,T[k])\stackrel{G}{\longrightarrow}\text{Hom}_{\mathcal{D}(E)}(E,E[k])\stackrel{?\otimes_E^\mathbf{L}T}{\longrightarrow}\text{Hom}_{\mathcal{D}(\mathcal{A})}(T,T[k]).$$ This map is  bijective since the three appearing $K$-modules are zero when $k\neq 0$.

By Lemma \ref{lem.image of G}, we know that $\text{Im}(G)\subseteq\text{Im}(\text{RHom}_\mathcal{A}(T,?))$. Moreover,  the restriction of $?\otimes_E^\mathbf{L}T$ induces an equivalence of triangulated categories $\text{Im}(\text{RHom}_\mathcal{A}(T,?))\stackrel{\sim}{\longrightarrow}\mathcal{D}(\mathcal{A})$. The already proved fully faithful condition of $\Psi_{| \mathcal{T}'} =(?\otimes_E^\mathbf{L}T)\circ G_{| \mathcal{T}'}$ then implies that also $ G_{| \mathcal{T}'}:\mathcal{T}'\longrightarrow\mathcal{D}(E)$ is fully faithful. Moreover, since each object of $\text{per}(E)=\mathcal{D}^c(E)$ is a direct summand of a finite iterated extension of stalk complexes $E[k]=G(T[k])$, we readily see  that $G$ induces an equivalence of categories $\mathcal{T}'=\text{thick}_{\mathcal{D}(\mathcal{H})}(T)\stackrel{\sim}{\longrightarrow}\text{thick}_{\mathcal{D}(E)}(E)=\text{per}(E)$ whose inverse is necessarily $\mathbf{L}F_{| \text{per}(E)}:\text{per}(E)\longrightarrow\text{thick}_{\mathcal{D}(\mathcal{H})}(T)$. Therefore $\mathbf{L}F$ is fully faithful when restricted to $\text{per}(E)$. 

We will now prove that $\Phi :\mathcal{D}(\mathcal{A})\longrightarrow\mathcal{D}(\mathcal{H})$ is fully faithful. To do it we follow a standard method (see the proof of \cite[Lemma 4.2]{K}). Consider the full subcategory $\mathcal{Z}$ of $\mathcal{D}(\mathcal{A})$ consisting of the dg $\mathcal{A}$-modules $N$ such that the canonical map $\Phi:\text{Hom}_{\mathcal{D}(\mathcal{A})}(A^\wedge ,N[k])\longrightarrow\text{Hom}_{\mathcal{D}(\mathcal{H})}(\Phi (A^\wedge),\Phi (N)[k])$ is an isomorphism, for all $A\in\mathcal{A}$ and all integers $k\in\mathbb{Z}$. It is clearly a triangulated subcategory and it contains all representable $A$-modules $B^\wedge$. Indeed, by \cite[Theorem 6.4 and Proposition 6.9]{NS2}, we know that $\Lambda$ preserves compact objects and we have already seen that the restriction of $\mathbf{L}F$ to $\text{per}(E)$ is fully faithful. We then get that the restriction of $\Phi$ to $\text{per}(\mathcal{A})=\mathcal{D}^c(\mathcal{A})$ is fully faithful, so that $B^\wedge\in\mathcal{Z}$, for all objects $B\in\mathcal{A}$. We claim that $\mathcal{Z}$ is also closed under taking coproducts, which, by d\'evissage,   will imply that  $\mathcal{Z}=\mathcal{D}(\mathcal{A})$. For this it is enough to prove that, for each $A\in\mathcal{A}$, the object $\Phi (A^\wedge )$ is a compact in $\mathcal{D}(\mathcal{H})$,  in the sense that if a family of objects $(Y_i^\bullet )_{i\in I}$ of $\mathcal{D}(\mathcal{H})$ has a coproduct in this category, then the canonical map $\coprod_{i\in I}\text{Hom}_{\mathcal{D}(\mathcal{H})}(\Phi (A^\wedge ),Y_i^\bullet )\longrightarrow\text{Hom}_{\mathcal{D}(\mathcal{H})}(\Phi (A^\wedge ),\coprod_{i\in I}Y_i^\bullet )$ is an isomorphism.   Indeed,  since $\Phi$ is a composition of left adjoint functors, it preserves coproducts. That is, if $(N_i)_{i\in I}$ is a family of objects of $\mathcal{D}(\mathcal{A})$, then the coproduct in $\mathcal{D}(\mathcal{H})$ of the $\Phi (N_i)$ exists and is isomorphic to $\Phi (\coprod_{i\in I}N_i)$. If the compactness condition of $\Phi (A^\wedge )$ is checked, then we  will get an isomorphism $$\coprod_{i\in I}\text{Hom}_{\mathcal{D}(\mathcal{H})}(\Phi (A^\wedge),\Phi (N_i))\stackrel{\sim}{\longrightarrow}\text{Hom}_{\mathcal{D}(\mathcal{H})}(\Phi (A^\wedge ),\coprod_{i\in I}\Phi (N_i))\cong\text{Hom}_{\mathcal{D}(\mathcal{H})}(\Phi (A^\wedge ),\Phi (\coprod_{i\in I}N_i)).$$
When the $N_i$ are in $\mathcal{Z}$, this isomorphism can be then inserted in a commutative diagram, where the two horizontal arrows and the left vertical one are isomorphisms:

$$
\begin{xymatrix}{ \coprod \text{Hom}_{\mathcal{D(A)}}( A^{\wedge}, N_i) \ar[r]^{\simeq} \ar[d]^{\phi}  &  \text{Hom}_{\mathcal{D(A)}}(  A^{\wedge},\coprod N_i) \ar[d]^{\phi} \\
\coprod \text{Hom}_{\mathcal{D(H)}} (\phi( A^{\wedge} ), \phi(N_i)) \ar[r]^{\simeq} &  \text{Hom}_{\mathcal{D(H)}} ( \phi(A^{\wedge}),\phi(\coprod N_i)) }
\end{xymatrix}
$$

It will follow that also the right vertical arrow is an isomorphism, which in turn implies that $\coprod_{i\in I}N_i\in\mathcal{Z}$, so that $\mathcal{Z}$ will be closed under taking coproducts in $\mathcal{D}(\mathcal{A})$ as desired. 

Let us prove then that $\Phi (A^\wedge )$ is compact in $\mathcal{D}(\mathcal{H})$, for each $A\in\mathcal{A}$.  Using successively the adjunctions $(\mathbf{L}F,G)$ and $(\Lambda ,?\otimes_E^\mathbf{L}T)$, given a family $(Y_i^\bullet )$ of objects in $\mathcal{D}(\mathcal{H})$ that has a coproduct in this category and an object $A\in\mathcal{A}$, we have  a chain of isomorphisms 

\begin{center}
$\text{Hom}_{\mathcal{D}(\mathcal{H})}(\Phi (A^\wedge ),\coprod Y_i^\bullet )\cong\text{Hom}_{\mathcal{D}(E)}(\Lambda (A^\wedge ),G(\coprod Y_i^\bullet ))\cong\text{Hom}_{\mathcal{D}(\mathcal{A})}(A^\wedge ,G(\coprod Y_i^\bullet )\otimes_E^\mathbf{L}T)$. \hspace*{1cm} (**)
\end{center}
But, by Lemma \ref{lem.coproducts and G},    we have an isomorphism $(\coprod G(Y_i^\bullet ))\otimes_E^\mathbf{L}T\stackrel{\sim}{\longrightarrow} G(\coprod Y_i^\bullet )\otimes_E^\mathbf{L}T$. This fact, the isomorphism (**) and  the compactness of $A^\wedge$ in $\mathcal{D}(\mathcal{A})$ tell us that we have an isomorphism 

\begin{center}
$\text{Hom}_{\mathcal{D}(\mathcal{H})}(\Phi (A^\wedge ),\coprod Y_i^\bullet )\cong\text{Hom}_{\mathcal{D}(\mathcal{A})}(A^\wedge ,\coprod G(Y_i^\bullet )\otimes_E^\mathbf{L}T)\cong\coprod \text{Hom}_{\mathcal{D}(\mathcal{A})}(A^\wedge , G(Y_i^\bullet )\otimes_E^\mathbf{L}T)\cong\coprod\text{Hom}_{\mathcal{D}(E)}(\Lambda (A^\wedge ),G(Y_i^\bullet ))\cong\coprod\text{Hom}_{\mathcal{D}(\mathcal{H})}(\Phi (A^\wedge ),Y_i^\bullet )$,
\end{center}
which is easily seen to be the inverse of the canonical morphism $\coprod\text{Hom}_{\mathcal{D}(\mathcal{H})}(\Phi (A^\wedge ),Y_i^\bullet )\longrightarrow\text{Hom}_{\mathcal{D}(\mathcal{H})}(\Phi (A^\wedge ),\coprod Y_i^\bullet )$. 

 We now finish the proof of the fully faithful condition of $\Phi$.  For any $N\in\mathcal{D}(\mathcal{A})$ fixed, we consider the full subcategory $\mathcal{C}^N$ of $\mathcal{D}(\mathcal{A})$ consisting of those dg $\mathcal{A}$-modules $M$ such that $\Phi :\text{Hom}_{\mathcal{D}(\mathcal{A})}(M[k],N)\longrightarrow\text{Hom}_{\mathcal{D}(\mathcal{H})}(\Phi (M)[k],\Phi (N))$ is an isomorphism, for all $k\in\mathbb{Z}$. This is a triangulated subcategory of $\mathcal{D}(\mathcal{A})$ clearly closed under taking coproducts. The previous paragraphs of this proof show that $\mathcal{C}^N$ contains all representable $\mathcal{A}$-modules $A^\wedge$, with $A\in\mathcal{A}$. By d\'evissage, we conclude that $\mathcal{C}^N=\mathcal{D}(\mathcal{A})$, and hence that $\Phi$ is fully faithful.

Let us put $\mathcal{X}=\text{Im}(\Phi )$, which is then a full triangulated subcategory of $\mathcal{D}(\mathcal{H})$ closed under taking coproducts, when they exist. The fact that $\Phi$ is fully faithful and has a right adjoint implies that the inclusion functor $\mathcal{X}\hookrightarrow\mathcal{D}(\mathcal{H})$ has a right adjoint. By \cite[Proposition 1]{KV}, we know that $(\mathcal{X},\mathcal{X}^\perp )$ is a semi-orthogonal decomposition of $\mathcal{D}(\mathcal{H})$. We will prove that $\mathcal{X}^\perp =0$, which will imply that $\mathcal{X}=\mathcal{D}(\mathcal{H})$ and, hence, that $\Phi$ is an equivalence of categories. Then its quasi-inverse $\Psi :\mathcal{D}(\mathcal{H})\longrightarrow\mathcal{D}=\mathcal{D}(\mathcal{A})$ will be the desired functor. Indeed, if $X\in\mathcal{H}$ then, as seen in the proof of Proposition \ref{prop.abelian category in module category}, we have  $\Psi (X)=G(X)[0]\otimes_E^\mathbf{L}T\cong X$, for each $X\in\mathcal{H}$, from which it is easily seen that $\Psi_{| \mathcal{H}}:\mathcal{H}\longrightarrow\mathcal{D}$ is naturally isomorphic to the inclusion functor. 

Let $Y^\bullet$ be an object of $\mathcal{X}^\perp$. Then, for each $A\in\mathcal{A}$ and $k\in\mathbb{Z}$, we have that  
$0=\text{Hom}_{\mathcal{D}(\mathcal{H})}(\Phi (A^\wedge )[k],Y^\bullet)\cong\text{Hom}_{\mathcal{D}(\mathcal{A})}(A^\wedge [k],\Psi (Y^\bullet ))$. This implies that $0=\Psi (Y^\bullet )=[(?\otimes_E^\mathbf{L}T)\circ G](Y^\bullet )$.  By Lemma \ref{lem.image of G}, we know that $G(Y^\bullet)\cong\text{RHom}_\mathcal{A}(T,M)$, for some $M\in\mathcal{D}(\mathcal{A})$ fixed in the sequel. 
 Moreover, since the counit $\delta :(?\otimes_E^\mathbf{L}T)\circ\text{RHom}_\mathcal{A}(T,?)\longrightarrow 1_{\mathcal{D}(\mathcal{A})}$ is a natural isomorphism we get an isomorphism $$0=\Phi (Y^\bullet )=\text{RHom}_\mathcal{A}(T,M)\otimes_E^\mathbf{L}T\stackrel{\epsilon_M}{\longrightarrow}M. $$ This implies that $G(Y^\bullet )=0$. In other words, if $Y^\bullet$  is the complex in $\mathcal{C}(\mathcal{H})$  $$\cdots\longrightarrow Y^{k-1}\longrightarrow Y^k\longrightarrow Y^{k+1}\longrightarrow \cdots, $$ then the complex in $\mathcal{C}(E)$ $$G(Y^\bullet ): \hspace*{0.5cm} \cdots\longrightarrow G(Y^{k-1})\longrightarrow G(Y^k)\longrightarrow G(Y^{k+1})\longrightarrow \cdots$$ is acyclic. But due to the exactness of $G$, we have that $0=H^k(G(Y^\bullet ))=G(H^k(Y^\bullet))$, so that we have $\text{Hom}_\mathcal{H}(T,H^k(Y^\bullet ))=0$, for all $k\in\mathbb{Z}$. This in turn implies that $H^k(Y^\bullet )=0$, for all $k\in\mathbb{Z}$, because $T$ is a projective generator of $\mathcal{H}$. That is,  $Y^\bullet$ is acyclic and hence $Y^\bullet =0$ in $\mathcal{D}(\mathcal{H})$. 
\end{proof}

A few corollaries follow now (see \cite{PV} for related results when $\mathcal{D}=\mathcal{D}(\mathcal{A})$ is the derived category of an abelian category as in Setup \ref{setup}). 

\begin{cor} \label{cor.restriction of the equivalence}
Let $\mathcal{D}$ be a compactly generated algebraic triangulated category, let $\mathcal{T}$ be a bounded tilting set in $\mathcal{D}$ and let $\mathcal{S}$ be a silting set of compact objects of $\mathcal{D}$ which is weakly equivalent to $\mathcal{T}$. 
If $\mathcal{H}$ is the heart of the t-structure in $\mathcal{D}$ associated to $\mathcal{T}$ and $\Psi :\mathcal{D}(\mathcal{H})\stackrel{\sim}{\longrightarrow}\mathcal{D}$ is any triangulated equivalence which extends the inclusion functor $\mathcal{H}\hookrightarrow\mathcal{D}$, then the following assertions hold:

\begin{enumerate}
\item $\Psi (\mathcal{D}^b(\mathcal{H}))$ consists of the objects $M$ of $\mathcal{D}$ such that, for some natural number $n=n(M)$, one has $\text{Hom}_\mathcal{D}(S,M[k])=0$ whenever  $|k|>n$ and $S\in\mathcal{S}$.
\item $\Psi (\mathcal{D}^-(\mathcal{H}))$ (resp. $\Psi (\mathcal{D}^+(\mathcal{H}))$) consists of the objects $M$ of $\mathcal{D}$ such that, for some natural number $n=n(M)$, one has $\text{Hom}_\mathcal{D}(S,M[k])=0$ whenever  $k>n$ (resp. $k<-n$) and $S\in\mathcal{S}$.
\end{enumerate}  
\end{cor}
\begin{proof}
Without loss of generality, we assume that $\mathcal{T}=\{T\}$. Let $X^\bullet\in\mathcal{D}(\mathcal{H})$ be any object. Using the fact that $T$ is a projective generator of $\mathcal{H}$ and that $\Psi (T)\cong T$, we get a chain of double implications: 

\begin{center}
$X^\bullet\in\mathcal{D}^-(\mathcal{H})$ $\Longleftrightarrow$ $\text{Hom}_{\mathcal{D}(\mathcal{H})}(T,X^\bullet [k])=0$, for $k\gg 0$ $\Longleftrightarrow$ $\text{Hom}_\mathcal{D}(T,\Psi (X^\bullet )[k])=0$, for $k\gg 0$. 
\end{center}
The fact that $\text{thick}_\mathcal{D}(\text{Sum}(T))=\text{thick}_\mathcal{D}(\text{Sum}(\mathcal{S}))$ gives then that $X^\bullet\in\mathcal{D}^-(\mathcal{H})$ if and only if there exists $n=n(X^\bullet )\in\mathbb{N}$ such that $\text{Hom}_\mathcal{D}(S,\Psi (X^\bullet )[k])=0$, for all $S\in\mathcal{S}$ and all integers $k>n$. Then the part of assertion 2 concerning the description of $\Psi (\mathcal{D}^-(\mathcal{H}))$ is clear. A symmetric argument proves the corresponding assertion for $\Psi (\mathcal{D}^+(\mathcal{H}))$, which in turn proves also assertion 1. 
\end{proof}

\begin{cor} \label{cor.restriction dg algebra}
Let $A$ be a dg algebra (e.g. an ordinary algebra), let $\mathcal{T}$ be a bounded tilting set in $\mathcal{D}(A)$ and let $\mathcal{H}$ be the heart of the t-structure $(\mathcal{T}^{\perp_{>0}}, \mathcal{T}^{\perp_{<0}})$ in $\mathcal{D}(A)$. If $\Psi :\mathcal{D}(\mathcal{H})\stackrel{\sim}{\longrightarrow}\mathcal{D}(A)$ is any triangulated equivalence which extends the inclusion functor $\mathcal{H}\hookrightarrow\mathcal{D}(A)$, then it induces by restriction triangulated equivalences $\mathcal{D}^*(\mathcal{H})\stackrel{\sim}{\longrightarrow}\mathcal{D}^*(A)$, for $*\in\{b,+,-\}$.
\end{cor}
\begin{proof}
 As usual, we assume that  $\mathcal{T}=\{T\}$. For each classical silting set $\mathcal{S}$ of $\mathcal{D}(A)$, one has $\text{thick}_{\mathcal{D}(A)}(\mathcal{S})=\text{per}(A)=\text{thick}_{\mathcal{D}(A)}(A)$ (*) (see \cite[Theorem 5.3]{K}). Note that then we have a finite subset $\mathcal{S}_0\subseteq\mathcal{S}$ such that $A\in\text{thick}_{\mathcal{D}(A)}(\mathcal{S}_0)$. Then $\mathcal{S}_0$ is also a classical silting set, because it is classical partial silting and generates $\mathcal{D}(A)$.  By \cite[Theorem 2.18]{AI}, we conclude that $\mathcal{S}=\mathcal{S}_0$, so that $\mathcal{S}$ is finite. But the equality $\text{thick}_{\mathcal{D}(A)}(\text{Sum}(\mathcal{S}))=\text{thick}_{\mathcal{D}(A)}(\text{Sum}(T))$  gives that, up to shift, we have an inclusion $\mathcal{S}\subset\text{Add}(T)\star\text{Add}(T)[1]\star \cdots\star\text{Add}(T)[n]$, for some $n\in\mathbb{N}$. Due to the tilting condition of $T$, we know that $\text{Hom}_{\mathcal{D}(A)}(X,Y[k])=0$, for  $|k|>2n$, whenever $X,Y\in\text{Add}(T)\star\text{Add}(T)[1]\star \cdots\star\text{Add}(T)[n]$. Therefore we have that $\text{Hom}_{D(A)}(S,S'[k])=0$, for $|k|>n$, whenever $S,S'\in\mathcal{S}$. 

Bearing in mind the equality (*) above, we get an $p\in\mathbb{N}$ such that $H^k(A)\cong\text{Hom}_{\mathcal{D}(A)}(A,A[k])=0$, for $|k|>p$. That is, $A$ is homologically bounded. Moreover, the finiteness of $\mathcal{S}$ implies that we actually have $\text{thick}_{\mathcal{D}(A)}(\text{Sum}(\mathcal{S}))=\text{thick}_{\mathcal{D}(A)}(\text{Sum}(A))$.  It then follows in a straightforward way that a dg $A$-module $Y$ is in $\mathcal{D}^-(A)$ if and only if there is a $m=m(Y)\in\mathbb{N}$ such that $\text{Hom}_{\mathcal{D}(A)}(S,Y[k])=0$ for $k>m$. A similar argument works when we take $+$ or $b$ instead of $-$. Now apply Corollary \ref{cor.restriction of the equivalence}.
 \end{proof}

Theorem \ref{teor.tilting derived equivalence} allows us to say something when we replace `tilting' by `partial tilting'. We refer to \cite[Section 1.4]{BBD} for the definition of recollement of triangulated categories. 

\begin{cor} \label{cor.recollement}
Let $\mathcal{D}$ be a compactly generated algebraic triangulated category, let $\mathcal{T}$ be a bounded partial tilting set of $\mathcal{D}$, let $\tau =({}^\perp (\mathcal{T}^{\perp_{\leq 0}}),\mathcal{T}^{\perp_{<0}})$ be its associated t-structure and let $\mathcal{H}$ be its heart. The inclusion $\mathcal{H}\hookrightarrow\mathcal{D}$ extends to a fully faithful functor $j_!:\mathcal{D}(\mathcal{H})\longrightarrow\mathcal{D}$ which fits in a recollement

$$\begin{xymatrix}{\mathcal{D'} \ar[r]^{i_*}& \mathcal{D} \ar@<3ex>[l]_{i^!}\ar@<-3ex>[l]_{i^*}\ar[r]^{j^*} & \mathcal{D(H)} \ar@<3ex>_{j_*}[l]\ar@<-3ex>_{j_!}[l]}
\end{xymatrix}$$

\end{cor}
\begin{proof}
Let $\mathcal{S}$ be a classical partial silting set in $\mathcal{D}$ which is weakly equivalent to $\mathcal{T}$. Consider the associated localizing subcategory $\mathcal{D}'':=\text{Loc}_\mathcal{D}(\mathcal{T})=\text{Loc}_{\mathcal{D}}(\mathcal{S})$. Since it is compactly generated it is smashing and we have a recollement  

$$\begin{xymatrix}{\mathcal{D'} \ar[r]& \mathcal{D} \ar@<3ex>[l]\ar@<-3ex>[l]\ar[r] & \mathcal{D''} \ar@<3ex>[l]\ar@<-3ex>[l]}
\end{xymatrix}$$

 (see \cite[Theorem 5]{NS}), where the upper arrow $\mathcal{D}\longleftarrow\mathcal{D}''$ is the inclusion functor. We claim that $\mathcal{D}''$   is also an algebraic triangulated category. To see this, abusing notation, we also denote by $(i^*,i_*,i^!,j_!,j^*,j_*)$ the arrows in the last recollement. 
 By properties of recollements, we know that $\mathcal{D}''$ is triangle-equivalent to the Verdier quotient $\mathcal{D}/\text{Im}(i_*)$ and we also know that $\text{Im}(i_*)=\text{Loc}_\mathcal{D}(i_*(\mathcal{G}))$, for a set $\mathcal{G}$ of compact generators of $\mathcal{D}'$ (e.g. one can take $\mathcal{G}$ to be a set of representatives of the isoclasses of the objects in $i^*(\mathcal{D}^c)$).  The algebraicity of $\mathcal{D}''$ follows then from \cite[Theorem 7.2]{Porta}.

   Note that $\mathcal{T}$ is a bounded tilting set of $\mathcal{D}''$.  By Theorem \ref{teor.presilting t-structures} and its proof, we have that $\mathcal{T}^{\perp_{>0}}=\mathcal{U}_\mathcal{T}\star\mathcal{T}^{\perp_{k\in\mathbb{Z}}}$, which implies that $\mathcal{T}^{\perp_{>0}\text{ }(\mathcal{D}'')}:=\mathcal{T}^{\perp_{>0}}\cap\mathcal{D}''=\mathcal{U}_\mathcal{T}={}^\perp (\mathcal{T}^{\leq 0})$. Then the heart of the t-structure in $\mathcal{D}''$ associated to $\mathcal{T}$ is $$\mathcal{H}'=\mathcal{T}^{\perp_{>0}\text{ }(\mathcal{D}'')}\cap\mathcal{T}^{\perp_{<0}\text{ }(\mathcal{D}'')}=\mathcal{U}_\mathcal{T}\cap\mathcal{T}^{\perp_{<0}}\cap\mathcal{D}''=\mathcal{U}_\mathcal{T}\cap\mathcal{T}^{\perp_{<0}}=\mathcal{H}$$ since $\mathcal{U}_\mathcal{T}\subset\mathcal{D}''$. 
Theorem \ref{teor.tilting derived equivalence} gives then a triangulated equivalence $\Psi:\mathcal{D}(\mathcal{H})\stackrel{\sim}{\longrightarrow}\mathcal{D}''$. The desired functor $j_!$ is the composition $\mathcal{D}(\mathcal{H})\stackrel{\Psi}{\longrightarrow}\mathcal{D}''\stackrel{incl}{\hookrightarrow}\mathcal{D}$.
\end{proof}

We end by giving a partial affirmative answer to Question \ref{ques.self-small silting} (see the end of Section \ref{sect.silting bijection}).

\begin{cor} \label{cor.self-small versus compact tilting}
Let $\mathcal{D}$ be a compactly generated algebraic triangulated category and let $T$ be an object of $\mathcal{D}$. Then $T$ is a self-small bounded tilting object if and only if it is a classical tilting object. 
\end{cor}
\begin{proof}
Just adapt the proof of \cite[Corollary 2.5]{FMS}, using our Theorem \ref{teor.tilting derived equivalence} instead of \cite[Proposition 2.3]{FMS}. 
\end{proof}

\section{Relation with exceptional sequences} \label{sec.exceptional sequences}

Starting with Rudakov's seminar (see \cite{Ru}) the concepts of exceptional and strongly exceptional sequence of coherent sheaves have played a fundamental role in Algebraic Geometry and it is still an open problem to identify those algebraic varieties $\mathbb{X}$ for which there exists an exceptional sequence in the category $\text{coh}(\mathbb{X})$ of coherent sheaves (see the introduction of \cite{HP}). 
In this final section we want to stress the relationship between these classical concepts and those of partial silting (or tilting) sets studied in earlier sections.  
In a typical pattern of this paper, we shall pass from  finite (strongly) exceptional sequences to infinite ones and from `coherent' objects to arbitrary ones. 

\begin{opr}
Let $\mathcal{G}$ be any Grothendieck category and let $(T_n)_{n\in\mathbb{Z}}$ be a sequence of (possibly zero) objects of $\mathcal{G}$. We will say that $(T_n)_{n\in\mathbb{Z}}$ is

\begin{enumerate}
\item \emph{Exceptional} when $\text{Ext}_\mathcal{G}^k(T_n,T_p^{(I)})=0$, for  all sets $I$ and all triples $(n,p,k)\in\mathbb{Z}\times\mathbb{Z}\times\mathbb{N}$ such that either $n>p$ or $n=p$ and $k>0$. 
\item \emph{Strongly exceptional} when $\text{Ext}_\mathcal{G}^k(T_n,T_p^{(I)})=0$, for   all sets $I$ and all triples  $(n,p,k)\in\mathbb{Z}\times\mathbb{Z}\times\mathbb{N}$ such that either $k>0$ or $n>p$ and $k=0$.
\item \emph{Superexceptional}  when $\text{Ext}_\mathcal{G}^k(T_n,T_p^{(I)})=0$, for all sets $I$ and all triples  $(n,p,k)\in\mathbb{Z}\times\mathbb{Z}\times\mathbb{N}$ such that $k>p-n$.
\end{enumerate}
All these sequences will be called \emph{complete} when they generate $\mathcal{D}(\mathcal{G})$. 
\end{opr}

Note that `strongly exceptional' implies `superexceptional', which in turn implies `exceptional'.  Note also that, in the three definitions, the condition for $k=0$ says that $\text{Hom}_\mathcal{G}(T_n,T_p^{(I)})=0$ for all sets $I$  and integers  $n>p$. This is equivalent to say that  $\text{Hom}_\mathcal{G}(T_n,T_p)=0$, for $n>p$, since the canonical morphism $T_p^{(I)}\longrightarrow T_p^I$ is a monomorphism in a Grothendieck category. 

 Recall that an abelian category $\mathcal{A}$  is \emph{hereditary} when $\text{gldim}(\mathcal{A})\leq 1$ (see Definition \ref{def.finite projective dimension}). 
 The category of quasi-coherent sheaves on a weighted projective line (see \cite{GL1}) or the module category over the path algebra of a possibly infinite quiver are examples of hereditary Grothendieck categories. 
Note that if $\mathcal{G}$ is a hereditary Grothendieck category, then  a sequence $(T_n)_{n\in\mathbb{Z}}$ in $\mathcal{G}$ is  superexceptional if and only if it is exceptional.  

Recall that a Grothendieck category $\mathcal{G}$ is \emph{locally coherent} when the objects $X$ such that $\text{Hom}_\mathcal{G}(X,?):\mathcal{G}\longrightarrow\text{Ab}$ preserves direct limits, usually called \emph{finitely presented objects}, form a skeletally small class $fp(\mathcal{G})$ of generators  of $\mathcal{G}$ which is closed under taking kernels. 

We are ready for the main result of the section. 

\begin{prop} \label{prop.exceptional sequences}
Let $\mathcal{G}$ be any  Grothendieck category with $\text{gldim}(\mathcal{G})\leq d<\infty$, and let $(T_n)_{n\in\mathbb{Z}}$ be a sequence of objects in $\mathcal{G}$. The following assertions hold:

\begin{enumerate}
\item $(T_n)_{n\in\mathbb{Z}}$ is a superexceptional sequence if and only if $\mathcal{T}=\{T_n[n]\text{: }n\in\mathbb{Z}\}$ is a strongly nonpositive set in $\mathcal{D}(\mathcal{G})$.
\item $(T_n)_{n\in\mathbb{Z}}$ is exceptional if and only if $\mathcal{T}=\{T_n[nd]\text{: }n\in\mathbb{Z}\}$ is a strongly nonpositive set in $\mathcal{D}(\mathcal{G})$.
\end{enumerate}

Moreover, when $\mathcal{G}$ is hereditary, each strongly nonpositive set in $\mathcal{D}(\mathcal{G})$ is equivalent to a set $\mathcal{T}$ as above (see Definition \ref{def.equivalent partial silting}). 

  For general $d>0$, under conditions 1 or 2 above,  if either one of the  conditions below holds then $\mathcal{T}$ is partial silting, and it is even  partial tilting  in case $(T_n)_{n\in\mathbb{Z}}$ is strongly exceptional. 

\begin{enumerate}
\item[a)] The family $(T_n)_{n\in\mathbb{N}}$ is finite, i.e. $T_n=0$ for almost all $n\in\mathbb{Z}$;
\item[b)] $\mathcal{G}$ is locally coherent and all the $T_n$ are finitely presented objects of $\mathcal{G}$.  
\end{enumerate}

\end{prop} 
\begin{proof}
For the first part of the proposition, we just need to prove assertion 1 since, as it is easily verified and left to the reader,  the sequence $(T_n)_{n\in\mathbb{Z}}$ is exceptional if and only if the sequence $(T'_n)_{n\in\mathbb{Z}}$ is superexceptional, where $T'_n=0$ for $n\not\in d\mathbb{Z}$ and $T'_n=T_{\frac{n}{d}}$ for $n\in d\mathbb{Z}$. 

We then prove assertion 1. The set $\mathcal{T}$ is strongly nonpositive if and only if $\text{Hom}_{\mathcal{D}(\mathcal{G})}(T_n[n],\coprod_{p\in\mathbb{Z}}T_p^{(I_p)}[p+k])=0$ (*), for all $n,p\in\mathbb{Z}$, for all integers $k>0$ and for each family $(I_p)_{p\in\mathbb{Z}}$ of sets. Let us fix $n\in\mathbb{Z}$ and $k>0$ and decompose $X:=\coprod_{p\in\mathbb{Z}}T_p^{(I_p)}[p+k]$ as $X=Y\coprod Z\coprod W$, where $Y=\coprod_{p\geq n-k+d+1}T_p^{(I_p)}[p+k]$, $Z=T_{n-k}^{(I_{n-k})}[n]\coprod T_{n-k+1}^{(I_{n-k+1})}[n+1]\coprod\cdots\coprod T_{n-k+d}^{(I_{n-k+d})}[n+d]$ and $W=\coprod_{p\leq n-k-1}T_p^{(I_p)}[p+k]$. Then we have that $Y\in\mathcal{D}^{\leq -n-d-1}(\mathcal{G})$ while $W\in\mathcal{D}^{> -n}(\mathcal{G})$. Then $\text{Hom}_{\mathcal{D}(\mathcal{G})}(T_n[n],W)=0$ and, by Lemma \ref{lem.finite projective dimension} and the fact that each object of $\mathcal{G}$ has projective dimension $\leq d$, we also have that $\text{Hom}_{\mathcal{D}(\mathcal{G})}(T_n[n],Y)=0$. Therefore equality (*) holds for $n$ and $k$ if and only if $\text{Hom}_{\mathcal{D}(\mathcal{G})}(T_n[n],Z)=0$, which is in turn equivalent to saying that $\text{Ext}_\mathcal{G}^j(T_n,T_{n-k+j}^{(I_{n-k+j})})=0$, for all $j=0,1,\dots,d$. Bearing in mind that $\text{Ext}_\mathcal{G}^j(?,?):\mathcal{G}^{op}\times\mathcal{G}\longrightarrow Ab$ vanishes for $j\in\mathbb{N}\setminus\{0,1,\dots,d\}$, we deduce that the equality (*) holds if and only if $\text{Ext}_\mathcal{G}^j(T_n,T_p^{(I)})=0$, for all sets I, whenever $j$ is a natural number such that $j>p-n$. That is, if and only if $(T_n)_{n\in\mathbb{Z}}$ is a superexceptional sequence in $\mathcal{G}$. 

Suppose now that $\mathcal{G}$ is hereditary and that $\mathcal{T}'$ is a strongly nonpositive set of objects in $\mathcal{D}(\mathcal{G})$. It is well-known that in this case each $X\in\mathcal{D}(\mathcal{G})$ is isomorphic to $\coprod_{n\in\mathbb{Z}}H^{-n}(X)[n]$. This implies that $\tilde{\mathcal{T}}'=\{H^{-n}(T')[n]\text{: }n\in\mathbb{Z}\text{, }T'\in\mathcal{T}'\}$ is also strongly nonpositive. Taking now $T_n=\coprod_{T'\in\mathcal{T}'}H^{-n}(T')$, we get a set $\{T_n[n]\text{: }n\in\mathbb{Z}\}$ which is also strongly nonpositive in $\mathcal{D}(\mathcal{G})$ and satisfies that $\text{Add}(\mathcal{T}')=\text{Add}(\mathcal{T})$. By the previous paragraph, we know that $(T_n)_{n\in\mathbb{Z}}$ is a  (super)exceptional sequence in $\mathcal{G}$.

We pass to prove the partial silting condition. When $(T_n)_{n\in\mathbb{Z}}$ is finite, due to the fact that $\text{gldim}(\mathcal{G})<\infty$,  we know that there are integers $r<s$ such that $\mathcal{T}\subset\mathcal{D}^{\leq s}(\mathcal{G})$ and $\text{Hom}_{\mathcal{D}(\mathcal{G})}(T_{n}[n],?)$ vanishes on $\mathcal{D}^{\leq r}(\mathcal{G})=\mathcal{D}^{\leq s}(\mathcal{G})[s-r]$, for all $n\in\mathbb{Z}$. Then Theorem \ref{teor.weakly equivalent to presilting} applies, with $\mathcal{V}=\mathcal{D}^{\leq s}(\mathcal{G})$ and $q=s-r$. 

Suppose now that $\mathcal{G}$ is locally coherent. We  claim if $X$ is a finitely presented object of $\mathcal{G}$, then  $X=X[0]$ is a compact object of $\mathcal{D}(\mathcal{G})$. Once this is proved the partial silting condition of $\mathcal{T}$ under assumption 2 will follow directly from Example \ref{ex.examples partial silting}(1). We need to prove that if $(M_i)_{i\in I}$ is any family of objects of $\mathcal{D}(\mathcal{G})$, then the canonical map $\coprod_{i\in I}\text{Hom}_{\mathcal{D}(\mathcal{G})}(X,M_i)\longrightarrow\text{Hom}_{\mathcal{D}(\mathcal{G})}(X,\coprod_{i\in I}M_i)$ is an isomorphism. By taking successively the canonical truncation triangles $\tau^{\leq -d-1}M_i\longrightarrow M_i\longrightarrow\tau^{>-d-1}M_i\stackrel{+}{\longrightarrow}$ and $\tau^{\leq 0}(\tau^{>-d-1}M_i)\longrightarrow\tau^{>-d-1}M_i\longrightarrow\tau^{>0}M_i\stackrel{+}{\longrightarrow}$ and bearing in mind that $\text{Hom}_{\mathcal{D}(\mathcal{G})}(X,?)$ vanishes on $\mathcal{D}^{\leq -d-1}(\mathcal{G})$ (see Lemma \ref{lem.finite projective dimension}) and on $\mathcal{D}^{>0}(\mathcal{G})$, we readily see that we can assume without loss of generality that $M_i=\tau^{\leq 0}(\tau^{>-d-1}M_i)$, for each $i\in I$. But then each $M_i$ has homology concentrated in degrees $-d,-d+1,\dots,-1,0$, so that each $M_i$ is a finite iterated extension in $\mathcal{D}(\mathcal{G})$ of the stalk complexes $H^{-d}(M_i)[d]$, $H^{-d+1}(M_i)[d-1]$, \dots, $H^0(M_i)[0]$. It follows that we can assume without loss of generality that, for some $k=0,1,\dots,d$, all $M_i$ are stalk complexes at $k$. The proof is whence reduced to prove that if $(Y_i)_{i\in I}$ is a family of objects of $\mathcal{G}$ and $k\in\{0,1,\dots,d\}$, then the canonical map $$\coprod_{i\in I}\text{Ext}_\mathcal{G}^k(X,Y_i)=\coprod_{i\in I}\text{Hom}_{\mathcal{D}(\mathcal{G})}(X,Y_i[k])\longrightarrow\text{Hom}_{\mathcal{D}(\mathcal{G})}(X,\coprod_{i\in I}Y_i[k])=\text{Ext}_\mathcal{G}^k(X,\coprod_{i\in I}Y_i)$$ is an isomorphism. But this is obvious since $\coprod_{i\in I}Y_i$ is the direct limit of its finite subcoproducts and, due to the locally coherent condition of $\mathcal{G}$,  all the functors $\text{Ext}_\mathcal{G}^k(X,?)$ ($k\geq 0$) preserve direct limits (see \cite[Proposition 3.5]{Sa}).

The fact that, under hypotheses a) or b), the set $\mathcal{T}$ is partial tilting whenever $(T_n)_{n\in\mathbb{Z}}$ is strongly exceptional is obvious. 
\end{proof}

\begin{rem} \label{rem.equivalence}
When $\mathbb{X}$ is a smooth algebraic variety, a well-known result of Baer and Bondal (see \cite{Ba} and \cite{Bondal}) says that if $\mathcal{E}=\{E_1,\dots,E_n\}$ is a  strongly exceptional sequence in $\text{coh}(\mathbb{X})$ which is complete (i.e.   $\text{thick}(\mathcal{E})=\mathcal{D}^b(\mathbb{X}):=\mathcal{D}^b(\text{coh}(\mathbb{X}))$), then there is a triangulated equivalence $\mathcal{D}^b(\mathbb{X})\cong\mathcal{D}^b(\text{mod}-R)$ which extends to the unbounded level, where $R=\text{End}_{\text{coh}(\mathbb{X})}(\oplus_{1\leq i\leq n}E_i)$. In view of Proposition \ref{prop.exceptional sequences} and 
Theorem \ref{teor.tilting derived equivalence}, we see that the analogous result, with $\text{Mod}-R$ replaced by the heart of the  t-structure generated by $\mathcal{E}$, is also true when   $\mathcal{E}$ is a strongly exceptional sequence of not necessarily  coherent sheaves in $\text{Qcoh}(\mathbb{X})$.  
\end{rem}

We now give an example of an infinite (super)exceptional sequence of `coherent' objects and a finite one consisting of `noncoherent' objects. 
\begin{ex}
\begin{enumerate}
\item (See also \cite[Example 4.18]{PV}) Consider a field $K$ and an acyclic quiver $Q$ admitting an infinite path $i_0\longrightarrow i_1\longrightarrow \cdots i_n\longrightarrow \cdots$. Putting $T_n=0$ for $n<0$, and $T_n=e_{i_n}KQ$ for $n\geq 0$, we get a strongly  exceptional sequence in the hereditary category $\text{Mod}-KQ$ of unitary (right) $KQ$-modules.
 
\item Let $\mathbb{X}=\mathbf{P}^n(K)$ be the projective $n$-space over the algebraically closed field $K$. Let $0\longrightarrow O_\mathbb{X}\longrightarrow E^0\longrightarrow E^1\longrightarrow \cdots\longrightarrow E^n\longrightarrow 0$ be the minimal injective resolution in $\text{Qcoh}(\mathbb{X})$ of the structural sheaf. If $?(i):\text{Qcoh}(\mathbb{X})\longrightarrow\text{Qcoh}(\mathbb{X})$ denotes the canonical shift equivalence, for each $i\in\mathbb{Z}$, and we put $\hat{E}^k=\oplus_{i=0}^nE^k(i)$, for each $k=0,1,\dots,n$, then $\mathcal{E}:=\{\hat{E}_0,\hat{E}_1,\dots,\hat{E}_n\}$ is a complete strongly exceptional sequence in $\text{Qcoh}(\mathbb{X})$. In particular, if $(\mathcal{E}^{\perp_{>0}},\mathcal{E}^{\perp_{<0}})$ is the associated t-structure in $\mathcal{D}(\mathbb{X}):=\mathcal{D}(\text{Qcoh}(\mathbb{X}))$, then we have a triangulated equivalence $\mathcal{D}(\mathcal{H})\cong\mathcal{D}(\mathbb{X})$, where $\mathcal{H}=\mathcal{E}^{\perp_{>0}}\cap\mathcal{E}^{\perp_{<0}}$ is the heart of this t-structure (see Remark \ref{rem.equivalence}).
\end{enumerate}
\end{ex}
\begin{proof}

1) Recall that $\text{Hom}_{KQ}(e_{i_n}KQ,e_{i_p}KQ)\cong e_{i_p}KQe_{i_n}$. If this $K$-vector space is nonzero, then there is a path $i_n\rightarrow ...\rightarrow i_p$ (possibly of length 0) in $Q$  (here we are viewing a path $j_0\stackrel{\alpha_1}{\rightarrow}j_1\stackrel{\alpha_2}{\rightarrow}...\stackrel{\alpha_n}{\rightarrow}j_n$ in $Q$ as the composition of arrows $\alpha_n\alpha_{n-1}...\alpha_1$ and, hence, it belongs to $e_{j_n}KQe_{j_0}$).  If now   $(n,p,k)\in\mathbb{Z}\times\mathbb{Z}\times\mathbb{N}$ is any triple, then  $\text{Ext}_{KQ}^k(T_n,T_p^{(I)})=0$ whenever $k>0$ due to the projective condition of the $T_n$ in $\text{Mod}-KQ$. On the other hand, if $k=0$ and $n>p$ we have that $\text{Ext}_{KQ}^0(T_n,T_p^{(I)})=\text{Hom}_{KQ}(T_n,T_p^{(I)})$. This is clearly the zero vector space since, due to the acyclicity of $Q$, there are no paths in $Q$ from $i_n$ to $i_p$ whenever $p\geq 0$ (the only case that need to be considered). Therefore $(T_n)_{n\in\mathbb{Z}}$ is a strongly exceptional sequence in $\text{Mod}-KQ$. 

\vspace*{0.3cm}

2) We claim that  $\text{Hom}_{\text{Qcoh}(\mathbb{X})}(E^r(i),E^s(j))=0$ (and hence   $\text{Hom}_{\text{Qcoh}(\mathbb{X})}(E^r(i),E^s(j)^{(I)})=0$ for all sets $I$) whenever $r>s$ and $i,j\in\{0,1,\dots,n\}$. To see this, we use  Serre's theorem and identify $\text{Qcoh}(\mathbb{X})$ with the `category of tails' $\text{Gr}-R/\text{Tor}-R$, where $R=K[x_0,x_1,\dots,x_n]$, $\text{Gr}-R$ is the category of graded $R$-modules and $\text{Tor}-R$ consists of the graded $R$-modules whose elements are annihilated by powers of the ideal $R^+=(x_0,\dots,x_1)$. Adapting to the graded situation the argument for the ungraded case (see \cite[Theorem 18.8]{M}) we get that if $0\longrightarrow R\longrightarrow I^0\longrightarrow I^1\longrightarrow \cdots\longrightarrow I^n\longrightarrow I^{n+1}\longrightarrow 0$ is the minimal injective resolution of $R$ in $\text{Gr}-R$, then $I^k$ is the direct sum of the $E_{gr}(R/\mathbf{p})$, where $E_{gr}(?)$ denotes the injective envelope in $\text{Gr}-R$ and $\mathbf{p}$ ranges over all graded prime ideals of $R$ with graded height equal to $k$ (with the obvious definition). This means that $I^k$ is torsionfree for $k=0,1,\dots,n$ while $I^{n+1}$ is in $\text{Tor}-R$. Then if $q:\text{Gr}-R\longrightarrow \text{Gr}-R/\text{Tor}-R\cong\text{Qcoh}(\mathbb{X})$ is the quotient functor, we get that the minimal injective resolution of $q(R)\cong O_\mathbb{X}$ in $\text{Gr}-R/\text{Tor}-R\cong\text{Qcoh}(\mathbb{X})$ is $0\longrightarrow q(I^0)\longrightarrow q(I^1)\longrightarrow \cdots\longrightarrow q(I^n)\longrightarrow 0$. Recall that $q$ has a fully faithful right adjoint $j:\text{Gr}-R/\text{Tor}-R\rightarrowtail \text{Gr}-R$ whose essential image consists of those $Y\in \text{Gr}-R$ such that $\text{Ext}_{\text{Gr}-R}^i(T,Y)=0$, for $i=0,1$ and all $T\in \text{Tor}-R$.  Then the unit map $I^k\longrightarrow (j\circ q)(I^k)$ is an isomorphism, for all $k=0,1,\dots,n$, and we have 

\begin{center}
$\text{Hom}_{\text{Qcoh}(\mathbb{X})}(E^r(i),E^s(j))=\text{Hom}_{\text{Gr}-R/\text{Tor}-R}(q(I^r)(i),q(I^s)(j))\cong\text{Hom}_{\text{Gr}-R}((j\circ q)(I^r)(i),(j\circ q)(I^r)(j))\cong\text{Hom}_{\text{Gr}-R}(I^r(i),I^s(j)).$
\end{center}
But if $f:I^r(i)\longrightarrow I^s(j)$ were a nonzero morphism in $\text{Gr}-R$, then $\text{Im}(f)$ would be a nonzero graded submodule of $I^s(j)$ whose graded support consists of graded prime ideals of graded height $\geq r$ and whose associated graded prime ideals are of graded height $s$. This is a contradiction, and our claim is settled. This also implies that $\mathcal{E}=\{\hat{E}^0,\hat{E}^1,\dots,\hat{E}^n\}$ is a strongly exceptional sequence due to the injective condition of all coproducts of the $\hat{E}^k$ in $\text{Qcoh}(\mathbb{X})$. On the other hand, we know that $\mathcal{T}:=\{O_\mathbb{X},O_\mathbb{X}(1),\dots,O_\mathbb{X}(n)\}$ is a classical tilting set of $\text{Qcoh}(\mathbb{X})$ (see \cite[Lemma 2]{Bei}), so that it generates $\mathcal{D}(\mathbb{X})$ as a triangulated category.  Since we clearly have that $\text{thick}_{\mathcal{D}(\mathbb{X})}(\mathcal{T})\subseteq\text{thick}_{\mathcal{D}(\mathbb{X})}(\mathcal{E})$ we conclude that $\mathcal{E}$ generates $\mathcal{D}(\mathbb{X})$ and, hence, $\mathcal{E}$ is  complete. 

\end{proof}

\end{document}